\documentclass[10pt,a4paper,reqno,sumlimits]{amsart}
\usepackage[numbers]{natbib}
\usepackage{amsmath}
\usepackage{amssymb}
\usepackage[mathscr]{euscript}
\usepackage{enumerate}
\usepackage{xspace}
\usepackage{color}
\usepackage{enumerate}
\usepackage{enumitem}
\usepackage{amsthm}

\begin{document}

\theoremstyle{plain}
\newtheorem{C}{Convention}
\newtheorem*{SA}{Standing Assumption}
\newtheorem{theorem}{Theorem}[section]
\newtheorem{condition}{Conditions}[section]
\newtheorem{lemma}[theorem]{Lemma}
\newtheorem{proposition}[theorem]{Proposition}
\newtheorem{corollary}[theorem]{Corollary}
\newtheorem{claim}[theorem]{Claim}
\newtheorem{definition}[theorem]{Definition}
\newtheorem{Ass}[theorem]{Assumption}
\newcommand{\q}{Q}
\theoremstyle{definition}
\newtheorem{remark}[theorem]{Remark}
\newtheorem{note}[theorem]{Note}
\newtheorem{example}[theorem]{Example}
\newtheorem{assumption}[theorem]{Assumption}
\newtheorem*{notation}{Notation}
\newtheorem*{assuL}{Assumption ($\mathbb{L}$)}
\newtheorem*{assuAC}{Assumption ($\mathbb{AC}$)}
\newtheorem*{assuEM}{Assumption ($\mathbb{EM}$)}
\newtheorem*{assuES}{Assumption ($\mathbb{ES}$)}
\newtheorem*{assuM}{Assumption ($\mathbb{M}$)}
\newtheorem*{assuMM}{Assumption ($\mathbb{M}'$)}
\newtheorem*{assuL1}{Assumption ($\mathbb{L}1$)}
\newtheorem*{assuL2}{Assumption ($\mathbb{L}2$)}
\newtheorem*{assuL3}{Assumption ($\mathbb{L}3$)}
\newtheorem{charact}[theorem]{Characterization}
\newcommand{\notiz}{\textup} 
\renewenvironment{proof}{{\parindent 0pt \it{ Proof:}}}{\mbox{}\hfill\mbox{$\Box\hspace{-0.5mm}$}\vskip 16pt}
\newenvironment{proofthm}[1]{{\parindent 0pt \it Proof of Theorem #1:}}{\mbox{}\hfill\mbox{$\Box\hspace{-0.5mm}$}\vskip 16pt}
\newenvironment{prooflemma}[1]{{\parindent 0pt \it Proof of Lemma #1:}}{\mbox{}\hfill\mbox{$\Box\hspace{-0.5mm}$}\vskip 16pt}
\newenvironment{proofcor}[1]{{\parindent 0pt \it Proof of Corollary #1:}}{\mbox{}\hfill\mbox{$\Box\hspace{-0.5mm}$}\vskip 16pt}
\newenvironment{proofprop}[1]{{\parindent 0pt \it Proof of Proposition #1:}}{\mbox{}\hfill\mbox{$\Box\hspace{-0.5mm}$}\vskip 16pt}
\newcommand{\s}{\u s}

\newcommand{\Law}{\ensuremath{\mathop{\mathrm{Law}}}}
\newcommand{\loc}{{\mathrm{loc}}}
\newcommand{\Log}{\ensuremath{\mathop{\mathscr{L}\mathrm{og}}}}
\newcommand{\Meixner}{\ensuremath{\mathop{\mathrm{Meixner}}}}
\newcommand{\of}{[\hspace{-0.06cm}[}
\newcommand{\gs}{]\hspace{-0.06cm}]}

\let\MID\mid
\renewcommand{\mid}{|}

\let\SETMINUS\setminus
\renewcommand{\setminus}{\backslash}

\def\stackrelboth#1#2#3{\mathrel{\mathop{#2}\limits^{#1}_{#3}}}

\renewcommand{\theequation}{\thesection.\arabic{equation}}
\numberwithin{equation}{section}

\newcommand\llambda{{\mathchoice
      {\lambda\mkern-4.5mu{\raisebox{.4ex}{\scriptsize$\backslash$}}}
      {\lambda\mkern-4.83mu{\raisebox{.4ex}{\scriptsize$\backslash$}}}
      {\lambda\mkern-4.5mu{\raisebox{.2ex}{\footnotesize$\scriptscriptstyle\backslash$}}}
      {\lambda\mkern-5.0mu{\raisebox{.2ex}{\tiny$\scriptscriptstyle\backslash$}}}}}

\newcommand{\prozess}[1][L]{{\ensuremath{#1=(#1_t)_{0\le t\le T}}}\xspace}
\newcommand{\prazess}[1][L]{{\ensuremath{#1=(#1_t)_{0\le t\le T^*}}}\xspace}

\newcommand{\tr}{\operatorname{tr}}
\newcommand{\lijepoa}{{\mathscr{A}}}
\newcommand{\lijepob}{{\mathscr{B}}}
\newcommand{\lijepoc}{{\mathscr{C}}}
\newcommand{\lijepod}{{\mathscr{D}}}
\newcommand{\lijepoe}{{\mathscr{E}}}
\newcommand{\lijepof}{{\mathscr{F}}}
\newcommand{\lijepog}{{\mathscr{G}}}
\newcommand{\lijepok}{{\mathscr{K}}}
\newcommand{\lijepoo}{{\mathscr{O}}}
\newcommand{\lijepop}{{\mathscr{P}}}
\newcommand{\lijepoh}{{\mathscr{H}}}
\newcommand{\lijepom}{{\mathscr{M}}}
\newcommand{\lijepou}{{\mathscr{U}}}
\newcommand{\lijepov}{{\mathscr{V}}}
\newcommand{\lijepoy}{{\mathscr{Y}}}
\newcommand{\cF}{{\mathscr{F}}}
\newcommand{\cG}{{\mathscr{G}}}
\newcommand{\cH}{{\mathscr{H}}}
\newcommand{\cM}{{\mathscr{M}}}
\newcommand{\cD}{{\mathscr{D}}}
\newcommand{\bD}{{\mathbb{D}}}
\newcommand{\bF}{{\mathbb{F}}}
\newcommand{\bG}{{\mathbb{G}}}
\newcommand{\bH}{{\mathbb{H}}}
\newcommand{\dd}{\operatorname{d}\hspace{-0.05cm}}
\newcommand{\ddd}{\operatorname{d}}
\newcommand{\er}{{\mathbb{R}}}
\newcommand{\ce}{{\mathbb{C}}}
\newcommand{\erd}{{\mathbb{R}^{d}}}
\newcommand{\en}{{\mathbb{N}}}
\newcommand{\de}{{\mathrm{d}}}
\newcommand{\im}{{\mathrm{i}}}
\newcommand{\indik}{{\mathbf{1}}}
\newcommand{\D}{{\mathbb{D}}}
\newcommand{\E}{E}
\newcommand{\N}{{\mathbb{N}}}
\newcommand{\Q}{{\mathbb{Q}}}
\renewcommand{\P}{{\mathbb{P}}}
\newcommand{\ud}{\operatorname{d}\!}
\newcommand{\ii}{\operatorname{i}\kern -0.8pt}
\newcommand{\cadlag}{c\`adl\`ag }
\newcommand{\p}{P}
\newcommand{\F}{\mathbf{F}}
\newcommand{\1}{\mathbf{1}}
\newcommand{\f}{\mathscr{F}^{\hspace{0.03cm}0}}
\newcommand{\lle}{\langle\hspace{-0.085cm}\langle}
\newcommand{\rre}{\rangle\hspace{-0.085cm}\rangle}
\newcommand{\llbr}{[\hspace{-0.085cm}[}
\newcommand{\rrbr}{]\hspace{-0.085cm}]}

\def\EM{\ensuremath{(\mathbb{EM})}\xspace}

\newcommand{\la}{\langle}
\newcommand{\ra}{\rangle}

\newcommand{\Norml}[1]{%
{|}\kern-.25ex{|}\kern-.25ex{|}#1{|}\kern-.25ex{|}\kern-.25ex{|}}

\title[Cylindrical Martingale Problems]{Cylindrical Martingale Problems Associated with L\'evy Generators} 
\author[D. Criens]{David Criens}
\address{D. Criens - Technical University of Munich, Center for Mathematics, Germany}
\email{david.criens@tum.de}

\keywords{Cylindrical martingale problem, L\'evy generator, Markov property, Cameron-Martin-Girsanov formula, stochastic partial differential equation\vspace{1ex}}


\thanks{D. Criens - Technical University of Munich, Center for Mathematics, Germany,  \texttt{david.criens@tum.de}.}

\date{\today}
\maketitle

\frenchspacing
\pagestyle{myheadings}

\begin{abstract}
We introduce and discuss L\'evy-type cylindrical martingale problems on separable reflexive Banach spaces.
Our main observations are the following:
Cylindrical martingale problems have a one-to-one relation to weak solutions of stochastic partial differential equations, 
well-posed problems possess the strong Markov property and a Cameron-Martin-Girsanov-type formula holds.
As applications, we derive existence and uniqueness results.
\end{abstract}

\section{Introduction}
We study a generalized martingale problem (GMP) associated to a L\'evy-type generator on a separable and reflexive Banach space, which allows for the possibility of explosion to an absorbing state. The generator has a linear unbounded part and non-linearities, which satisfy mild boundedness conditions.

Beside pure mathematical interest, GMPs are interesting due to their one-to-one relation to semilinear stochastic partial differential equations (SPDEs) driven by L\'evy noise. 
Typically, SPDEs are studied under Lipschitz-type assumptions on the coefficients and many arguments are tailor-made for this case.
In contrast, by using the GMP formulation it is possible to prove properties of SPDEs under minimal regularity assumptions on the coefficients. 
Let us give an overview on our results, which are collected in Section \ref{sec: MR}.

In Section \ref{sec: MP}, we show that well-posed GMPs have the strong Markov property.
This observation suggests future analysis using tools from Markov process theory. 

In Section~\ref{sec: CMG}, we give a Cameron-Martin-Girsanov-type (CMG) theorem that relates two solutions of GMPs with different drift and jump coefficients in the spirit of the classical Girsanov theorem for semimartingales.
We provide an explicit formula for the Radon-Nikodym density. If one of the solutions is conservative, our CMG theorem gives a formula for the distribution of the explosion time under the other solution. For an application of such a formula in a one-dimensional diffusion case we refer to \cite{Karatzas2016}.
Girsanov-type theorems in general have far-reaching consequences. For instance, we deduce that two well-posed conservative GMPs which can be related via a CMG formula are locally equivalent. Furthermore, we connect their existence and uniqueness properties.

In Section \ref{sec: relation SDE} we formalize the relation of GMPs and SPDEs
and  in Section \ref{sec: appl} we apply our previous results to deduce deterministic conditions for the existence and uniqueness of weak solutions to SPDEs with additive noise.

Let us comment on the existing literature.
The conservative diffusion case was systematically discussed by Kunze \cite{EJP2924} and Ondrej{\'a}t \cite{Ondrejat2005}. Our results extend their work to non-conservative setups and allow for the presence of jumps. 

In general, for Girsanov-type theorems and the equivalence of laws of stochastic processes there exists a vast literature. 
We mention just a few closely related papers:
For a continuous setting in an arbitrary dimension, a CMG formula is studied by Mikulevi{\v{c}}ius and Rozovskii \cite{Mikulevicius1994}. In a finite-dimensional jump-diffusion setting, Cheridito, Filipovi\'c and Yor \cite{CFY} give a CMG formula up to explosion, which is similar to ours. The CMG formula presented in this article and all its consequences are new for infinite-dimensional setups. 

Martingale problems in the context of evolution equations with continuous noise are also studied by Goldys, R\"ockner and Zhang \cite{Goldys20091725}.
In the case of L\'evy-driven equations, Mytnik \cite{Mytnik2002} uses martingale problems to derive existence results for SPDEs driven by stable noise.

Finally, we comment on our proofs, which are given in Section \ref{sec: pf}.
The observation that well-posed martingale problems have the strong Markov property goes back to Stroock and Varadhan \cite{SV69-1}. 
The crucial point is to show that GMPs are determined by a countable set of test functions. This is clear in finite-dimensional conservative settings, but
requires more attention in an infinite-dimensional setup and in non-conservative cases.
To overcome this problem, we first show that explosion can only happen in a continuous manner. This allows us to relate solutions to GMPs to families of stopped finite-dimensional semimartingales. This observation will also be important in the proof of the CMG theorem, since it allows us to apply finite-dimensional Girsanov-type theorems.
In continuous settings it can be shown more directly that the set of linear and quadratic test functions determines the martingale problem.

Our strategy to prove the CMG theorem goes back to Jacod and M\'emin \cite{JM76}. 
The idea is to use a local change of measure, Girsanov-type theorems and a strong version of uniqueness, which is called local uniqueness in the monographs \cite{J79, JS}.
In our infinite-dimensional setting we cannot define the usual candidate density process for the local change of measure immediately from the coordinate process. Instead, we study the jump measure associated to our GMP and use extension results for cylindrical continuous local martingales, see \cite{Ondrejat2005}.
In comparison to finite-dimensional cases, Girsanov-type theorems are less well-established in infinite dimensions, see \cite{daprato2013} for a result for cylindrical Brownian motion. However, due to our observation that GMPs are connected to families of finite-dimensional semimartingales, we can reduce the necessary applications to well-understood situations, see \cite{J79,JS}.
Another difficulty in our settings is that local uniqueness is only studied in conservative cases and in finite-dimensional non-conservative settings, see \cite{EK, J79, JS}. 
We derive the necessary results on local uniqueness by adjusting the arguments presented in the monographs \cite{J79, JS} to our setup. 

The relation between martingale problems and SPDEs is proven via the representation theorem for cylindrical continuous local martingales as given by Ondrej{\'a}t \cite{Ondrejat2005} and the representation theorem for Poisson random measures as given by Jacod \cite{J79}. Here, we use the relation of GMPs and finite-dimensional semimartingales and our previous study of the jump measure associated to a GMP.

\section{Generalized Martingale Problems} \label{2.1}
Let \(\mathbb{B}\) be a real separable reflexive Banach space. 
In general, we will denote any norm by \(\|\cdot\|\).
If not stated otherwise, we call a function with values in a Banach space measurable if it is weakly measurable.
If the function is separably valued, Pettis' measurability theorem implies that it is also strongly measurable. Furthermore, if the Banach space is separable, then weak, strong and usual measurability are equivalent, see \cite{hytönen2016analysis}.
Let \(\mathbb{B}^*\) be the (topological) dual of \(\mathbb{B}\) and equip it with the operator norm.
For \(x^*\in \mathbb{B}^*\) and \(x \in \mathbb{B}\) we write
\[
x^*(x) \triangleq \langle x, x^*\rangle.
\]
It is well-known that \(\mathbb{B}^*\) is also a real separable reflexive Banach space.
We note that
\begin{align*}d_\mathbb{B}(x, y) \triangleq \frac{\|x - y\|}{1 + \|x - y\|},\quad x, y\in \mathbb{B},\end{align*} 
is a metric on \(\mathbb{B}\) such that \(\mathbb{B}\) is Polish when equipped with \(d_\mathbb{B}\), see \cite[Lemma 3.6]{aliprantis2013infinite}.
We equip \(\mathbb{B}\) with the topology induced by \(d_\mathbb{B}\), which equals the usual norm topology.

Let \(\Delta\) be a point outside \(\mathbb{B}\) and denote 
\(\mathbb{B}_\Delta \triangleq \mathbb{B} \cup \{\Delta\}\). 
As pointed out in \cite[Remark 2.8]{van2011markov}, when equipped with the metric \(d \colon \mathbb{B}_\Delta\times \mathbb{B}_\Delta \to [0,1]\) defined by
\[
d (x, y) \triangleq d_\mathbb{B}(x, y) \1_\mathbb{B} (x) \1_\mathbb{B} (y) + |\1_\Delta (x) - \1_\Delta (y)|,\quad x, y \in \mathbb{B}_\Delta,
\]
the space \(\mathbb{B}_\Delta\) is Polish. We equip \(\mathbb{B}_\Delta\) with the topology induced by \(d\).

We set \(\|\Delta\| \triangleq \infty\).
If not mentioned otherwise, any measurable function \(f \colon \mathbb{B} \times \bullet \to \mathbb{R}\) is extended to \(\mathbb{B}_\Delta \times \bullet\) by setting \(f(\Delta, \cdot) \triangleq 0\).

We define \(\Omega\) to be the space of all \cadlag functions \(\alpha \colon [0, \infty)\to \mathbb{B}_\Delta\) such that \(\alpha(t-) = \Delta\) or \(\alpha(t) = \Delta\) implies \(\alpha(s) = \Delta\) for all \(s \geq t\). 
The coordinate process \(X\) on \(\Omega\) is defined by \(X_t(\alpha) = \alpha(t)\) for all \(\alpha \in \Omega\). Moreover, we denote \[\mathscr{F} \triangleq \sigma(X_s, s \in [0, \infty)), \quad\mathscr{F}_t \triangleq \sigma(X_s, s \in [0, t]), \quad \F\triangleq (\mathscr{F}_t)_{t \in [0, \infty)}.\]
For a \cadlag function \(\omega \colon [0, \infty) \to \mathbb{B}_\Delta\), we define the explosion time \begin{align*}
\tau_{\Delta} (\omega) \triangleq \inf(t \in [0, \infty) \colon \omega(t-) = \Delta \textup{ or } \omega(t) = \Delta),\end{align*}
 and for \(n \in \mathbb{N}\) the first approach times
\begin{equation}\begin{split}
\label{eq:tau n}
\tau^*_n (\omega) &\triangleq  \inf(t \in [0, \infty) \colon \|\omega(t-)\| \geq n \textup{ or } \|\omega(t)\| \geq n),\\
\tau_n (\omega) &\triangleq \tau^*_n (\omega) \wedge n.
\end{split}
\end{equation}
For \(\omega \in \Omega\), we have 
\begin{align*}
\tau_{\Delta} (\omega) = \inf(t \in [0, \infty) \colon \omega(t) = \Delta).
\end{align*}
By \cite[Proposition 2.1.5]{EK}, \(\tau_{\Delta}, \tau_n^*\) and \(\tau_n\) are \(\F\)-stopping times. Furthermore, we note that \(\tau_n (\omega) \nearrow \tau_{\Delta}(\omega)\) as \(n \to \infty\) for all \(\omega \in \Omega\).

For an \(\F\)-stopping time \(\xi\) we denote
\[
\mathscr{F}_{\xi} \triangleq \left\{ A \in \mathscr{F} \colon A \cap \{\xi \leq t\} \in \mathscr{F}_t \textup{ for all } t \in [0, \infty)\right\}.
\]
It is well-known that \(\mathscr{F}_\xi\) is a \(\sigma\)-field. We define \(\mathscr{F}_{\xi-}\) to be the \(\sigma\)-field generated by \(\mathscr{F}_0\) and the events of the form
\(
A \cap \{t < \xi\}\) for \(t \in [0, \infty)\) and \(A \in \mathscr{F}_t.
\)
By \cite[Proposition 25]{10.2307/2373011}, we have
\begin{align*}
\mathscr{F}_{\tau_\Delta-} = \mathscr{F}.
\end{align*}

An operator \(Q \colon \mathbb{B}^* \to \mathbb{B}\) is called positive, if \(\la Q x^*, x^*\ra \geq 0\) for all \(x^* \in \mathbb{B}^*\), and called symmetric, if \(\la Q x^*, y^*\ra = \la Q y^*, x^*\ra\) for all \(x^*, y^* \in \mathbb{B}^*\).         
We denote by \(S^+(\mathbb{B}^*, \mathbb{B})\) the set of all linear, bounded, positive and symmetric operators \(\mathbb{B}^* \to \mathbb{B}\).     
For a topological space \(E\) we denote the corresponding Borel \(\sigma\)-field by \(\mathscr{B}(E)\).

Next, we introduce the parameters for our generalized martingale problem:
\begin{enumerate}
	\item[(i)]
	Let \(A \colon D(A) \subseteq \mathbb{B} \to \mathbb{B}\) be a linear, densely defined and closed operator.
	Here, \(D(A)\) denotes the domain of the operator \(A\).
	\item[(ii)]
	Let \(b \colon \mathbb{B} \to\mathbb{B}\) be Borel and such that for all bounded sequences \((y^*_n)_{n \in \mathbb{N}} \subset \mathbb{B}^*\) and all bounded sets \(G \in \mathscr{B}(\mathbb{B})\) it holds that 
	\begin{align*}
	\sup_{n \in \mathbb{N}} \sup_{x \in G} |\la b(x), y^*_n\ra| < \infty.
	\end{align*}
	\item[(iii)] 
	Let \(a\colon \mathbb{B}  \to S^+(\mathbb{B}^*, \mathbb{B})\) be bounded on bounded subsets of \(\mathbb{B}\) and such that \(x \mapsto a(x) y^*\) is Borel for all \(y^* \in \mathbb{B}^*\). Here, bounded refers to the operator norm.
	\item[(iv)]
	Let \(K\) be a Borel transition kernel from \(\mathbb{B}\) into \(\mathbb{B}\), such that for all bounded sequences \((y^*_n)_{n \in \mathbb{N}} \subset \mathbb{B}^*\), all bounded sets \(G \in \mathscr{B}(\mathbb{B})\) and all \(\epsilon > 0\) it holds that
	\begin{align}\label{eq: assp K}
	\sup_{n \in \mathbb{N}} &\sup_{x \in G} \int \1_{\{\|y\| \leq \epsilon\}} |\la y, y^*_n\ra|^2K(x, \dd y) < \infty,\\
	\label{eq: int ball around origin}
	&\ \sup_{x \in G} K(x, \{z \in \mathbb{B} \colon \|z\| \geq \epsilon\}) < \infty,
	\end{align}
	and \(K(\cdot, \{0\}) = 0\).
	\item[(v)]
	Let \(\eta\) be a probability measure on \((\mathbb{B}_\Delta, \mathscr{B}(\mathbb{B}_\Delta))\). 
\end{enumerate}
The probability measure \(\eta\) serves as the initial law. In typical applications the (usually unbounded) operator \(A\) is an infinitesimal generator of a \(C_0\)-semigroup on~\(\mathbb{B}\).

Let us justify the conditions \eqref{eq: assp K} and \eqref{eq: int ball around origin}.
It is well-known that a measure \(F\) on \((\mathbb{R}^d, \mathscr{B}(\mathbb{R}^d))\) is a L\'evy measure if and only if \(F(\{0\}) = 0\) and \(\int 1 \wedge \|x\|^2 F(\dd x) < \infty.\)
Since intuitively \(K(\cdot, \dd x)\) plays the role of a L\'evy measure, it would be natural to impose assumptions on the map 
\(
x \mapsto  \int \left(1\wedge \|y\|^2\right) K(x, \dd y).
\)
However, there exist reflexive Banach spaces such that not all L\'evy measures integrate \(1 \wedge \|y\|^2\), see \cite{araujo1978, PeterG2003}.
But, as stated in \cite[Proposition 5.4.5]{linde1983probability}, every L\'evy measure \(F\) satisfies
\begin{align*}
\sup_{\|y^*\| \leq 1} \int \1_{\{\|x\| \leq 1\}} |\la x, y^*\ra|^2 F(\dd x) + F(\{z \in \mathbb{B} \colon \|z\| \geq \epsilon\}) < \infty
\end{align*} for all \(\epsilon > 0\).

Let \(A^*\) be the Banach adjoint of \(A\).
Since \(\mathbb{B}\) is assumed to be reflexive, \(A^*\) is densely defined and closed, see \cite[Theorem 5.29]{kato1980}. 

Let \(C^2_{c}(\mathbb{R}^d)\) be the set of twice continuously differentiable functions \(\mathbb{R}^d\to \mathbb{R}\) with compact support. 
The set of test functions for our generalized martingale problem consists of cylindrical functions:
\[\mathcal{C} \triangleq \left\{g(\langle \cdot, y^*_1\rangle, ..., \langle \cdot, y^*_n\rangle) \colon g \in C^2_c(\mathbb{R}^n), y^*_1, ..., y^*_n \in D(A^*), n \in \mathbb{N}\right\}.\]
To clarify our convention, we stress that for \(f = g(\la \cdot, y^*_1\ra, ..., \la \cdot, y^*_n\ra)\), where \(g \colon \mathbb{R}^n \to \mathbb{R}\) is Borel and \(y^*_1, ..., y^*_n \in \mathbb{B}^*\), we set \(f(\Delta) \triangleq 0\). 
In the same case, if \(g\) is twice continuously differentiable, we write \(\partial_{i} f\) for the partial derivative \((\partial_{i} g)(\la \cdot, y^*_1\ra, ..., \la \cdot, y^*_n\ra)\) and define \(\partial^2_{ij} f\) in the same manner.

For a normed space \(E\), we call a bounded Borel function \(h\colon E \to E\) a truncation function on \(E\), if there exists an \(\epsilon > 0\) such that \(h(x) = x\) on the set \(\{x \in E\colon \|x\| \leq \epsilon\}\). Throughout the article we fix a truncation function \(h\) on \(\mathbb{B}\).

For \(f = g(\langle \cdot, y^*_1\rangle, ..., \langle \cdot, y^*_n\rangle) \in \mathcal{C}\) we set
\begin{equation}\label{K}
\begin{split}
\mathcal{K}f (x) \triangleq  \sum_{i = 1}^n (& \langle x, A^* y^*_i\rangle + \langle b(x), y^*_i \rangle)\partial_i f(x) +  \frac{1}{2} \sum_{i=1}^n\sum_{j = 1}^n \langle a(x)y^*_i, y^*_j\rangle \partial^2_{ij}f(x) 
\\&+ \int \left( f(x + y) - f(x) - \sum_{i = 1}^n \langle h(y), y^*_i\rangle \partial_i f(x) \right) K(x, \dd y)
\end{split}
\end{equation}
if \(x \in \mathbb{B}\) and \(\mathcal{K}f(\Delta) \triangleq 0\).

We first define a \emph{stopped martingale problem}. Let \(\xi\) be an \(\F\)-stopping time.
\begin{definition} 
	We call a probability measure \(P\) on \((\Omega, \mathscr{F})\) a \emph{solution to the martingale problem (MP) \((A, b,a,K, \eta, \xi)\)}, if the following hold:
	\begin{enumerate}
		\item[\textup{(i)}]
		\(P \circ X^{-1}_0 = \eta\). 
		\item[\textup{(ii)}]
		For all \(f \in \mathcal{C}\) the process
		\begin{equation}\label{f - K}
		\begin{split}
		M^f_{\cdot \wedge \xi} \triangleq f(X_{\cdot \wedge \xi}) - f(X_0) - \int_0^{\cdot\wedge \xi} \mathcal{K}f (X_{s-})\dd s,
		\end{split}
		\end{equation}
		is a local \((\F, P)\)-martingale. 
	\end{enumerate}
	We say that the MP has a \emph{unique} solution, if all solutions coincide on \(\mathscr{F}_{\xi}\). 
	We denote the set of solutions by \(\mathcal{M}(A,b,a, K, \eta, \xi)\).
\end{definition}

Next, we also introduce an MP up to explosion. We will call it the \emph{generalized} martingale problem. Following the notation of Jacod and Shiryaev \cite{JS}, we denote the Dirac measure by \(\varepsilon\).
\begin{definition}
	\begin{enumerate}
		\item[\textup{(i)}]
		We call a probability measure \(P\) on \((\Omega, \mathscr{F})\) a \emph{solution to the GMP \((A,b, a, K, \eta, \tau_\Delta-)\)}, if \(P\) solves the GMP \((A, a, b,K, \eta, \tau_n)\) for all \(n \in \mathbb{N}\). The set of solutions is denoted by \(\mathcal{M}(A, b,a, K, \eta, \tau_\Delta-)\).
		\item[\textup{(ii)}]	
		We say that the GMP \((A, b, a, K, \eta, \tau_{\Delta}-)\) satisfies \emph{uniqueness}, if all solutions coincide on \(\mathscr{F}_{\tau_\Delta-} = \mathscr{F}\).
		If there exists a unique solution for all \(\eta \in \{\varepsilon_{x}, x \in \mathbb{B}_\Delta\}\), we call the GMP
		\((A, b, a, K, \tau_{\Delta}-)\) \emph{well-posed}.	
		If there exists a unique solution for all initial laws \(\eta\), then we call the GMP \emph{completely well-posed}.
		If \(P\) solves a GMP and \(P(\tau_{\Delta} = \infty) = 1\) we call \(P\) \emph{conservative}.
	\end{enumerate}
\end{definition} 
For a Polish space \(E\), we define \(D([0, \infty), E)\) to be the space of all \cadlag functions \([0, \infty) \to E\).
The space \(D([0, \infty), E)\) equipped with the Skorokhod topology is itself Polish. Moreover, it holds that \(\mathscr{B}(D([0, \infty), E)) = \sigma(X_t, t \in [0, \infty))\), where \(X\) denotes the coordinate process on \(D([0, \infty), E)\). For details we refer to \cite[Theorem 3.5.6, Proposition 3.7.1]{EK}.
We introduce the following short notation: \begin{align*}
(\mathbb{D},\mathscr{D}) \triangleq (D([0, \infty), \mathbb{B}), \mathscr{B}(D([0, \infty), \mathbb{B}))).
\end{align*}
\begin{remark}
	Solutions to MPs with \(\xi \triangleq \infty\) or conservative solutions can be viewed as probability measures on \((\mathbb{D},\mathscr{D})\).
	In these cases, we remove \(\xi\) and \(\tau_\Delta-\) from all notations.
\end{remark}
\begin{remark}\label{rem: trivial}
	The Dirac measure on the constant path \(\omega^\Delta_t \equiv \Delta\)
	is the unique solution of the GMP \((A, b, a, K, \varepsilon_\Delta, \xi)\) for all \(\F\)-stopping times \(\xi\). In particular, this probability measure is the unique solution of the GMP \((A, b, a, K, \varepsilon_\Delta, \tau_\Delta-)\).
\end{remark}

\section{Main Results}\label{sec: MR}
In this section we state our main results.
The section is split into four parts. 
In Section \ref{sec: MP} we discuss the strong Markov property of well-posed problems.
Then, in Section \ref{sec: CMG}, we relate two solutions to GMPs via a Cameron-Martin-Girsanov formula. 
A one-to-one correspondence of conservative solutions to GMPs and (analytically and probabilistically) weak solutions to SDEs is presented in Section \ref{sec: relation SDE}. Finally, in Section \ref{sec: appl}, we give existence and uniqueness results as applications.
\subsection{The Markov Property}\label{sec: MP} 
A Hausdorff space that is the image of a Polish space under a continuous bijection is called Lusin space. Any Polish space is a Lusin space and, more generally, any Borel subset of a Polish space, seen as a subspace, is a Lusin space, see \cite[Example 8.6.11]{cohn13}. The Borel \(\sigma\)-field of a Lusin space is countably generated, see \cite[Lemma 8.6.12]{cohn13}.

Note that \begin{align*}
\Omega = \left\{\omega \in D([0, \infty), \mathbb{B}_\Delta) \colon\ \  \!\begin{aligned}  &\tau_\Delta (\omega) = \infty, \textup{ or } \tau_{\Delta}(\omega) < \infty \text{ and}\\& \omega(\tau_{\Delta}(\omega) + s)= \Delta \text{ for all } s \in \mathbb{Q}_+\end{aligned}\right\},\end{align*}
where we denote the non-negative rational numbers by \(\mathbb{Q}_+\). Thus, the space
\(\Omega\), seen as a subspace of \(D([0, \infty), \mathbb{B}_\Delta)\), is a Lusin space. Furthermore, by \cite[Lemma 7.2.2]{cohn13}, \[\mathscr{B}(\Omega) = \mathscr{B}(D([0, \infty), \mathbb{B}_\Delta)) \cap\Omega = \mathscr{F}.\] Here, \(\mathscr{B}(D([0, \infty), \mathbb{B}_\Delta)) \cap\Omega\) denotes the trace of \(\mathscr{B}(D([0, \infty), \mathbb{B}_\Delta))\) on \(\Omega\).
Due to this observation, the following lemma is classical, see \cite[Theorem II.89.1]{RW1}.
\begin{lemma}\label{lem: cond prob}
	Let \(P\) be a probability measure on \((\Omega, \mathscr{F})\) and suppose that \(\mathscr{G}\) is a sub-\(\sigma\)-field of \(\mathscr{F}\).
	Then there exists a regular conditional probability \(P(\cdot |\mathscr{G}) (\cdot)\) of \(P\) given \(\mathscr{G}\), that is, a function \(P(\cdot |\mathscr{G})(\cdot) \colon \mathscr{F} \times \Omega \to [0, 1]\) such that the following holds:
	\begin{enumerate}
		\item [\textup{(i)}] For all \(\omega \in \Omega\) the map \(G \mapsto P (G|\mathscr{G})(\omega)\) is a probability measure on \((\Omega, \mathscr{F})\).
		\item [\textup{(ii)}] For all \(G \in \mathscr{F}\) the function \(\omega \mapsto P(G |\mathscr{G})(\omega)\) is a \(P\)-version of \(P(G | \mathscr{G})\).
	\end{enumerate}
	Additionally, if \(\mathscr{G}\) is countably generated, then there exists a \(P\)-null set \(N \in \mathscr{G}\) such that for all \(\omega \in \complement N\) it holds that \(P(G|\mathscr{G})(\omega) = \1_G(\omega)\) for all \(G \in \mathscr{G}\).
\end{lemma}
We denote by \(\theta_\cdot (\cdot) \colon [0, \infty) \times \Omega \to \Omega\) the shift operator \(\theta_t \omega = \omega(\cdot + t)\). If \(E\) is a topological space and \(\{P_x, x \in E\}\) is a family of probability measure on \((\Omega, \mathscr{F})\), we say that \(E \ni x \mapsto P_x\) is Borel, if the map \(E \ni x \mapsto P_x(G)\) is Borel for all \(G \in \mathscr{F}\).

In the finite-dimensional diffusion setting of Stroock and Varadhan, the strong Markov property is related to the well-posedness of the martingale problem. As the following theorem shows, this is also the case in our setting:
\begin{theorem}\label{theo:markov}
	Suppose that the GMP \((A, b, a, K, \tau_{\Delta}-)\) is well-posed and let \(P_{x}\) be the unique solution of the GMP \((A, b, a, K, \varepsilon_{x}, \tau_{\Delta}-)\) for \(x \in \mathbb{B}_\Delta\). The family \(\{P_x, x \in \mathbb{B}_\Delta\}\) is a strong Markov family in the following sense: the mapping \(\mathbb{B}_\Delta \ni x \mapsto P_x\) is Borel and for all \(x \in \mathbb{B}_\Delta\) and all \(\F\)-stopping times \(\xi\) there exists a \(P_{x}\)-null set \(N \in \mathscr{F}_\xi\) such that for all \(\omega \in \complement N \cap \{\xi < \infty\}\) and all \(F \in \mathscr{F}\)
	\[
	P_{x} \left(\theta^{-1}_\xi F\big|\mathscr{F}_\xi\right)(\omega) = P_{X_{\xi(\omega)}(\omega)} (F).
	\]
	Moreover, for all probability measures \(\eta\) on \((\mathbb{B}_\Delta, \mathscr{B}(\mathbb{B}_\Delta))\) the probability measure 
	\begin{align}\label{eq: int sol}
	P_\eta \triangleq \int P_x \eta (\dd x)
	\end{align}
	on \((\Omega, \mathscr{F})\) is the unique solution to the GMP \((A, b, a, K, \eta, \tau_{\Delta}-)\), and for all \(\F\)-stopping times~\(\xi\) there exists a \(P_\eta\)-null set \(N \in \mathscr{F}_\xi\) such that for all \(\omega \in \complement N \cap \{\xi < \infty\}\) and all \(F \in \mathscr{F}\)
	\begin{align}\label{eq: markov initial arbi}
	P_{\eta} \left(\theta^{-1}_\xi F\big|\mathscr{F}_\xi\right)(\omega) = P_{X_{\xi(\omega)} (\omega)} (F). 
	\end{align}
\end{theorem} 
The theorem can be viewed as a generalization of \cite[Theorem 10.1.1]{SV}, \cite[Theorem 4.1]{Stroock75} and \cite[Theorem 4.2]{EJP2924} to a setting of arbitrary dimension which includes jumps and allows for explosion to an absorbing state.
A detailed proof is given in Section \ref{sec: proof MP} below.

\subsection{Cameron-Martin-Girsanov Pairs}\label{sec: CMG}
In this section we give a Cameron-Martin-Girsanov-type theorem for GMPs.
We set \[\mathscr{F}_{t+} \triangleq \bigcap_{s > t} \mathscr{F}_s,\qquad \F^+ \triangleq (\mathscr{F}_{t+})_{t \in [0, \infty)}.\] 
\begin{definition}
	Let \(P\) and \(Q\) be two probability measures defined on \((\Omega, \mathscr{F})\). We say that \((Q, P)\) is a \emph{Cameron-Martin-Girsanov pair}, if there exists a non-negative local \((\F^+, P)\)-martingale \(Z^*\) starting at \(1\) such that for all \(\F\)-stopping times \(\rho\) and all \(G \in \mathscr{F}_\rho\)
	\begin{align}\label{eq: CMG formula}
	Q(G \cap \{\tau_\Delta > \rho\}) = \E^P\left[Z^*_\rho \1_{G \cap \{\tau_\Delta > \rho\}}\right].
	\end{align}
	We call \(Z^*\) the \emph{CMG density} of \(P\) and \(Q\) and \eqref{eq: CMG formula} the \emph{CMG formula}.
\end{definition}
\begin{remark}
	CMG formulas are for instance known for solution measures to one-dimensional diffusion-type SDEs, see \cite{Karatzas2016}, and in finite-dimensional jump-diffusion settings, see \cite{CFY}.
	In the case where  \(P\) and \(Q\) are solutions to GMPs and \(P\) is conservative, the CMG formula gives a representation of the \(Q\)-distribution of the explosion time as a \(P\)-expectation.
\end{remark}
Let \(c\colon \mathbb{B} \to \mathbb{B}^*\) be Borel such that \(\la ac, c\ra\) is bounded on bounded subsets of \(\mathbb{B}\). Let us stress that \(ac = a(\cdot)(c(\cdot))\) is Borel, see \cite[Proposition 1.1.28]{hytönen2016analysis}. Therefore, also \(\la ac, c\ra\) is Borel.
Moreover, let \(Y \colon \mathbb{B} \times \mathbb{B}\to (0, \infty)\) be Borel such that for all \(x \in \mathbb{B}\) and all \(y^* \in \mathbb{B}^*\)
\[
\int |\la h(y), y^*\ra| | Y(x, y) - 1| K(x, \dd y) + \int \left(1 - \sqrt{Y(x, y)}\right)^2 K(x, \dd y) < \infty.
\]
Since \(\mathbb{B}\) is assumed to be reflexive, these integrability conditions suffice to define the integral 
\(
\int h(y) (Y(\cdot, y) - 1) K(\cdot, \dd y)
\)
as a Pettis integral, see \cite[Proposition 3.4]{Kazimierz2002531}.
We further suppose that the map
\begin{align*}
x \mapsto  \int \left(1 - \sqrt{Y(x, y)}\right)^2 K(x, \dd y)
\end{align*}
is Borel and bounded on bounded subsets of \(\mathbb{B}\), and that the map \begin{align*}x \mapsto  \int h(y) (Y(x, y) - 1)K(x, \dd y)\end{align*} is Borel and for
all bounded sequences \((y^*_n)_{n \in \mathbb{N}} \subset \mathbb{B}^*\), all bounded sets \(G \in \mathscr{B}(\mathbb{B})\) and all \(\epsilon > 0\) it holds that 
\begin{align*}
\sup_{n \in \mathbb{N}} &\sup_{x \in G} \int \1_{\{\|y\| \leq \epsilon\}} |\la y, y^*_n\ra|^2 Y(x, y) K(x, \dd y) < \infty,\\ &\ \ \sup_{x \in G} \int \1_{\{\|y\| \geq \epsilon\}} Y(x, y) K(x, \dd y) < \infty.
\end{align*}
We denote 
\begin{equation}\label{eq: b', Y'}
\begin{split}
b' (x) & \triangleq b(x) + a(x)c(x) + \int h(y) (Y(x, y) - 1) K(x, \dd y),
\\ K' (x, \dd y) &\triangleq Y(x, y) K(x, \dd y).
\end{split}
\end{equation}
Furthermore, we assume that for all bounded sequences \((y^*_n)_{n \in \mathbb{N}} \subset \mathbb{B}^*\) and all bounded sets \(G \in \mathscr{B}(\mathbb{B})\) it holds that 
\begin{align}\label{eq: bdd cond b'}
\sup_{n \in \mathbb{N}} \sup_{x \in G} |\la b'(x), y^*_n\ra| < \infty.
\end{align}

The following theorem can be viewed as a generalization of \cite[Theorem 6.4.2]{SV} and \cite[Theorem 2.4]{CFY} to a setting of arbitrary dimension. We note that in \cite{CFY} the martingale problem also includes a killing rate, which is not included in our setting.
\begin{theorem}\label{theo:main1}
	Suppose that the GMP \((A, b ', a, K', \tau_\Delta-)\) is well-posed and that \(\eta\) is a probability measure on \((\mathbb{B}, \mathscr{B}(\mathbb{B}))\). Then there exists a unique solution \(Q_\eta\) to the GMP \((A, b', a, K', \eta, \tau_{\Delta}-)\). Suppose that \(P_\eta\in \mathcal{M}(A, b, a, K, \eta, \tau_\Delta-)\), then \((Q_\eta, P_\eta)\) is a CMG pair.
	Moreover, the CMG density is \(P\)-indistinguishable from the process \(Z^*\) given in Lemma \ref{lem: candidate density} below.
\end{theorem}
A detailed proof is given in Section \ref{sec: pf CMG} below.
\begin{remark}
	In the setting of Theorem \ref{theo:main1}, \(P_\eta\) and \(Q_\eta\) are equivalent on \(\mathscr{F}_{\tau_n}\) for all \(n \in \mathbb{N}\). This follows either from the theorem itself together with the Lemmata \ref{lem:tau} and \ref{lem: candidate density} below, or from Proposition \ref{prop:locuni} and the Lemmata \ref{lem: candidate density} and \ref{lem: Q n sol} below.
\end{remark}
An in-depth study of the CMG density allows us to obtain the following uniqueness result. A detailed proof is given in Section \ref{sec: pf coro uni} below.
\begin{proposition}\label{theo:main2}
	Let \(\eta\) be an arbitrary probability measure on \((\mathbb{B}, \mathscr{B}(\mathbb{B}))\).
	If the GMP \((A, b, a, K, \tau_{\Delta}-)\) is well-posed, then the GMP \((A,b',a, K', \eta, \tau_\Delta -)\) satisfies uniqueness.
	Conversely, if the GMP \((A, b', a, K', \tau_{\Delta}-)\) is well-posed, then the GMP \((A, b, a, K, \eta, \tau_{\Delta}-)\) satisfies uniqueness.
\end{proposition}	
We also give an existence result.
Suppose that the GMP \((A, b, a, K, \tau_\Delta-)\) is well-posed, let \(x \in \mathbb{B}\) and take \(P_x \in \mathcal{M}(A, b, a, K, \varepsilon_x, \tau_\Delta-)\). We will see in the Lemmata \ref{lem: candidate density} and \ref{lem: Q n sol} below that there exists a local \(P_x\)-martingale \(Z^*\) starting at 1 such that \(Z^*_{\cdot \wedge \tau_n}\) is a non-negative uniformly integrable \(P_x\)-martingale for all \(n \in \mathbb{N}\) and such that the probability measure \(Q_x^n\) defined by \(Q_x^n(G) \triangleq E^{P_x}[ \1_GZ_{\tau_n}^*]\) solves the MP \((A, b', a, K', \varepsilon_x, \tau_n)\).
In a conservative setting we can now give a precise condition for the existence of a solution to the GMP \((A, b', a, K', \varepsilon_x, \tau_\Delta-)\).

\begin{proposition}\label{coro:main2}
	Suppose that the GMP \((A, b, a, K, \tau_\Delta-)\) is well-posed. For all \(x \in \mathbb{B}\), denote by \(P_x\) the solution to the GMP \((A, b, a, K, \varepsilon_x, \tau_{\Delta}-)\) and assume that \(P_x(\tau_\Delta = \infty) = 1\). The following are equivalent:
	\begin{enumerate}
		\item[\textup{(i)}]
		For all probability measures \(\eta\) on \((\mathbb{B}, \mathscr{B}(\mathbb{B}))\) the GMP \((A, b', a, K', \eta, \tau_\Delta-)\) has a conservative solution.
		\item[\textup{(ii)}]
		It holds that \(\lim_{n \to \infty} Q_x^n (\tau_n \leq t) = 0\) for all \(t \in [0, \infty)\) and \(x \in \mathbb{B}\).
	\end{enumerate}
	In particular, if one of the above conditions holds, then the GMP \((A, b', a, K', \tau_\Delta-)\) is completely well-posed.
\end{proposition}
The previous proposition can be viewed as a generalization of \cite[Corollary 10.1.2]{SV} to a setting of arbitrary dimension which includes jumps. The proof
relies on an extension theorem and is given in Section \ref{sec: pf coro cons existence} below.

In a non-conservative setting one can try to conclude existence from an extension argument in a larger path space. However, in this case one has to prove that the extension is supported on \((\Omega, \mathscr{F})\), see \cite{pinsky1995positive} for details in a finite-dimensional diffusion setting.
We do not address this question here and leave it for future research.

Nevertheless, let us give some conditions. 
The next proposition can be viewed as a generalization of \cite[Theorem 6.4.2]{SV}.
\begin{proposition}\label{coro: well-posed equivalence}
	Suppose that the mappings \begin{align}\label{eq: bdd assp}
	x \mapsto \la a(x)c(x), c(x)\ra, \qquad x \mapsto \int \left(1 - \sqrt{Y(x, y)}\right)^2 K(x, \dd y)
	\end{align}
	are bounded (on \(\mathbb{B}\)). The following are equivalent:
	\begin{enumerate}
		\item [\textup{(i)}]
		The MP \((A, b, a, K)\) is well-posed.
		\item [\textup{(ii)}]
		The MP \((A, b', a, K')\) is well-posed.
	\end{enumerate}
\end{proposition}

Finally, we study local equivalence and local absolute continuity of two GMPs.
Let \(P\) and \(Q\) be probability measures on \((\Omega, \mathscr{F})\). We say that \(Q\) is locally absolutely continuous w.r.t. \(P\) if \(Q\) is absolutely continuous w.r.t. \(P\) on \(\mathscr{F}_t\) for all \(t \in [0, \infty)\).
We say that \(P\) and \(Q\) are locally equivalent if \(Q\) is equivalent to \(P\) on \(\mathscr{F}_t\) for all \(t \in [0, \infty)\). A proof for the following observation can be found in Section \ref{sec: pf coro bdd coefficients} below.
\begin{proposition}\label{coro:main1}
	Assume that the GMP \((A, b, a, K, \tau_{\Delta}-)\) or \((A, b', a, K', \tau_{\Delta}-)\) is well-posed and that \(\eta\) is a probability measure on \((\mathbb{B}, \mathscr{B}(\mathbb{B}))\). Let
	\(P\) be a solution to the GMP \((A, b,a, K, \eta, \tau_\Delta-)\) and \(Q\) be a solution to the GMP \((A, b', a, K', \eta, \tau_\Delta-)\). 
	\begin{enumerate}
		\item[\textup{(i)}] Suppose that either \(P\) or \(Q\) is conservative. Then \(P\) and \(Q\) are locally equivalent if and only if both \(P\) and \(Q\) are conservative.
		\item[\textup{(ii)}] If \(Q\) is conservative, then \(Q\) is locally absolutely continuous w.r.t. \(P\).
	\end{enumerate}
\end{proposition}
If both \(P\) and \(Q\) are not conservative, the question when \(P\) and \(Q\) are locally equivalent is open (in general). For results in this direction for the case of one-dimensional It\^o-diffusions we refer to Mijatovi\'c and Urusov \cite{MU(2012)} and for a discussion of a finite-dimensional It\^o-diffusion case we refer to Ruf \cite{RufSDE}.

\subsection{MPs and SDEs}\label{sec: relation SDE}
In this section we discuss a one-to-one correspondence between solutions to MPs and analytically and probabilistically weak solutions to SDEs.
We start by formally introducing SDEs. The random drivers will be a cylindrical Brownian motion and a (homogeneous) Poisson random measure.

For completeness, let us recall the definitions.
Suppose that \(\mathbb{H}\) is a real Banach space and denote its (topological) dual by \(\mathbb{H}^*\). Below \(\mathbb{H}\) will be a Hilbert space, which explains our choice of notation. However, to define the CMG density, we also need the definition of a cylindrical continuous local martingale on a Banach space:
We call a family \(M \triangleq \{M(y^*), y^* \in \mathbb{H}^*\}\) a cylindrical continuous local martingale if for all \(y \in \mathbb{H}^*\) the process \(M(y^*)\) is a real-valued continuous local martingale and \(y^* \mapsto M(y^*)\) is linear. 

A cylindrical Brownian motion \(W\) with covariance \(U \in S^+(\mathbb{H}^*, \mathbb{H})\) is a cylindrical continuous local martingale such that 
\(
\lle W(y^*)\rre = \int_0^\cdot \la U y^*, y^*\ra \dd s\) for all \(y^* \in \mathbb{H}^*,
\)
where \(\lle \cdot \rre\) denotes the predictable quadratic variation process.

Let \(E\) be a Lusin space and \(\mathscr{E} \triangleq \mathscr{B}(E)\).
A (homogeneous) Poisson random measure \(\mu\)  on a filtered probability space \((\Omega^o, \mathscr{F}^o, (\mathscr{F}^o_t)_{t \in [0, \infty)}, P^o)\) is an integer-valued random measure on \([0, \infty) \times E\) such that 
\begin{enumerate}
	\item[(a)] The non-negative measure \(p(\cdot) \triangleq E^{P^o} [\mu(\cdot)]\) on \(([0, \infty) \times E, \mathscr{B}([0, \infty)) \otimes \mathscr{E})\) has a decomposition \(p (\dd t, \dd x) = \dd t \otimes F(\dd x)\), where \(F\) is a \(\sigma\)-finite measure on \((E, \mathscr{E})\).
	\item[(b)] For all \(s \in [0, \infty)\) and every \(G \in \mathscr{B}([0, \infty)) \otimes \mathscr{E}\) with \(G \subseteq (s, \infty) \times E\) and \(p(G) < \infty\), the random variable \(\mu(\cdot, G)\) is \(P^o\)-independent of \(\mathscr{F}^o_s\).
\end{enumerate}
The measure \(p\) is called the intensity measure of \(\mu\). Moreover, we note that the intensity measure is the predictable compensator of the random measure \(\mu\), see \cite[Proposition II.1.21]{JS} or Section \ref{sec: st int} below.

Now, we assume that \(\mathbb{H}\) is a real separable Hilbert space. With a minor abuse of notation we denote the corresponding scalar product by \(\langle \cdot, \cdot \rangle\). Moreover, we identify \(\mathbb{H}^*\) with \(\mathbb{H}\).
The parameters for an SDE are the following:
\begin{enumerate}
	\item[(i)] Let \(U \in S^+(\mathbb{H}, \mathbb{H})\) be the covariance of a cylindrical Brownian motion.
	\item[(ii)] Let \(p = \dd t \otimes F\) be the intensity measure of a Poisson random measure on~\(E\).
	\item[(iii)] Let \(A \colon D(A) \subseteq \mathbb{B} \to \mathbb{B}\) be a linear, densely defined and closed operator.
	\item[(iv)] Let \(b \colon \mathbb{B} \to \mathbb{B}\) be Borel and such that for all bounded sequences \((y^*_n)_{n \in \mathbb{N}} \subset \mathbb{B}^*\) and all bounded sets \(G \in \mathscr{B}(\mathbb{B})\) it holds that 
	\begin{align*}
	\sup_{n \in \mathbb{N}} \sup_{x \in G} |\la b(x), y^*_n\ra| < \infty.
	\end{align*}
	\item[(v)] Let \(\sigma \colon \mathbb{B} \to L(\mathbb{H}, \mathbb{B})\), where \(L(\mathbb{H}, \mathbb{B})\) denotes the set of all linear bounded operators from \(\mathbb{H}\) to \(\mathbb{B}\), be bounded on bounded subsets of \(\mathbb{B}\) and such that the map \(x \mapsto \sigma^* (x) y^*\) is Borel for all \(y^*\in \mathbb{B}^*\).
	Here, \(\sigma^* (x) \colon \mathbb{B}^* \to \mathbb{H}\) denotes the adjoint of \(\sigma(x)\).
	\item[(vi)] Let \(\delta \colon \mathbb{B} \times E \to \mathbb{B}\) be Borel and such that for all bounded sequences \((y^*_n)_{n \in \mathbb{B}} \subset \mathbb{B}^*\), all bounded sets \(G \in \mathscr{B}(\mathbb{B})\) and all \(\epsilon >0\) it holds that 
	\begin{equation}\label{eq: int intensity measure}
	\begin{split}
	\sup_{n \in \mathbb{N}}&\sup_{x \in G} \int \1_{\{\|\delta(x, y)\| \leq \epsilon\}} |\la \delta (x, y), y^*_n\ra|^2 F(\dd y) < \infty,\\
	&\ \ \sup_{x \in G} F(\{y \in E \colon \|\delta (x, y)\| \geq \epsilon\}) < \infty.
	\end{split}
	\end{equation}
	\item[(vii)] Let \(\eta\) be a probability measure on \((\mathbb{B}, \mathscr{B}(\mathbb{B}))\).
\end{enumerate}
Recall that \(h\) is a truncation function on \(\mathbb{B}\). We denote \(h' (x) \triangleq x - h(x)\).

The following solution concept is in the spirit of an analytically and probabilistically weak solution: 
\begin{definition}
	We call a quadruple \(((\Omega^o, \mathscr{F}^o, \F^o, P^o), Y, W, \mu)\) a solution to the SDE associated with \((U, p, A, b, \sigma, \delta, \eta)\) if
	\begin{enumerate}
		\item[\textup{(i)}] \((\Omega^o, \mathscr{F}^o, \F^o, P^o)\) is a filtered probability space which supports a \cadlag \(\F^o\)-adapted \(\mathbb{B}\)-valued process \(Y\), a cylindrical \((\F^o, P^o)\)-Brownian motion \(W\) with covariance \(U\) and an \((\F^o, P^o)\)-Poisson random measure \(\mu\) with intensity measure~\(p\).
		Moreover, \(\mathscr{F}^o\) is \(P^o\)-complete and \(\F^o\) is \(P^o\)-augmented, see Section \ref{sec: st int} below.
		\item[\textup{(ii)}] \(P^o \circ Y^{-1}_0 = \eta\).
		\item[\textup{(iii)}] For all \(y^* \in D(A^*)\) it holds \(P^o\)-a.s.
		\begin{equation}\label{eq: SSPDE}
		\begin{split}
		\la Y, y^*\ra = \la Y_0, y^*\ra &+ \int_0^\cdot \left(\la Y_{s-}, A^* y^*\ra + \la b(Y_{s-}), y^*\ra \right)\dd s + \int_0^\cdot \la \dd W_s, \sigma^* (Y_{s-}) y^*\ra 
		\\&+ \la h(\delta(\cdot, Y_-)), y^*\ra \star (\mu - p) + \la h'( \delta(\cdot, Y_-)), y^*\ra \star \mu. 
		\end{split}
		\end{equation}
	\end{enumerate}
The process \(Y\) is called a \emph{solution process}.
	The law of a solution process, seen as a probability measure on \((\mathbb{D},\mathscr{D})\), is called a \emph{solution measure} to the SDE.
	We say that a SDE satisfies \emph{uniqueness} if all solution measures coincide.
\end{definition}
For the last two integrals in \eqref{eq: SSPDE} we use the classical notation of Jacod and Shiryaev \cite{JS}. The definitions of all (stochastic) integrals are recalled in Section \ref{sec: st int} below.
Let us emphasis that, based on the local boundedness assumptions on \(\sigma\) and \(\delta\), the integrals are well-defined. 

The case where the driving Brownian motion is Banach space valued is included in this setting, see Remark \ref{rem: RKHS} below.
\begin{remark}
	In the case where \(A\) is a \(C^0\)-semigroup one typically speaks of a (semilinear) stochastic partial differential equation (SPDE). 
	Besides the concept of (analytically and probabilistically) weak solutions to SPDEs, a frequently used solution concept are so-called mild solutions, see, e.g., \cite{DePrato,roeckner15,mandrekar2014stochastic,peszat2007stochastic}. 
	In many cases, the solution concepts are equivalent, see \cite[Appendix G]{roeckner15} and \cite[Section 9.3]{peszat2007stochastic}.
\end{remark}

The following relations between laws of SDEs and solutions to MPs hold.
A proof is given in Section \ref{sec: pf equi GMP SDE} below.
\begin{theorem}\label{theo: GMP SDE}
	The set of solution measures to the SDE \((U, p, A, b, \sigma, \delta, \eta)\) coincides with the set of solutions to the MP \((A, b, \sigma U \sigma^*, K, \eta)\), where
	\begin{align}\label{eq: K SDE}
	K(\cdot,G) \triangleq \int \1_G (\delta(\cdot, y)) F(\dd y), \ G \in \mathscr{B}(\mathbb{B}), 0 \not \in G.
	\end{align}
	In particular, the SDE \((U, p, A, b, \sigma, \delta, \eta)\) has a solution if and only if the MP \((A, b,  \sigma U \sigma^*, K, \eta)\) has a solution.
\end{theorem}
Let us also deduce some results for SDEs. We denote by \(\mathbf{D}\) the filtration on \((\mathbb{D},\mathscr{D})\) generated by the coordinate process \(X\), i.e. \(\mathbf{D} \triangleq(\mathscr{D}_t)_{t \in [0, \infty)}\) with \(\mathscr{D}_t \triangleq \sigma(X_s, s \in [0, t])\).
For a \(\mathbf{D}\)-stopping time \(\xi\) the \(\sigma\)-field \(\mathscr{D}_\xi\) is defined in the same manner as \(\mathscr{F}_\xi\) is defined for an \(\F\)-stopping time \(\xi\).

Resulting from Theorem \ref{theo:main1} and Theorem \ref{theo: GMP SDE} we obtain that unique solution measures to SDEs form a strong Markov family.
\begin{corollary}
	Suppose that for all \(x \in \mathbb{B}\) the SDE \((U, p, A, b, \sigma, \delta, \varepsilon_x)\) has a unique solution measure \(P_x\). Then the family \(\{P_x, x \in \mathbb{B}\}\) is strong Markov in the following sense: the mapping \(\mathbb{B} \ni x \mapsto P_x\) is Borel, and for all \(P_x\)-a.s. finite \(\mathbf{D}\)-stopping times \(\xi\) we have \(P_x\)-a.s. for all \(F \in \mathscr{D}\)
	\[
	P_x \left( \theta^{-1}_\xi F\big| \mathscr{D}_\xi\right) = P_{X_\xi} (F).
	\]
	Moreover, for all probability measures \(\eta\) on \((\mathbb{B}, \mathscr{B}(\mathbb{B}))\) the SDE \((U, p, A, b, \sigma, \delta, \eta)\) has the unique solution measure \eqref{eq: int sol}.
\end{corollary}

Let \(b', c\) and \(Y\) be as in the previous section with \(a\) replaced by \(\sigma U \sigma^*\) and \(K\) given by \eqref{eq: K SDE}, i.e.
\[
b' =  b + \sigma U \sigma^* c + \int h(\delta(\cdot, y)) \left(Y(\cdot, \delta(\cdot, y)) - 1 \right)) F(\dd y).
\]
Now, the following corollary is an immediate consequence of the Propositions \ref{theo:main2}, \ref{coro: well-posed equivalence} and \ref{coro:main1} and Theorem \ref{theo: GMP SDE}.
\begin{corollary}\label{coro: existence SDE}
	Let \(E'\) be a Lusin space and \(\mathscr{E}' \triangleq \mathscr{B}(E')\). Suppose that there exists an intensity measure \(p' = \dd t \otimes F'\) on \(([0, \infty) \times E', \mathscr{B}([0, \infty)) \otimes\mathscr{E}')\) and a Borel function \(\delta' \colon \mathbb{B} \times E' \to \mathbb{B}\) such that
	\[
	\int \1_G(\delta (\cdot, x)) Y(\cdot, \delta (\cdot, x)) F(\dd x) = \int \1_G (\delta'(\cdot, x)) F'(\dd x),\quad G \in \mathscr{B}(\mathbb{B}), 0 \not \in G.
	\]
	If for all \(x \in \mathbb{B}\) the SDE \((U, p, A, b, \sigma, \delta, \varepsilon_x)\) has a unique solution, then for all \(x \in \mathbb{B}\) the SDE \((U, p', A, b', \sigma, \delta', \varepsilon_x)\) satisfies uniqueness. If, additionally, the mappings
	\[
	x \mapsto \la \sigma (x) U \sigma^*(x) c(x), c(x)\ra,\qquad x \mapsto \int \left(1 - \sqrt{Y(x, \delta(x, y))} \right)^2 F(\dd y)
	\]
	are bounded (on \(\mathbb{B}\)), then for all \(x \in \mathbb{B}\) the SDE \((U, p', A, b', \sigma, \delta', \varepsilon_x)\) has a unique solution.
	Moreover, in this case the solution measures to the SDEs \((U, p, A, b, \sigma, \delta, \varepsilon_x)\) and  \((U, p', A, b', \sigma, \delta', \varepsilon_x)\) are locally equivalent (w.r.t. \(\mathbf{D}\)). 
\end{corollary}
The first part of the previous corollary is in the spirit of \cite[Corollary 14.82]{J79}. 
In the following section we discuss a simplified situation.
\subsection{Applications: Existence and Uniqueness Results for MPs and SDEs}\label{sec: appl}
In this section we use our CMG formula to transfer existence and uniqueness results from SDEs of the type
\begin{align*}
\dd Y_t = A Y_{t-} \dd t + \dd L_t,
\end{align*}
to SDEs of the type 
\begin{align}\label{eq: SDE uni cont}
\dd Y_t = \left(A Y_{t-} + c(Y_{t-})\right) \dd t + \dd L_t,
\end{align}
where \(L\) is a L\'evy process. 
This is a classical application of Girsanov-type theorems. 

In the case where \(L\) is a cylindrical Brownian motion, a similar application is shown in \cite{roeckner15}. In this continuous setting, under additional assumptions, Kunze \cite{doi:10.1080/17442508.2012.712973} proved a more general existence and uniqueness result and Da Prato, Flandoli, Priola and R\"ockner \cite{daprato2013,DaPrato2015} showed that SDEs of the type \eqref{eq: SDE uni cont} are even pathwise unique.

We stress that our CMG formula also applies to situations where we have a term \(\gamma(Y_{t-}) \dd L_t\) instead of only \(\dd L_t\). Thus, the technique is in general very robust. We only want to illustrate the idea and therefore we restrict ourselves to a simple setup.

We impose the following additional assumptions:
Suppose that \(\mathbb{B} =  \mathbb{H} = E\) is a separable Hilbert space.
Let \(A\) be the generator of a pseudo-contraction semigroup \(S = (S_t)_{t \in [0, \infty)}\), i.e. \(S\) is a \(C^0\)-semigroup on \(\mathbb{H}\) and there is a \(\beta \in \mathbb{R}\) such that \(\|S_t\| \leq e^{\beta t}\) for all \(t \in [0, \infty)\).
Let \(F\) be a measure on \((\mathbb{H}, \mathscr{B}(\mathbb{H}))\) such that \(F(\{0\}) = 0\) and 
\(
\int \|x\|^2 F(\dd x) < \infty.
\)
In this case, \(\dd t \otimes F\) is the intensity measure of a homogeneous Poisson random measure on \(\mathbb{H}\), see \cite[Theorem 3.2.11]{mandrekar2014stochastic}.
We set \(b(x) \triangleq b\) for \(b \in \mathbb{H}\), \(a (x) \triangleq a\), where \(a \in S^+(\mathbb{H}, \mathbb{H})\) is a trace class operator from \(\mathbb{H}\) to \(\mathbb{H}\), and
\(K(y, \dd t, \dd x) \triangleq \dd t \otimes F(\dd x)\). The following lemma is a consequence of existence and uniqueness results for SDEs as given in \cite{doi:10.1080/17442501003624407} together with a classical Yamada-Watanabe-type argument. We give the details in Section \ref{sec: pf lemma well-posed} below. 
\begin{lemma}\label{lem: linear well-posed}
	The MP \((A, b, a, K)\) is well-posed.
\end{lemma}
Let \(b'\) and \(K'\) be given as in \eqref{eq: b', Y'}. 
As an immediate consequence of Proposition \ref{theo:main2}, Proposition \ref{coro: well-posed equivalence} and Corollary \ref{coro: existence SDE}, we obtain the following
\begin{corollary}\label{coro: appl uni}
	\begin{enumerate}
		\item [\textup{(i)}]
		The GMP \((A, b', a, K', \eta, \tau_{\Delta}-)\) satisfies uniqueness. If the mappings \eqref{eq: bdd assp} are bounded (on \(\mathbb{H}\)), then the MP \((A, b', a, K')\) is well-posed.
		\item[\textup{(ii)}] Suppose that there exists a Borel mapping \(\delta' \colon \mathbb{H} \times \mathbb{H} \to \mathbb{H}\) and a measure \(F'\) on \((\mathbb{H}, \mathscr{B}(\mathbb{H}))\) such that 
		\begin{align}\label{eq: comp formula}
		\int_G Y(x, y) F(\dd y) = \int \1_G (\delta' (\cdot, y)) F'(\dd y),\quad G \in \mathscr{B}(\mathbb{B}), 0 \not \in G.
		\end{align}
		Moreover, assume that \(F'\) integrates \(1 \wedge \|x\|^2\), \(F'(\{0\}) = 0\) and that \(\delta'\) satisfies \eqref{eq: int intensity measure} with \(\delta\) replaced by \(\delta'\). For all \(x \in \mathbb{H}\) let \(l(x) = x\) be the identity on \(\mathbb{H}\). Set \(p' \triangleq \dd t\otimes F'\).
		Then the SDE \((a, p', A, b', l, \delta', \eta)\) satisfies uniqueness. In particular, if the mappings \eqref{eq: bdd assp} are bounded (on \(\mathbb{H}\)), then the SDE \((a, p', A, b', l, \delta, \eta)\) has a unique solution.
	\end{enumerate}
\end{corollary}
The formula \eqref{eq: comp formula} clearly holds in the continuous case and also if \(Y(x, y) = Y(y)\) and 
\(
\int \left(1 \wedge \|x\|^2\right)Y(x) F(\dd x) < \infty.
\)
In the next section we prove our results.

\section{Proofs}\label{sec: pf}
\subsection{Some Preparations}\label{sec: prep}
We start with some technical results.
The main observations in this section are that GMPs explode in a continuous manner, see
Lemma \ref{lem:tau} below, and that GMPs are determined by a countable set of test functions, see Proposition \ref{coro:tau} below.
\subsubsection{A Short Recap on Stochastic Integration}\label{sec: st int}
For two random times \(\rho\) and \(\tau\) on \((\Omega, \mathscr{F})\) we define the stochastic interval
\[
\of \rho, \tau\gs \triangleq \{(\omega, t) \in \Omega \times [0, \infty) \colon \rho(\omega) \leq t \leq \tau(\omega)\}.
\]
In the same manner, we define the stochastic intervals \(\of \rho, \tau\of, \gs \rho, \tau\gs, \gs \rho, \tau \of\). Moreover, we set \(\of \tau\gs \triangleq \of \tau, \tau\gs\).

Let \(P\) be a probability measure on \((\Omega, \mathscr{F})\). The set of all \(P\)-null sets is denoted by \(\mathscr{N}\).
We define
\(
\mathscr{F}_t^P \triangleq \sigma(\mathscr{F}_{t+}, \mathscr{N}).
\)
In fact, it holds that \(\mathscr{F}^P_t = \bigcap_{s > t} \sigma(\mathscr{F}_s, \mathscr{N})\), see \cite[Lemma 6.8]{Kallenberg}.
Denote \(\mathscr{F}^P\triangleq \sigma(\mathscr{F}, \mathscr{N})\).
The filtration \(\F^P = (\mathscr{F}^P_t)_{t \in [0, \infty)}\) on \((\Omega, \mathscr{F}^P)\) is called the \(P\)-augmentation of \(\F\).

The following material concerning random measures is taken from \cite{JS,liptser1989theory}.
The \(\F^P\)-predictable \(\sigma\)-field is denoted by \(\mathscr{P}^P\) and the \(\F^P\)-optional \(\sigma\)-field is denoted by \(\mathscr{O}^P\).
Let \(E\) be a Lusin space and set \(\mathscr{E} \triangleq \mathscr{B}(E)\).
We say that a random measure \(\mu\) on \(([0, \infty) \times E, \mathscr{B}([0, \infty)) \otimes \mathscr{E})\) is \(\F^P\)-predictable (resp. \(\F^P\)-optional) if for all non-negative \(\mathscr{P}^P\otimes \mathscr{E}\)-measurable (resp. \(\mathscr{O}^P\otimes \mathscr{E}\)-measurable) functions \(U\) the process 
\[
W \star \mu_\cdot \triangleq \int_0^\cdot \int W(s, x) \mu(\dd s, \dd x)
\]
is \(\F^P\)-predictable (resp. \(\F^P\)-optional).
We define the Dol\'eans measure \(M^P_\mu\) by
\[
M^P_\mu(\dd \omega, \dd t, \dd x) \triangleq \mu(\omega, \dd t, \dd x) P(\dd \omega).
\]
We say that \(M^P_\mu\) is \(\mathscr{P}^P\)-\(\sigma\)-finite if there exist a sequences \((\Omega_{n})_{n \in \mathbb{N}} \subset \mathscr{P}^P \otimes \mathscr{E}\) such that \(\bigcup_{n \in \mathbb{N}} \Omega_n = \Omega \times [0, \infty) \times E\) and
\(
M^P_\mu(\Omega_{n} ) < \infty\) for all \(n \in \mathbb{N}.
\)
Clearly, this is equivalent to the existence of a sequence \((\Omega^*_n)_{n \in \mathbb{N}} \subset \mathscr{P}^P \otimes \mathscr{E}\) such that \(\Omega^*_n \nearrow \Omega \times [0, \infty) \times E\) as \(n \to \infty\) and
\(
M^P_\mu(\Omega^*_{n} ) < \infty\) for all \(n \in \mathbb{N}.
\)
Indeed, \(\Omega^*_n \triangleq \bigcup_{k = 1}^n \Omega_k\) is such a sequence.
An \(\F^P\)-predictable random measure \(\nu\) is called an \(\F^P\)-predictable \(P\)-compensator of an \(\F^P\)-optional random measure \(\mu\) with \(\mathscr{P}^P\)-\(\sigma\)-finite Dol\'eans measure \(M^P_\mu\) if 
for all non-negative \(\mathscr{P}^P\otimes \mathscr{E}\)-measurable functions \(W\) it holds that \(
M^P_\mu (W) = M^P_\nu( W).
\)
It is classical that each \(\F^P\)-optional random measure \(\mu\) with \(\mathscr{P}^P\)-\(\sigma\)-finite Dol\'eans measure \(M^P_\mu\) has an \(\F^P\)-predictable \(P\)-compensator which is unique to a \(P\)-null set. 
Suppose that 
\begin{align}\label{eq: rmj}
\mu(\omega, \dd t, \dd x) \triangleq \sum_{s \in [0, \infty)} \1_D (\omega, s) \varepsilon_{(s, \beta_s(\omega))} (\dd t,\dd x),
\end{align}
where \(D\) is an \(\F^P\)-thin set with \((\omega, 0) \not \in D\) and \(\beta\) is an \(E\)-valued \(\F^P\)-optional process.
Here, a set \(D \subset \Omega \times [0, \infty)\) is called \(\F^P\)-thin, if \(D = \bigcup_{n \in \mathbb{N}} \of \rho_n\gs\) for a sequence \((\rho_n)_{n \in \mathbb{N}}\) of \(\F^P\)-stopping times. 
In this case, it is well-known that \(\mu\) is an \(\F^P\)-optional random measure. Furthermore, suppose that \(M^P_\mu\) is \(\mathscr{P}^P\)-\(\sigma\)-finite and denote the \(\F^P\)-predictable \(P\)-compensator of \(\mu\) by \(\nu\).
Define \(G_\textup{loc}(\mu)\) to be the set of all real-valued \(\mathscr{P}^P \otimes \mathscr{E}\)-measurable functions \(V\) such that 
\begin{align}\label{eq: jump jump int} 
\left(\sum_{s \in [0, \cdot]} \left(V(\omega, s, \beta_s(\omega)) \1_D(\omega, s) - \int V(\omega, s, x) \nu(\omega, \{s\}, \dd x)\right)^2 \right)^{\frac{1}{2}}\end{align}
is \(\F^P\)-locally \(P\)-integrable. 
If \(V \in G_\textup{loc}(\mu)\), then there exists a discontinuous local \((\F^P, P)\)-martingale \(V \star (\mu - \nu)\) whose jump process is \(P\)-indistinguishable from \[V(\cdot, \cdot, \beta) \1_D - \int V(\cdot, \cdot, x) \nu(\cdot,\{\cdot\}, \dd x).\] 

Let us also recall the definition of a stochastic integral w.r.t. a cylindrical continuous local \((\F^P, P)\)-martingale, following the exposition given in \cite{Mikulevicius1998}.
We stress that we only consider functionals as integrands. The case where the integrand takes values in an arbitrary Banach space is much more delicate.

Let \(K\colon \Omega \times [0, \infty) \to S^+(\mathbb{B}^*, \mathbb{B})\) be such that \(K y^*\) is \(\F^P\)-predictable for all \(y^* \in \mathbb{B}^*\).
We denote by \(L^2_\textup{loc}(K)\) the space of all \(\F^P\)-predictable functions \(f^* \colon \Omega \times [0, \infty) \to\mathbb{B}^*\) such that \(P\)-a.s.
\begin{align}\label{eq: f int}
\int_0^t \la K_s f_s^*, f^*_s\ra \dd s < \infty, \quad  t \in [0, \infty).
\end{align}
Let \(M\) be a cylindrical continuous local \((\F^P, P)\)-martingale with
\[
\lle M (y^*)\rre = \int_0^\cdot \la K_s y^*, y^*\ra \dd s.
\] 
Due to \cite[Proposition 9]{Mikulevicius1998}
and \cite[Propositions IV.2.4 and IV.2.5]{RY} we have the following
\begin{lemma}\label{lem: definition si}
	For all \(f^* \in L^2_\textup{loc}(K)\) there exists a continuous local \((\F^P, P)\)-martingale \(\int_0^\cdot \la \dd M_s, f^*_s \ra\), which is unique up to \(P\)-indistinguishability, such that 
	\begin{align}\label{eq: qv}
	\left \la\hspace{-0.2cm} \left \la \int_0^\cdot \la \dd M_s, f^*_s \ra \right\ra\hspace{-0.2cm} \right\ra = \int_0^\cdot \la K_sf^*_s, f^*_s \ra \dd s
	\end{align}
	up to \(P\)-evanescence.
	Furthermore, for \(\lambda \in \mathbb{R}\) and \(g^* \in L^2_\textup{loc}(K)\) we have \(\lambda f^* + g^* \in L^2_\textup{loc}(K)\) and 
	\[
	\int_0^\cdot \la \dd M_s, \lambda f^*_s + g^*_s \ra = \lambda \int_0^\cdot \la \dd M_s, f^*_s \ra + \int_0^{\cdot} \la \dd M_s, g^*_s \ra
	\]
	up to \(P\)-evanescence,
	and for all \(\F^P\)-stopping times \(\xi\) it holds that 
	\[
	\int_0^{\cdot \wedge \xi} \la \dd M_s, f^*_s \ra = \int_0^\cdot \la \dd M_s, f^*_s \1_{\{s \leq \xi\}} \ra = \int_0^\cdot \la \dd M_{s \wedge \xi}, f^*_s\ra
	\]
	up to \(P\)-evanescence.
	Finally, if \((f^n)_{n \in \mathbb{N}} \subset L^2_\textup{loc}(K)\) is a sequence such that 
	\begin{align*}
	\int_0^t \langle K_s (f^*_s - f^n_s), f^*_s - f^n_s\rangle \dd s \xrightarrow{n \to \infty} 0 
	\end{align*}
	in \(P\)-probability, then 
	\begin{align*}
	\sup_{s \in [0, t]} \left| \int_0^s \langle \dd M_r, f^*_r\rangle - \int_0^s \langle \dd M_r, f^n_r\rangle \right| \xrightarrow{n \to \infty} 0
	\end{align*}
	in \(P\)-probability.
\end{lemma}
\begin{remark}\label{rem: RKHS}
	The proof of \cite[Proposition 9]{Mikulevicius1998} relies on the following facts (see \cite[Proposition 2, Corollary 3]{Mikulevicius1998}): For \(K \in S^+(\mathbb{B}^*, \mathbb{B})\) there exists a completion \(\mathbb{H}\) of \(K \mathbb{B}^*\) w.r.t. the scalar product \((\cdot, \cdot)\) defined by
	\(
	(Kx^*, Ky^*) \triangleq \la K x^*, y^*\ra\) for \(x^*, y^* \in \mathbb{B}^*,
	\)
	such that \(\mathbb{H} \subseteq \mathbb{B}\) and the natural embedding \(i_{K} \colon \mathbb{H} \hookrightarrow \mathbb{B}\) is continuous.
	Moreover, \(\mathbb{H}\) is a separable Hilbert space, which is typically called the reproducing kernel Hilbert space associated with \(K\).
	The operator \(K\) is the covariance of a Gaussian measure on \(\mathbb{B}\) if and only if the embedding \(i_K\) is \(\gamma\)-radonifying, see \cite{J2005}. 
	In this case, if \(W\) is a cylindrical Brownian motion with covariance \(K\), then
	\(
	W^\mathbb{H} i^*_K x^* \triangleq W( x^*)\)
	uniquely extends to a cylindrical Brownian motion with identity covariance on the Hilbert space \(\mathbb{H}\).
\end{remark}

\subsubsection{Continuous Explosion}\label{sec: cont expl}
Let \(\tau_n\) be defined as in \eqref{eq:tau n}.
Obviously, \(\tau_n \leq \tau_\Delta\). In fact, if \(P\) solves a GMP up to \(\tau_\Delta\) and \(P\)-a.s. \(X_0 \not = \Delta\), then the inequality is strict up to a \(P\)-null set.
The proof is inspired by ideas given in \cite{CFY}.
\begin{lemma}\label{lem:tau}
	Suppose that \(P\) is a probability measure on \((\Omega, \mathscr{F})\) such that for all \(f = g(\la \cdot, y^*\ra)\) with \(g \in C^2_c(\mathbb{R})\) and \(y^* \in D(A^*)\) the process \(M^f_{\cdot \wedge \tau_n}\) is a local \((\F, P)\)-martingale. Then
	\begin{align}\label{eq: tau n eta}
	P(\tau_n < \tau_\Delta) =  P (X_0 \not = \Delta).\end{align}
	In particular, if \(P\)-a.s. \(X_0 \not = \Delta\), then \(P\)-a.s. \(X_{\tau_{\Delta}-} = \Delta\) on \(\{\tau_{\Delta} < \infty\}\).
\end{lemma}
\begin{proof}
	For all \(k \in \mathbb{N}\) let \(g_k \colon \mathbb{R} \to [0, 1]\) be in \(C^2_c(\mathbb{R})\) such that \(g_k = 1\) on \((- k, k)\). 
	Fix \(y^* \in D(A^*)\) such that \(\|y^*\| = 1\).
	Set \[G_{k, n} \triangleq \{z \in \mathbb{B} \colon \|z\| \geq  k - n\}\] and note that for all \(k > n\), \(y \in \complement G_{k, n}\) and \(x \in \{z \in \mathbb{B} \colon \|z\| \leq n\}\) we have
	\begin{align*}
	g_k(\la x + y, y^*\ra) - g_k(\la x, y^*\ra) &= 0,\\
	\partial g_k (\la x, y^*\ra) &=  0,\\ 
	\partial^2 g_k (\la x, y^*\ra) &= 0.\end{align*}
	Hence, if we set \(f_k \triangleq g_k(\la \cdot, y^*\ra) \in \mathcal{C}\), then for all \(k > n\) we have
	\begin{align*}
	M^{f_k}_{\cdot \wedge \tau_n} 
	&= g_k(\la X_{\cdot \wedge \tau_n}, y^* \ra) - g_k(\la X_0, y^* \ra) 
	\\&\qquad\quad- \int_0^{\cdot \wedge \tau_n} \int_{G_{k, n}}  \left(g_k(\la X_{s-} + x, y^*\ra) - 1 \right)K(X_{s-}, \dd x) \dd s.
	\end{align*}
	By assumption, \(M^{f_k}_{\cdot \wedge \tau_n}\) is a local \((\F, P)\)-martingale. In fact, due to the uniform bound
	\begin{align}\label{eq:uniformbound}
	|M^{f_k}_{t \wedge \tau_n(\omega)}(\omega)| \leq 2 + 2 n \sup_{\|z\| \leq n}K(z, G_{k, n}), \quad t \in [0, \infty),\omega \in \Omega,
	\end{align}
	which is finite due to \eqref{eq: int ball around origin}, it is even an \((\F, P)\)-martingale. 
	For all \(\epsilon > 0\), \(k \geq n + \epsilon\) and \((\omega, t) \in \of 0, \tau_n\gs\) we have
	\begin{align*}
	\int_{G_{k, n}} | g_k(\la X_{t-}(\omega) + x, y^*\ra) - 1| K(X_{t-}(\omega), \dd x) &\leq 2 \sup_{\|z\| \leq n} K(z, G_{k, n}) 
	\\&\leq 2 \sup_{\| z\| \leq n} K(z, \{x \in \mathbb{B} \colon \|x\| \geq \epsilon\}), 
	\end{align*}
	which is finite due to \eqref{eq: int ball around origin}.
	Moreover, for all \((\omega, t) \in \of 0, \tau_n\gs\)
	\begin{align*}
	\int_{G_{k, n}} | g_k(\la X_{t-}(\omega) + x, y^*\ra) - 1| K&(X_{t-}(\omega), \dd x) 
\\&	\leq 2 K(X_{t-}(\omega), G_{k, n}) \xrightarrow{k \to \infty} 0.
	\end{align*}
	Hence, by the dominated convergence theorem, for all \(\omega \in \Omega\)
	\begin{align*}
	\int_0^{\tau_n(\omega)} \int_{G_{k, n}} | g_k(\la X_{s-}(\omega) + x, y^*&\ra) - 1| K(X_{s-}(\omega), \dd x) \dd s \xrightarrow{k \to \infty} 0.
	\end{align*}
	Therefore, we obtain that for all \(\omega \in \Omega\) 
	\begin{align*} \lim_{k \to \infty} M^{f_k}_{n \wedge \tau_n(\omega)}(\omega) &=  \1_\mathbb{B} (X_{\tau_n(\omega)}(\omega)) - \1_{\mathbb{B}}(X_0(\omega))
	\\&= \1_{\{\tau_n(\omega) < \tau_{\Delta}(\omega)\}} - \1_\mathbb{B}(X_0(\omega)).\end{align*}
	For all \(\epsilon > 0\) and \(k \geq n + \epsilon\) we deduce from \eqref{eq:uniformbound} that 
	\[
	|M^{f_k}_{n \wedge \tau_n(\omega)} (\omega)| \leq 2 + 2 n \sup_{\| z \| \leq n} K(z, \{x \in \mathbb{B} \colon \|x\| \geq \epsilon\}) < \infty,
	\]
	see \eqref{eq: int ball around origin}.
	Therefore, by the dominated convergence theorem and the martingale property of \(M^{f^k}_{\cdot \wedge \tau_n}\), we obtain
	\begin{align*}
	P(\tau_n < \tau_{\Delta}) - P(X_0 \not = \Delta) &= \lim_{k \to \infty} E^P \left[ M^{f_k}_{n \wedge \tau_n}\right] = \lim_{k \to \infty} E^P\left[M^{f_k}_0\right] = 0.
	\end{align*}
	This concludes the proof of formula \eqref{eq: tau n eta}. 
\end{proof}
\begin{remark}\label{rem: exp time pred}
	If \(P\in \mathcal{M}(A, b, a, K, \eta, \tau_{\Delta}-)\) with \(\eta(\mathbb{B}) = 1\), then \(\tau_{\Delta}\) is an \(\F^P\)-predictable time.
	To see this, note that Lemma \ref{lem:tau} yields that \(\tau_{\Delta}\) can be \(P\)-announced by \((\tau_n)_{n \in \mathbb{N}}\).
	Hence,
	\cite[Theorem IV.71]{DellacherieMeyer78} yields that \(\tau_{\Delta}\) is \(\F^P\)-predictable.
\end{remark}
\subsubsection{An Approximation Lemma}\label{sec: approx lem}
Let \(C^2_b(\mathbb{R}^d)\) be the set of all bounded twice continuously differentiable functions \(\mathbb{R}^d\to \mathbb{R}\) with bounded gradient and bounded Hessian matrix.
Denote 
\[
\mathcal{B} \triangleq \left\{g(\langle \cdot, y^*_1\rangle, ..., \langle \cdot, y^*_m\rangle) \colon g \in C^2_b(\mathbb{R}^m), y^*_1, ..., y^*_m \in D(A^*), m \in \mathbb{N}\right\}.
\]
For \(f \in \mathcal{B}\) we define \(M^f_{\cdot \wedge \tau_n}\) as in \eqref{f - K}.
\begin{lemma}\label{lem: mart}
	For all \(f \in \mathcal{B}\) the function \(M^f_{\cdot \wedge \tau_n}\) is uniformly bounded on \(\Omega \times [0, \infty)\). Moreover, for any probability measure on \((\Omega, \mathscr{F})\) and any \(f \in \mathcal{B}\) there exists a sequence \((f_k)_{k \in \mathbb{N}}\subset \mathcal{C}\) such that, for all \(t \in [0, \infty)\), \(M_{t \wedge \tau_n}^{f_k} \to M^f_{t \wedge \tau_n}\) identically and in \(L^1\) as \(k \to \infty\).
\end{lemma}
\begin{proof}
	Let \(f \in \mathcal{B}\) and suppose that \(f = g(\la \cdot, y^*_1\ra, ..., \la \cdot, y^*_m\ra)\).
	Since \(h\) is a truncation function, there exists an \(\epsilon > 0\) such that \(h(x) = x\) on \(\{y \in \mathbb{B}\colon \|y\| \leq \epsilon \}\).
	By Taylor's theorem we obtain that for all \(z \in \mathbb{B}\) and \(x \in \{y\in \mathbb{B} \colon \|y\| \leq \epsilon\}\)
	\begin{equation}\label{taylor}
	\begin{split}
	&\left| f(z + x) - f(z) -  \sum_{i = 1}^m \langle h(x), y^*_i\rangle \partial_i f(z) \right|
	\\&\quad\qquad\qquad \leq \left(\frac{1}{2} \sum_{i = 1}^m \sum_{j = 1}^m \|\partial^2_{ij} g\|_\infty \right) \max_{i =1,..., m} |\la x, y^*_i\ra|^2.
	\end{split}
	\end{equation}
	Hence, for all \((\omega, s) \in \of 0, \tau_n\gs\)
	\begin{align*}
	\int&\left|f (X_{s-} (\omega) +x) - f(X_{s-}(\omega)) - \sum_{i = 1}^m \langle h(x), y^*_i\rangle \partial_i f(X_{s-}(\omega))\right| K(X_{s-}(\omega), \dd x) 
	\\&\hspace{2.5cm}\leq \textup{const. } \left(1 + \max_{i = 1, ..., m} \|y^*_i\|\right) \sup_{\|z\| \leq n} K(z, \{x \in \mathbb{B} \colon \|x\| \geq \epsilon\}) \\&\hspace{3cm}\quad \quad +\textup{const. } \max_{i =1, ..., m} \sup_{\|z\| \leq n} \int \1_{\{\|x\| \leq \epsilon\}} |\la x, y^*_i\ra|^2 K(z, \dd x),
	\end{align*}
	where the constants depend on \(g, h\) and \(m\).
	Therefore, we obtain the uniform bound
	\begin{equation}\label{uni bound}
	\begin{split}
\left|M^{f}_{\cdot \wedge \tau_n}\right| \leq 2 \|g\|_\infty &+ n \sum_{i = 1}^m \|\partial_i g\|_\infty \left(n \|A^* y^*_i\| 
	+ \sup_{\|x\| \leq n} |\la b(x),y^*_i\ra|\right) 
	\\&+ \frac{n}{2} \sum_{i = 1}^m\sum_{j = 1}^m \sup_{\|x\| \leq n} \la a(x) y^*_i, y^*_j\ra \|\partial^2_{ij} g\|_\infty
	\\&+ \textup{const.} \left(1 + \max_{i = 1, ..., m} \|y^*_i\|\right) \sup_{\|z\| \leq n} K(z, \{x \in \mathbb{B} \colon \|x\| \geq \epsilon\}) \\&+ \textup{const.} \max_{i =1, ..., m} \sup_{\|z\| \leq n} \int \1_{\{\|x\| \leq \epsilon\}} |\la x, y^*_i\ra|^2 K(z, \dd x),
	\end{split}
	\end{equation}
	where the constants depend on \(n, g, h\) and \(m\).
	This proves the first claim of the lemma.
	
	For \(k \in \mathbb{N}\) let \(g_k \colon \mathbb{R}^m \to [0, 1]\) be in \(C^2_c(\mathbb{R}^m)\) such that \(g_k = 1\) on \(\{z \in \mathbb{R}^m \colon \sum_{i = 1}^m |z^i| < k\}\). 
	Then, \(gg_k \in C^2_c(\mathbb{R}^m)\).
	Set 
	\begin{align}\label{eq: fk}
	f^k \triangleq g(\langle \cdot, y^*_1\rangle, ..., \langle \cdot, y^*_m\rangle) g_k(\langle \cdot, y^*_1\rangle, ..., \langle \cdot, y^*_m\rangle).
	\end{align}
	W.l.o.g. we may assume that \(\sum_{j = 1}^{m} \|y^*_j\| > 0\). Set 
	\[
	G_{k, n} \triangleq \left\{z \in \mathbb{B} \colon \|z\| \geq k \left(\sum_{j = 1}^{m} \|y^*_j\|\right)^{-1} - n\right\}.
	\]
	Then, for all \(k > n \sum_{j = 1}^{m} \|y^*_j\|\), \(y \in \complement G_{k, n}\) and \(x \in \{z \in \mathbb{B}\colon \|z\| \leq n\}\) we have 
	\begin{align*}
	f^k(x + y) - f(x + y) &= 0,\\
	\partial_i f^k (x) - \partial_i f (x) &= 0,\\\partial^2_{ij} f^k(x) - \partial^2_{ij} f (x)&= 0.\end{align*}
	Let \(t \in [0, \infty)\) be fixed. We obtain for \(k > n \sum_{j = 1}^{m} \|y^*_j\|\)
	\begin{align*}
	|M^{f^k}_{t \wedge \tau_n} - M^f_{t \wedge \tau_n}| &\leq | f^k (X_{t \wedge \tau_n}) - f (X_{t \wedge \tau_n})| 
	\\&\quad + \int_0^{t \wedge \tau_n} \int_{G_{k, n}} | f^k(X_{s-} + x) - f (X_{s-} + x)| K(X_{s-}, \dd x) \dd s.
	\end{align*}
	Note that the first term converges to zero as \(k \to \infty\). Since it is bounded by \(2 \|g\|_\infty\), by the dominated convergence theorem, it also converges in \(L^1\) to zero as \(k \to \infty\).
	Arguing as in the proof of Lemma \ref{lem:tau}, 
	we deduce that the second term converges to zero as \(k \to \infty\). In particular, we have
	\begin{align*}
	E^P &\left[ \int_0^{t \wedge \tau_n}\int_{G_{k, n}} | f^k(X_{s-} + x) - f (X_{s-} + x)| K(X_{s-}, \dd x) \dd s\right] \\&\qquad\qquad\qquad\qquad\qquad\qquad\leq 2 \|g\|_\infty E^P \left[ \int_{0}^{t \wedge \tau_n} K(X_{s-}, G_{k, n})\dd s\right] \to 0
	\end{align*}
	as \(k \to \infty\) by the monotone convergence theorem.
	This concludes the proof.
\end{proof}
\begin{corollary}\label{coro: mart C}
	If \(P\) is a solution to the GMP \((A, b, a, K, \eta, \tau_{\Delta}-)\), then
	\(M^{f}_{\cdot \wedge \tau_n}\) is an \((\F, P)\)-martingale for all \(f \in \mathcal{B}\).	
\end{corollary}	
\begin{proof}
	Note the following elementary fact: If \((Y^n)_{n \in \mathbb{N}}\) is a sequence of \((\F, P)\)-martingales and \(Y\) is an \(\F\)-adapted and \(P\)-integrable process such that \(Y^n_t \to Y_t\) in \(L^1\) as \(n \to \infty\) for all \(t \in [0, \infty)\), then \(Y\) is an \((\F, P)\)-martingale. To see this, it suffices to note that for all \(s \in [0, t]\) and \(A \in \mathscr{F}_s\)
	\begin{align*}
	E^P \left[\1_A (Y_t - Y_s)\right] = \lim_{n \to \infty} E^P \left[\1_A (Y^n_t - Y^n_s)\right] = 0.
	\end{align*}
	Now, the claim follows from Lemma \ref{lem: mart}.
\end{proof}
\subsubsection{GMPs and Semimartingales}\label{sec: mp semi}
In this section we reformulate the GMP in terms of classical semimartingale theory.
We shortly recall the concept of semimartingale characteristics, see \cite{JS} for more details. An \(m\)-dimensional semimartingale \(Y\) is an \(\mathbb{R}^m\)-valued \cadlag adapted process which has a decomposition \[Y = Y_0 + M + B,\] where \(Y_0\) is measurable w.r.t. the initial \(\sigma\)-field, \(M\) is a local martingale starting at 0, and \(B\) is a \cadlag adapted process of finite variation which also starts at 0. Let \(k \colon \mathbb{R}^m \to \mathbb{R}^m\) be a truncation function. The modified semimartingale \[Y(k) \triangleq Y - \sum_{s \leq \cdot} (\Delta Y_s - k(\Delta Y_s))\] is a so-called special semimartingale, i.e. it admits a unique decomposition \[Y(k) = Y_0 + M(k) + B(k),\] where \(M(k)\) is a local martingale starting at 0, and \(B(k)\) is a \cadlag predictable process of finite variation which starts at 0. 
The local martingale \(M(k)\) has a unique decomposition into a continuous local martingale \(Y^c\) and a purely discontinuous local martingale. The continuous local martingale \(Y^c\) is independent of the choice of \(k\). We denote \(C \triangleq \lle Y^c \rre\).
Moreover, it is well-known that the integer-valued random measure \(\mu^Y\) associated to the jumps of \(Y\), i.e.
\[
\mu^Y(\dd t, \dd x) \triangleq  \sum_{s \in [0, \infty)} \1_{\{\Delta Y_s \not = 0\}}  \varepsilon_{(s, \Delta Y_s)}(\dd t, \dd x),
\] has a \(\mathscr{P}\)-\(\sigma\)-finite Dol\'eans measure. We denote the compensator of \(\mu^Y\) by \(\nu\). The triplet \((B(k), C, \nu)\) is called the semimartingale characteristics of \(Y\).

Let us introduce some additional notation. We set 
\begin{align}\label{eq: gamma}
\Gamma \triangleq \bigcap_{n \in \mathbb{N}} \{\tau_n < \tau_{\Delta}\}
\end{align}
and define the process 
\begin{align}\label{eq: hat X}
\widehat{X} \triangleq \begin{cases} X \1_{\of 0, \tau_{\Delta}\of},&\textup{on } \Gamma,\\
\widehat{x},&\textup{otherwise},
\end{cases}
\end{align}
where \(\widehat{x}\) is an arbitrary element in \(\mathbb{B}\).
We stress that \(\widehat{X}\) is still not a \cadlag \(\mathbb{B}\)-valued process because \(\widehat{X}_{\tau_{\Delta}-} = \Delta\) is possible on \(\{\tau_{\Delta} < \infty\}\).
However, if \(Q\) is a probability measure on \((\Omega, \mathscr{F})\) such that \(\Gamma\) is \(Q\)-full, then \(\widehat{X}_{\cdot \wedge \tau_n}\) is a \cadlag \(\mathbb{B}\)-valued \(\F^Q\)-adapted process and \(\widehat{X}_{\cdot \wedge \tau_n}\) and \(X_{\cdot \wedge \tau_n}\) are \(Q\)-indistinguishable.
Moreover, for all \(\omega \in \mathbb{D}\) it holds that \(\widehat{X}(\omega) = X(\omega)\).

Recall that a process with values in the complex plane is called local martingale if both its real and its imaginary part are (real-valued) local martingales.

\begin{lemma}\label{lem: equi JS}
	Suppose that \(\eta\) is a probability measure on \((\mathbb{B}_\Delta, \mathscr{B}(\mathbb{B}_\Delta))\) such that \(\eta(\mathbb{B}) = 1\).
	For a probability measure \(P\) on \((\Omega, \mathscr{F})\) the following are equivalent:
	\begin{enumerate}
		\item[\textup{(i)}]\(P\) is a solution to the GMP \((A, b, a, K, \eta, \tau_\Delta -)\).
		\item[\textup{(ii)}] For all \(n \in \mathbb{N}\), \(g \in C^2_c(\mathbb{R})\) and \(y^* \in D(A^*)\) it holds that \(P \circ X^{-1}_0 = \eta\) and \(M^f_{\cdot \wedge \tau_n}\), given by \eqref{f - K} with \(f = g(\la \cdot, y^*\ra)\), is an \((\F, P)\)-martingale.
		\item[\textup{(iii)}] For all \(n \in \mathbb{N}\) and \(u \in \mathbb{R}\) it holds that \(P \circ X^{-1}_0 = \eta\), \(P(\tau_n < \tau_\Delta) = 1\) and the process
		\begin{align*}
		M^*_{\cdot \wedge \tau_n} \triangleq e^{iu \la \widehat{X}_{\cdot \wedge \tau_n}, y^* \ra} &- V^*_{\cdot \wedge \tau_n},
		\end{align*}
		where \(i \triangleq \sqrt{-1}\) and \(V^*_{\cdot \wedge \tau_n}\) is defined to be
		\begin{align*}
		&\int_0^{\cdot \wedge \tau_n}e^{iu \la \widehat{X}_{s-}, y^*\ra} \left( iu \left(\langle \widehat{X}_{s-}, A^* y^*\ra + \la b(X_{s-}), y^*\ra \right) - \frac{u^2}{2}  \langle  a(\widehat{X}_{s-}) y^*, y^*\ra \right)\dd s\\&\hspace{0.9cm}+ \int_0^{\cdot \wedge \tau_n}e^{iu \la \widehat{X}_{s-}, y^*\ra} \int \left(e^{i u \la x, y^*\ra} - 1 - i u \langle h(x), y^*\ra \right)K(\widehat{X}_{s-}, \dd x) \dd s,
		\end{align*}
		is a complex-valued local \((\F^P, P)\)-martingale.
		\item[\textup{(iv)}] For all \(n \in \mathbb{N}\) and \(y^* \in D(A^*)\) it holds that \(P \circ X^{-1}_0 = \eta\), \(P(\tau_n < \tau_\Delta) = 1\) and that \(\langle \widehat{X}_{\cdot \wedge \tau_n}, y^*\rangle\) is an \((\F^P, P)\)-semimartingale with semimartingale characteristics
		\begin{equation}\label{eq:smc}
		\begin{split}
		B(k) &= \int_0^{\cdot \wedge \tau_n} \left( \langle \widehat{X}_{s-}, A^* y^*\rangle + \langle b(\widehat{X}_{s-}), y^*\rangle \right)\dd s 
		\\&\quad\quad + \int_0^{\cdot \wedge \tau_n} \int \left(k(\langle x, y^*\rangle) - \langle h(x), y^*\rangle\right) K(\widehat{X}_{s-}, \dd x) \dd s,
		\\
		C&= \int_0^{\cdot \wedge \tau_n} \langle a(\widehat{X}_{s-}) y^*, y^*\rangle \dd s,
		\\
		\nu([0, \cdot ], G) &= \int_0^{\cdot \wedge \tau_n} \int \1_G(\langle x, y^*\rangle) K(\widehat{X}_{s-}, \dd x) \dd s,\ G \in \mathscr{B}(\mathbb{R}), 0 \not \in G,
		\end{split}
		\end{equation}
		corresponding to the truncation function \(k \colon \mathbb{R} \to \mathbb{R}\).
		\item[\textup{(v)}] For all \(n, d \in \mathbb{N}\) and \(y^*_1, ..., y^*_d \in D(A^*)\) it holds that \(P \circ X^{-1}_0 = \eta\), \(P(\tau_n < \tau_\Delta) = 1\) and \((\langle \widehat{X}_{\cdot \wedge \tau_n}, y^*_1\rangle, ..., \la \widehat{X}_{\cdot \wedge \tau_n}, y^*_d\ra)\) is an \((\F^P, P)\)-semimartingale with semimartingale characteristics
		\begin{equation*}
		\begin{split}
		B^j(k) &= \int_0^{\cdot \wedge \tau_n} \left( \langle \widehat{X}_{s-}, A^* y^*_j\rangle + \langle b(\widehat{X}_{s-}), y^*_j\rangle \right)\dd s 
		\\&\quad\quad + \int_0^{\cdot \wedge \tau_n} \int \left(k^j(\langle x, y^*_1\rangle, ..., \la x, y^*_d\ra) - \langle h(x), y^*_j\rangle\right) K(\widehat{X}_{s-}, \dd x) \dd s,
		\\
		C^{ij}&= \int_0^{\cdot \wedge \tau_n} \langle a(\widehat{X}_{s-}) y^*_i, y^*_j\rangle \dd s,
		\\
		\nu([0, \cdot ], G) &= \int_0^{\cdot \wedge \tau_n} \int \1_G(\langle x, y^*_1\rangle, ..., \la x, y^*_d\ra) K(\widehat{X}_{s-}, \dd x) \dd s, G \in \mathscr{B}(\mathbb{R}^d), 0 \not \in G,
		\end{split}
		\end{equation*}
		corresponding to the truncation function \(k \colon \mathbb{R}^d \to \mathbb{R}^d\).
	\end{enumerate}	
\end{lemma}
\begin{proof}
	We prove the following implications:
	\begin{center}
		(i) \(\Longleftrightarrow\) (v),\qquad (iv) \(\Longleftrightarrow\) (iii) \(\Longrightarrow\) (i) \(\Longrightarrow\) (ii) \(\Longrightarrow\) (iii).
	\end{center}
	Let us first establish the equivalence (i) \(\Longleftrightarrow\) (v).
	Suppose that (i) holds. Observe that Corollary \ref{coro: mart C}, together with \cite[Lemma II.67.10]{RW1}, implies that for all \(f \in \mathcal{B}\) the process \(M^f_{\cdot \wedge \tau_n}\) is an \((\F^P, P)\)-martingale.
	Thus, since \(\widehat{X}_{\cdot \wedge \tau_n}\) and \(X_{\cdot \wedge \tau_n}\) are \(P\)-indistinguishable, the process
	\begin{align}\label{eq: sm mp}
	f(\widehat{X}_{\cdot \wedge \tau_n}) - f(\widehat{X}_0) - \int_0^{\cdot \wedge \tau_n} \mathcal{K} f (\widehat{X}_{s-})\dd s
	\end{align}
	is an \((\F^P, P)\)-martingale. Finally, \cite[Theorem II.2.42]{JS} and Lemma \ref{lem:tau} imply (v).
	Suppose now that (v) holds. Since \(X_{\cdot \wedge \tau_n}\) and \(\widehat{X}_{\cdot \wedge \tau_n}\) are \(P\)-indistinguishable, \cite[Theorem II.2.42]{JS}, together with Lemma \ref{lem: mart}, implies that for all \(f \in \mathcal{C}\) the process \(M^f_{\cdot \wedge \tau_n}\) is an \((\F^P, P)\)-martingale. 
	The tower rule and the fact that \(M^f_{\cdot \wedge \tau_n}\) is \(\F\)-adapted imply that \(M^f_{\cdot \wedge \tau_n}\) is also an \((\F, P)\)-martingale. Therefore, (i) follows.

	The equivalence (iv) \(\Longleftrightarrow\) (iii) follows directly from \cite[Theorem II.2.42]{JS}.
	\\
	\noindent
	The implication (i) \(\Longrightarrow\) (ii) follows from Corollary \ref{coro: mart C}.
	
	Suppose that (ii) holds. In this case, by Lemma \ref{lem: mart}, for all \(f = g(\la X_{\cdot \wedge \tau_n}, y^*\ra)\) with \(g \in C^2_b(\mathbb{R})\) the process \(M^f_{\cdot \wedge \tau_n}\) is an \((\F, P)\)-martingale. By \cite[Lemma 67.10]{RW1}, it is also an \((\F^P, P)\)-martingale. 
	Now, since by Lemma \ref{lem:tau} the processes \(\widehat{X}_{\cdot \wedge \tau_n}\) and \(X_{\cdot \wedge \tau_n}\) are \(P\)-indistinguishable, (iii) follows from Euler's formula \(e^{ix} = \cos(x) + i \sin (x)\) and a short computation. 
	
	It remains to prove the implication (iii) \(\Longrightarrow\) (i).
	Let \(y^*_1, ..., y^*_d \in D(A^*)\) and \(\lambda_1, ..., \lambda_d \in \mathbb{R}\) for an arbitrary \(d \in \mathbb{N}\). Set \(y^* \triangleq \sum_{k = 1}^d \lambda_k y^*_k \in D(A^*)\) and \(u \triangleq 1\).
	Now, we see that \(M^*_{\cdot \wedge \tau_n}\) equals 
	\[
	\widetilde{M}^*_{\cdot \wedge \tau_n} \triangleq e^{i \sum_{k = 1}^d \lambda_k \langle \widehat{X}_{\cdot \wedge \tau_n}, y^*_k\ra} - \int_0^{\cdot \wedge \tau_n} e^{i \sum_{k = 1}^d \lambda_k \la \widehat{X}_{s-}, y^*_k\ra}\widetilde{V}^*_s \dd s,
	\]
	with  
	\begin{align*}
	\widetilde{V}^*_s &\triangleq i \sum_{k = 1}^d \lambda_k \left(\langle \widehat{X}_{s-}, A^* y^*_k\ra + \la b(\widehat{X}_{s-}), y^*_k\ra \right) 
	- \frac{1}{2} \sum_{k = 1}^d \sum_{j = 1}^d \lambda_k \lambda_j  \langle c(\widehat{X}_{s-}) y^*_k, y^*_j\ra \\&\hspace{1.4cm}+ \int \left(e^{i \sum_{k = 1}^d \lambda_k \la x, y^*_k\ra} - 1 - i \sum_{k = 1}^d \lambda_k \langle h(x), y^*_k\ra \right)K(\widehat{X}_{s-}, \dd x).
	\end{align*}
	Since we suppose (iii), the process \(\widetilde{M}^*_{\cdot \wedge \tau_n}\) is a local \((\F^P, P)\)-martingale.
	Due to \cite[Theorem III.2.42]{JS}, this implies that for each \(f \in \mathcal{C}\) the process \(\eqref{eq: sm mp}\) is a local \((\F^P, P)\)-martingale. Since \(\widehat{X}_{\cdot \wedge \tau_n}\) and \(X_{\cdot \wedge \tau_n}\) are \(P\)-indistinguishable, together with Lemma~\ref{lem: mart}, this implies that for all \(f \in \mathcal{C}\) the process \(M^f_{\cdot \wedge \tau_n}\) is an \((\F^P, P)\)-martingale. By the tower rule, \(M^f_{\cdot \wedge \tau_n}\) is also an \((\F, P)\)-martingale and the implication (iii) \(\Longrightarrow\) (i) is proven.
\end{proof}

\subsubsection{A Countable Set of Test Functions}\label{sec: count set}
Next, we show that GMPs are completely determined by a countable set of test functions.
\begin{proposition}\label{coro:tau}
	Let \(x \in \mathbb{B}_\Delta\).
	There exists a countable set \(\mathcal{D}\subseteq \mathcal{C}\) such that a probability measure \(P\) on \((\Omega, \mathscr{F})\) solves the GMP \((A, b, a, K, \varepsilon_x, \tau_{\Delta}-)\)
	if and only if \(P \circ X^{-1}_0 = \varepsilon_x\) and for all \(n \in \mathbb{N}\) and \(f \in \mathcal{D}\) the process \(M^f_{\cdot \wedge \tau_n}\) is an \((\F, P)\)-martingale. 
\end{proposition}
\begin{proof}
	In view of Remark \ref{rem: trivial}, it suffices to consider \(x \in \mathbb{B}\).
	As pointed out in the proof of \cite[Lemma 4.1]{EJP2924}, there exists a countable subset \(D\) of \(D (A^*)\) such that for each \(x^* \in D(A^*)\) there exists a sequence \((x^*_n)_{n \in \mathbb{N}} \subset D\) such that \(x^*_n\rightharpoonup x^*\) and \(A^* x^*_n\rightharpoonup A^*x^*\) as \(n \to \infty\). 
	Here, \(\rightharpoonup\) denotes weak-\(*\) convergence.
		Furthermore, we can w.l.o.g. assume that all convex combinations of elements of \(D\) with rational coefficients belong to \(D\).
	Let \(\mathcal{P}\) be a countable dense subset of \(C^2_c(\mathbb{R})\) when equipped with usual norm \(\|\cdot\|_\infty + \|\partial \cdot\|_\infty + \|\partial^2 \cdot\|_\infty\), see \cite{llavona1986approximation}.
	Define the countable set 
	\begin{align*}
	\mathcal{D} \triangleq \left\{ g(\la \cdot, y^*\ra)\colon g \in \mathcal{P}, y^* \in D\right\}.
	\end{align*}
	Clearly, \(\mathcal{D}\subseteq \mathcal{C}\).
	Hence, if \(P \in \mathcal{M}(A, b, a, K, \varepsilon_x, \tau_{\Delta}-)\), then \(M^f_{\cdot \wedge \tau_n}\) is an \((\F, P)\)-martingale for all \(f \in \mathcal{D}\), see Corollary \ref{coro: mart C}. 
	
	We fix \(n \in \mathbb{N}\).
	Let us suppose that \(P\) is a probability measure on \((\Omega, \mathscr{F})\) such that \(P \circ X^{-1}_0 = \varepsilon_x\) and for all \(f \in \mathcal{D}\) the process \(M^f_{\cdot \wedge \tau_n}\) is an \((\F, P)\)-martingale.
	We now prove that \(P\in \mathcal{M}(A, b, a, K, \varepsilon_x, \tau_n)\).
	Let \(f\in C^2_c(\mathbb{R})\) and \(y^* \in D(A^*)\).
	Fix a sequence \((y^*_k)_{k \in \mathbb{N}}\subset D\) so that \(y^*_k \rightharpoonup y^*\) and \(A^* y ^*_k \rightharpoonup A^*y^*\) as \(k \to \infty\). 
	Since \(\mathbb{B}\) is reflexive, the weak topology on \(\mathbb{B}^*\) coincides with the weak-\(*\) topology, see \cite[Theorem V.4.2]{conway2013course}. Thus, by \cite[Corollary V.1.5]{conway2013course}, the norm closure of the set of convex combinations of elements of \(\{y^*_k, k \in \mathbb{N}\}\) with rational coefficients is weak-\(*\) closed. 
	Therefore, there exists a sequence
	\begin{align*}
	x^*_k \triangleq \sum_{i = 1}^{N_k} \lambda^k_i y^*_i \in D,\quad \lambda^k_i \in \mathbb{Q}_+,\ \sum_{i = 1}^{N_k} \lambda^k_i = 1,
	\end{align*}
	such that \(x^*_k \to y^*\) as \(k \to \infty\), where \(\to\) denotes norm convergence in \(\mathbb{B}^*\).
	Let us now argue that we also have, possibly along a subsequence, \(A^* x^*_k  \rightharpoonup A^* y^*\) as \(k \to \infty\).
	First, note that \((A^* x_k^*)_{k \in \mathbb{N}}\) is bounded by the uniform boundedness principle, see \cite[Section 6.4]{aliprantis2013infinite}. Since in reflexive Banach spaces each bounded sequence has a weakly convergent subsequence, see \cite[Theorem III.3.7]{werner2007funktionalanalysis}, the sequence \((A^* x_k^*)_{k\in \mathbb{N}}\) has a weakly convergent subsequence.  Since closed operators are also weakly closed, see \cite[Problem 5.12]{kato1980}, we conclude that, possibly along a subsequence, \(A^* x^*_k\) converges weakly to \(A^* y^*\) as \(k \to \infty\).  Now, since for \(\mathbb{B}^*\) the weak topology and the weak-\(*\) topology coincide, we have, possibly along a subsequence, \(A^* x^*_k  \rightharpoonup A^* y^*\) as \(k \to \infty\). In the following, we will always refer to this subsequence.
	
	There exists a sequence \((g_{m})_{m \in \mathbb{N}} \subset \mathcal{P}\) such that \(g_{m} \to g\) as \(m \to \infty\), where the convergence is norm convergence w.r.t. \(\|\cdot\|_\infty + \|\partial \cdot\|_\infty + \|\partial^2 \cdot\|_\infty\).
	In particular, we have
	\begin{align} \label{sup dense assp}
	\sup_{m \in \mathbb{N}} \left(\|g_{m}\|_\infty + \|\partial g_{m}\|_\infty +  \|\partial^2 g_{m}\|_\infty\right) < \infty.
	\end{align}
	Now, define \(f_{m, k}\colon x \mapsto g_{m}(\langle x, x^*_k\ra)\) and \(f_m\colon x \mapsto g_m(\la x, y^*\ra)\).   We note that \(f_{m, k} \in \mathcal{D}\), which implies that \(M^{f_{m, k}}_{\cdot \wedge \tau_n}\) is an \((\F, P)\)-martingale.
	In view of \eqref{uni bound}, we deduce from our assumptions on the coefficients \(b, a, K\) together with the fact that weak-\(*\) convergent sequences in Banach spaces are bounded by the uniform boundedness principle, that \(M^{f_{m, k}}_{t \wedge \tau_n(\omega)}(\omega)\) is bounded uniformly in \(k\) and \(\omega\).
	Since \(x^*_k \to y^*\) and \(A^* x^*_k \rightharpoonup A^* y^*\) as \(k \to \infty\) imply that \begin{align}\label{eq: app 1}
	M^{f_{m, k}}_{t \wedge \tau_n} \to M^{f_m}_{t \wedge \tau_n}\end{align} as \(k \to \infty\), the dominated convergence theorem implies that \eqref{eq: app 1} holds also
	in \(L^1\).
	We conclude that \(M^{f_m}_{\cdot \wedge \tau_n}\) is an \((\F, P)\)-martingale. 
	In view of  \eqref{uni bound} and \eqref{sup dense assp}, \(M^{f_m}_{t \wedge \tau_n(\omega)} (\omega)\) is bounded uniformly in \(m\) and \(\omega\).
	Hence, using  \(g_{m} \to g\) as \(m \to \infty\) and again the dominated convergence theorem, we obtain
	\[
	M^{f_{m}}_{t \wedge \tau_{n}} \to M^{f}_{t \wedge \tau_n} 
	\]
	in \(L^1\) as \(m \to \infty\).
	Thus, also \(M^f_{\cdot \wedge \tau_n}\) is an \((\F, P)\)-martingale. 
	By Lemma \ref{lem: equi JS}  it is sufficient to consider test functions of the type \(g(\la \cdot, y^*\ra)\) for \(g \in C^2_c (\mathbb{R})\) and \(y^* \in D(A^*)\). 
	We conclude that \(P \in \mathcal{M}(A, b, a, K, \varepsilon_x, \tau_n)\).
\end{proof}

\subsection{Proof of Theorem \ref{theo:markov}}\label{sec: proof MP}
We suppose that the GMP \((A, b, a, K, \tau_{\Delta}-)\) is well-posed and denote \(P_{x} \in \mathcal{M}(A, b, a, K, \varepsilon_{x}, \tau_{\Delta}-)\) for \(x \in \mathbb{B}_\Delta\).
Let us start with an elementary observation.
\begin{lemma}\label{lem: cg}
	Let \(Y\) be a right-continuous measurable process taking values in a Polish space. Then \(\sigma(Y_t, t \in [0, \infty))\) is countably generated. 
\end{lemma}
\begin{proof}
	The claim follows from the identity \(\sigma(Y_t, t\in [0, \infty)) = \sigma (Y_t, t \in \mathbb{Q}_+)\), which is due to right-continuity of \(Y\).
\end{proof}
The following lemma is also used in the proof that well-posedness implies local well-posedness, see Section \ref{sec: loc uni} below. The idea to use Kuratowski's theorem is due to \cite[Exercise 6.7.4]{SV}.
\begin{lemma}\label{lem: kern Borel}
	The mapping \(\mathbb{B}_\Delta \ni x \mapsto P_x\) is Borel, i.e. the map \(\mathbb{B}_\Delta \ni x \mapsto P_x(G)\) is Borel for all \(G \in \mathscr{F}\).
\end{lemma}
\begin{proof}
	Denote the set of probability measures on \((D([0, \infty), \mathbb{B}_\Delta), \mathscr{B}(D([0, \infty), \mathbb{B}_\Delta)))\) by \(\hat{\mathcal{P}}\) and the set of probability measure on \((\mathbb{B}_\Delta, \mathscr{B}(\mathbb{B}_\Delta))\) by \(\mathcal{P}^*\) and equip both with the topology of convergence in distribution. It is well-known that \(\hat{\mathcal{P}}\) and \(\mathcal{P}^*\) are Polish, see \cite[Theorem 15.15]{aliprantis2013infinite}. 
	
	Recall that \(\Omega\) is a Borel subset of \(D([0, \infty), \mathbb{B}_\Delta)\).
	Due to Proposition \ref{coro:tau} and the uniqueness assumption, there exists a countable set \(\mathcal{D} \subseteq \mathcal{C}\) such that the set \(\{P_x, x \in \mathbb{B}_\Delta\}\) consists of all \(Q \in \hat{\mathcal{P}}\) such that \(Q \circ X_0^{-1} \in \{\varepsilon_x, x \in \mathbb{B}_\Delta\}\)
and
	\begin{align}\label{count}
	Q (\Omega) = 1,\quad \E^Q\left[\1_G M^f_{t \wedge \tau_n}\right] = \E^Q\left[\1_G M^f_{s \wedge \tau_n}\right],
	\end{align}
	for \(f \in \mathcal{D}, s < t, G \in \mathscr{F}_s, n \in \mathbb{N}\). 
	Since \(M^f_{\cdot \wedge \tau_n}\) has right-continuous paths, we can restrict ourselves to \(s, t \in \mathbb{Q}_+\). Moreover, since \(\mathscr{F}_s = \sigma(X_r, r \in [0, s])\) is countably generated by Lemma \ref{lem: cg}, it suffices to take only countably many \(G \in \mathscr{F}_s\) into consideration.
	Hence, by \cite[Theorem 15.13]{aliprantis2013infinite}, the subset of elements of \(\hat{\mathcal{P}}\) which satisfies \eqref{count} is Borel.
	It follows from \cite[Corollary 2.57]{aliprantis2013infinite} that \(\varepsilon_{x_k}\) converges in distribution to \(\varepsilon_x\) as \(k \to \infty\) if and only if \(x_k \to x\) as \(k \to \infty\). This implies that \(\mathbb{B}_\Delta \ni x \mapsto \varepsilon_x \in \mathcal{P}^*\) is a continuous injection. Thus, by \cite[Theorem 8.3.7]{cohn13}, the set \(\{\varepsilon_x, x \in \mathbb{B}_\Delta\}\) is a Borel subset of \(\mathcal{P}^*\) and by Kuratowski's theorem, see \cite[Proposition 8.3.5]{cohn13}, the inverse of \(\mathbb{B}_\Delta \ni x \mapsto \varepsilon_x \in \mathcal{P}^*\) is Borel.
	We note that the map \(D([0, \infty), \mathbb{B}_\Delta) \ni \omega \mapsto \omega(0) \in \mathbb{B}_\Delta\) is continuous, see \cite[VI.2.3]{JS}. Thus, by \cite[Theorem 15.14]{aliprantis2013infinite}, the map \(\hat{\mathcal{P}} \ni Q \mapsto Q \circ X_0^{-1} \in \mathcal{P}^*\) is continuous. We conclude that the set \(\{P_x, x \in \mathbb{B}_\Delta\}\) is a Borel subset of \(\hat{\mathcal{P}}\).  
	Define \(\Phi\colon \{P_x, x \in \mathbb{B}\} \to \mathbb{B}_\Delta\) such that \(\Phi(Q)\) is the (unique) \(x\in \mathbb{B}_\Delta\) such that \(Q \circ X_0^{-1} = \varepsilon_x\).
	As a composition of \(\hat{\mathcal{P}} \ni Q \mapsto Q \circ X_0^{-1} \in \mathcal{P}^*\) and the inverse of \(\mathbb{B}_\Delta \ni x \mapsto \varepsilon_x \in \mathcal{P}^*\), the map \(\Phi\) is Borel. Since \(\Phi\) is an injection from a Borel subset of a Polish space into a Polish space, Kuratowski's theorem yields that the inverse map \(\Phi^{-1}\) is Borel. This concludes our proof.
\end{proof}
Suppose that \(\xi\) is an \(\F\)-stopping time. 
It is implied by \cite[Theorem 1.6]{aries2007optimal} that \(\mathscr{F}_\xi = \sigma(X_{t \wedge \xi}, t \in [0, \infty)).\)
Hence, by Lemma \ref{lem: cg}, the \(\sigma\)-field \(\mathscr{F}_\xi\) is countably generated. Moreover, recalling Lemma \ref{lem: cond prob}, for any \(P \in \mathcal{M}(A, b, a, K, \eta, \tau_{\Delta}-)\) a regular conditional probability of \(P\) given \(\mathscr{F}_\xi\) exists.

We stress that the strong Markov property as claimed in Theorem \ref{theo:markov} is a direct consequence of \cite[Problem 2.6.9, Theorem 2.6.10]{KaraShre}, the well-posedness assumption, \(\mathscr{F}_{\tau_\Delta -}= \mathscr{F}\) and the following lemma, which can be seen as a version of \cite[Lemma 5.4.19]{KaraShre} for GMPs.

\begin{lemma}\label{lem: Mark prep}
	Let \(\xi\) be a bounded \(\F\)-stopping time and let \(P \in \mathcal{M}(A, b, a, K, \eta, \tau_{\Delta}-)\). There exists a \(P\)-null set \(N \in \mathscr{F}_\xi\) such that for all \(\omega \in \complement N\) it holds that 
	\[P\left(\theta^{-1}_\xi \cdot \big|\mathscr{F}_\xi\right) (\omega)\in \mathcal{M}(A, b, a, K, \varepsilon_{X_{\xi(\omega)}(\omega)}, \tau_{\Delta}-).\]
\end{lemma} 
\begin{proof}
	Since \(\mathscr{F}_\xi = \sigma(X_{t \wedge \xi}, t \in [0, \infty))\) is countably generated, Lemma \ref{lem: cond prob} yields the existence of a \(P\)-null set \(N^* \in \mathscr{F}_\xi\) such that for all \(\omega \in \complement N^*\)
	\begin{align*}
	P\left(\theta^{-1}_{\xi} \{X_0 = X_{\xi(\omega)}(\omega)\} \big| \mathscr{F}_\xi\right) (\omega) &=
	P\left(X_\xi = X_{\xi(\omega)}(\omega)\big|\mathscr{F}_\xi\right) (\omega) 
	\\&=  \1_{\{X_\xi = X_{\xi(\omega)}(\omega)\}} (\omega) 
	= 1.
	\end{align*}

	By \cite[Proposition 7.8]{Kallenberg}, for all \(t \in [0, \infty)\) the random time \[\sigma_t (\alpha) \triangleq t \wedge \tau_n(\theta_{\xi(\alpha)} \alpha) + \xi(\alpha)\] is an \(\F^+\)-stopping time. Take \(f \in \mathcal{C}\) and \(s < t\). We now show that \(P\)-a.s.
	\begin{align}\label{eq: stopped to show s + xi}
	E^P \left[\left(M^f_{t \wedge \tau_n} - M^f_{s \wedge \tau_n}\right) \circ \theta_{\xi}| \mathscr{F}_{s + \xi}\right] = 0.
	\end{align}
	First, note that on \(\{\xi \geq \tau_{\Delta}\}\) we have
	\begin{align*}
	\left(M^f_{t \wedge \tau_n} - M^f_{s \wedge \tau_n}\right) \circ \theta_{\xi} = 0.
	\end{align*}
	Thus, \(P\)-a.s.
	\begin{align*}
	E^P& \left[\left(M^f_{t \wedge \tau_n} - M^f_{s \wedge \tau_n}\right) \circ \theta_{\xi}| \mathscr{F}_{s + \xi}\right] \\&\hspace{2.5cm}= E^P \left[\left(M^f_{t \wedge \tau_n} - M^f_{s \wedge \tau_n}\right) \circ \theta_{\xi}\1_{\{\xi < \tau_{\Delta}\}}| \mathscr{F}_{s + \xi}\right]
	\\&\hspace{2.5cm}= \lim_{n \to \infty} E^P \left[\left(M^f_{t \wedge \tau_n} - M^f_{s \wedge \tau_n}\right) \circ \theta_{\xi}\1_{\{\xi < \tau_{n}\}}| \mathscr{F}_{s + \xi}\right]
	\\&\hspace{2.5cm}= \lim_{n \to \infty} E^P \left[\left(M^f_{\sigma_t} - M^f_{\sigma_s}\right)\1_{\{\xi < \tau_n\}}| \mathscr{F}_{s + \xi}\right].
	\end{align*}
	For \(\alpha \in \{\xi < \tau_n\}\) we have 
	\begin{align*}
	\tau^*_n (\theta_{\xi(\alpha)} \alpha) = \tau^*_n (\alpha) - \xi(\alpha)
	\end{align*}
	and hence for \(r \in \{s, t\}\)
	\begin{align*}
	\sigma_r (\alpha) = \begin{cases}
	r \wedge n + \xi (\alpha),& r \wedge n + \xi (\alpha) \leq \tau^*_n(\alpha),\\
	\tau^*_n(\alpha),&r \wedge n + \xi(\alpha) \geq \tau^*_n(\alpha).
	\end{cases}
	\end{align*}
	Therefore, \(\sigma_r = \sigma_r \wedge \tau_{2n}\) on \(\{\xi < \tau_n\}\).
	This yields that \(P\)-a.s.
	\begin{align*}
	E^P \left[\left(M^f_{\sigma_t} - M^f_{\sigma_s}\right)\1_{\{\xi < \tau_n\}}| \mathscr{F}_{s + \xi}\right] &= E^P \left[\left(M^f_{\sigma_t \wedge \tau_{2n}} - M^f_{\sigma_s \wedge \tau_{2n}}\right)\1_{\{\xi < \tau_n\}}| \mathscr{F}_{s + \xi}\right]\\&= E^P \left[M^f_{\sigma_t \wedge \tau_{2n}} - M^f_{\sigma_s \wedge \tau_{2n}}| \mathscr{F}_{s + \xi}\right]\1_{\{\xi < \tau_n\}}.
	\end{align*}
	Since \(M^f_{\cdot \wedge \tau_{2n}}\) is a uniformly integrable \((\F, P)\)-martingale, \(\sigma_t\) and \(\sigma_s\) are \(\F^+\)-stopping times and \(s + \xi\) is an \(\F\)-stopping time, the optional stopping theorem, see \cite[Theorem 1.3.22]{KaraShre}, yields that \(P\)-a.s.
	\begin{align*}
	E^P &\left[M^f_{\sigma_t \wedge \tau_{2n}} - M^f_{\sigma_s \wedge \tau_{2n}}| \mathscr{F}_{s + \xi}\right] \1_{\{\xi < \tau_n\}} \\&\hspace{2cm}= \left(M^f_{\sigma_t \wedge \tau_{2n} \wedge (s + \xi)} - M^f_{\sigma_s \wedge \tau_{2n}\wedge (s + \xi)}\right)\1_{\{\xi < \tau_n\}}
	\\&\hspace{2cm}= \left(M^f_{\sigma_s} - M^f_{\sigma_s} \right) \1_{\{\xi < \tau_n\}} = 0.
	\end{align*}
	Therefore, we conclude that \eqref{eq: stopped to show s + xi} holds.

	By \cite[Lemma 6.5]{Kallenberg}, we have \(\theta_\xi^{-1} (\mathscr{F}_s) \subseteq \mathscr{F}_{s + \xi}\). 
	Thus, for all \(G \in \mathscr{F}_s\), \(P\)-a.s.
	\begin{align*}
	E^P &\left[ \left(\left(M^f_{t \wedge \tau_n} - M^f_{s \wedge \tau_n}\right)\1_G\right)\circ \theta_\xi |\mathscr{F}_{\xi}\right] \\&= E^P \left[ E^P \left[\left(M^f_{t \wedge \tau_n} - M^f_{s \wedge \tau_n}\right)\circ \theta_\xi|\mathscr{F}_{s + \xi}\right]\1_{\theta^{-1}_\xi G} |\mathscr{F}_{\xi}\right] = 0.
	\end{align*}
Hence, we find a \(P\)-null set \(N \in \mathscr{F}_\xi\) such that for all \(\omega \in \complement N\) the equality
	\begin{align} \label{eq: p null set hold}
	E^P &\left[ \left((M^f_{t \wedge \tau_n} - M^f_{s \wedge \tau_n})\1_G\right)\circ \theta_\xi |\mathscr{F}_{\xi}\right] (\omega)
= 0
	\end{align}
	holds for all \(n\in \mathbb{N}\), all non-negative rationals \(s < t\), all \(f \in \mathcal{D}\), where \(\mathcal{D} \subseteq \mathcal{C}\) is the countable set from Proposition \ref{coro:tau}, and all \(G \in \mathscr{E}_s\), where \(\mathscr{E}_s\) is a countable determining class of \(\mathscr{F}_s\).
	
For all \(\omega \in \complement N\), the process \(M^f_{\cdot \wedge \tau_n}\) is an \((\F, P ( \theta_\xi^{-1} \cdot |\mathscr{F}_\xi) (\omega) )\)-martingale for all \(f \in \mathcal{D}\) and \(n \in \mathbb{N}\): Since \(\mathscr{E}_s\) is a determining class, the uniqueness theorem for measures yields that \eqref{eq: p null set hold} also holds for all \(G \in \mathscr{F}_s\) and, by the right-continuity of \(M^f_{\cdot \wedge \tau_{n}}\), \eqref{eq: p null set hold} also holds for all irrationals \(s <t\). Hence, by Proposition \ref{coro:tau}, we conclude the proof.
\end{proof}

Let us show that \(P_\eta\) as defined in \eqref{eq: int sol} is the unique solution to the GMP \((A, b, a, K, \eta, \tau_{\Delta}-)\).
First, note that \(P_\eta \circ X^{-1}_0 = \eta\).
Second, for all \(f \in \mathcal{C}, s < t\), and \(G \in \mathscr{F}_s\) it holds that
\[
E^{P_\eta} \left[ \1_G \left(M^f_{t \wedge \tau_n} - M^f_{s \wedge \tau_n} \right)\right] = \int E^{P_x} \left[ \1_G \left(M^f_{t \wedge \tau_n} - M^f_{s \wedge \tau_n} \right)\right] \eta(\dd x) = 0,
\]
since \(M^f_{\cdot \wedge \tau_n}\) is an \((\F, P_x)\)-martingale for all \(x \in \mathbb{B}_\Delta\), see Corollary \ref{coro: mart C}.
Therefore, we conclude that \(P_\eta \in \mathcal{M}(A, b, a, K, \eta, \tau_{\Delta}-)\).

To prove uniqueness, suppose that \(Q\) solves the GMP \((A, b, a, K, \eta, \tau_{\Delta}-)\).
Due to Lemma \ref{lem: Mark prep}, \(Q\)-a.s. \(Q(\cdot |\mathscr{F}_0)\) solves the GMP \((A, b, a, K, \varepsilon_{X_0}, \tau_{\Delta}-)\). Hence, by the well-posedness assumption, 
\(Q\)-a.s. \(Q(\cdot |\mathscr{F}_0) = P_{X_0}\). 
Finally, for all \(G \in \mathscr{F}\)
\[
Q (G)= E^Q [Q(G|\mathscr{F}_0)] = E^Q [P_{X_0} (G)] = \int P_x (G) \eta(\dd x).
\]
This proves the uniqueness.

The very last claim of Theorem \ref{theo:markov} follows by integration. More precisely, the strong Markov property of the family \(\{P_x, x \in \mathbb{B}_\Delta\}\) implies that for all \(F \in \mathscr{F}\) and \(G \in \mathscr{F}_\xi\)
\begin{align*}
E^{P_\eta} \left[\1_{G \cap \{\xi < \infty\}} P_{X_\xi} (F) \right] &= \int E^{P_x} \left[\1_{G \cap \{\xi < \infty\}} P_{X_\xi} (F) \right] \eta(\dd x) 
\\&= \int E^{P_x} \left[\1_{G \cap \{\xi < \infty\}} P_{x} \left(\theta_{\xi}^{-1} F|\mathscr{F}_\xi\right) \right] \eta(\dd x) 
\\&= \int E^{P_x} \left[\1_{G \cap \{\xi < \infty\}} \1_{\theta_{\xi}^{-1} F} \right] \eta(\dd x) 
\\&= E^{P_\eta} \left[\1_{G \cap \{\xi < \infty\}} \1_{\theta_{\xi}^{-1} F} \right],
\end{align*}
where we use the fact \(\{\xi < \infty\} \in \mathscr{F}_\xi\).
Hence, we find a \(P\)-null set \(N\in \mathscr{F}_\xi\) such that for \(\omega \in \complement N \cap \{\xi < \infty\}\) the equality \eqref{eq: markov initial arbi} holds for all \(F\) in a countable determining class of the countable \(\sigma\)-field \(\mathscr{F}= \sigma(X_{t}, t \in [0, \infty))\), see Lemma \ref{lem: cg}. A monotone class argument concludes the proof.
\qed

\subsection{Proof of Theorem \ref{theo:main1}}\label{sec: pf CMG}
The strategy of proof is based on ideas used in the finite dimensional case without explosion to study equivalence of laws of semimartingales, see \cite{J79,KLS-LACOM1, KLS-LACOM2}.
More precisely, it consists of three parts. First, we show that well-posedness is equivalent to complete local well-posedness. Second, we construct the candidate density process. Third, we apply Girsanov-type theorems to conclude the proof.
\subsubsection{Local Well-Posedness}\label{sec: loc uni}
The following definition is inspired by the concept of local uniqueness as introduced by Jacod and M\'emin \cite{JM76}.
\begin{definition}
	We say that the GMP \((A, b, a, K, \tau_\Delta -)\) is \emph{locally well-posed}, if for all \(x \in \mathbb{B}_\Delta\) there exists a \(P\in \mathcal{M}(A, b, a, K, \varepsilon_{x}, \tau_\Delta-)\) and for all \(\F\)-stopping times \(\xi\) we have 
	\(
	P = Q \textit{ on } \mathscr{F}_\xi
	\)
	whenever \(Q \in \mathcal{M}(A, b, a, K, \varepsilon_{x}, \xi)\).
	Moreover, if for all probability measures \(\eta\) on \((\mathbb{B}_\Delta, \mathscr{B}(\mathbb{B}_\Delta))\) there exists a \(P \in \mathcal{M}(A, b, a, K, \eta, \tau_\Delta-)\) and for all \(\F\)-stopping times \(\xi\) the equality \(P = Q\) holds on \(\mathscr{F}_\xi\) whenever \(Q \in \mathcal{M}(A, b, a, K, \eta, \xi)\), then we call the GMP \((A, b, a, K, \tau_\Delta-)\) \emph{completely locally well-posed}.
\end{definition}
We now show that local well-posedness is equivalent to well-posedness.
\begin{proposition}\label{prop:locuni}
	For a GMP \((A, b, a, K, \tau_{\Delta}-)\) the following are equivalent:
	\begin{enumerate}
		\item[\textup{(i)}] It is locally well-posed.
		\item[\textup{(ii)}] It is completely locally well-posed.
		\item[\textup{(iii)}] It is well-posed.
		\item[\textup{(iv)}] It is completely well-posed.
	\end{enumerate}
\end{proposition}
\begin{proof}
	We prove the following implications:
	\begin{center}
	(iii) \(\Longleftarrow\) (iv) \(\Longleftarrow\) (iii) \(\Longleftarrow\) (i) \(\Longleftarrow\) (ii) \(\Longleftarrow\) (i) \(\Longleftarrow\) (iii). 
	\end{center}
	First of all, note that (iv) \(\Longrightarrow\) (iii) is trivial. 
	
	The implication (iii) \(\Longrightarrow\) (iv) is implied by Theorem \ref{theo:markov}.\\
	\noindent
	Now, suppose that (i) holds. 
	Recalling Remark \ref{rem: trivial}, to prove (iii) it suffices to prove that the GMP \((A, b, a, K, \varepsilon_x, \tau_\Delta-)\) has a unique solution for all \(x \in \mathbb{B}\).
	By the definition of local well-posedness, for all \(x\in \mathbb{B}\), the GMP \((A, b, a, K, \varepsilon_{x}, \tau_{\Delta}-)\) has a solution.
	Suppose that \(P, Q \in \mathcal{M}(A, b, a, K, \varepsilon_{x}, \tau_\Delta-)\) for \(x \in \mathbb{B}\).
	By the definition of the GMP and local well-posedness, for all \(n \in \mathbb{N}\) we have
	\(
	P = Q \textup{ on } \mathscr{F}_{\tau_n}.
	\)
	Hence, \(P = Q\) on the \(\pi\)-system \(\bigcup_{n \in \mathbb{N}} \mathscr{F}_{\tau_n-}\). By a monotone class argument, \(P = Q\) on \(\bigvee_{n \in \mathbb{N}} \mathscr{F}_{\tau_n-} = \mathscr{F}_{\tau_\Delta-} = \mathscr{F}\). We have shown (i) \(\Longrightarrow\) (iii). 
	
	The implication (ii) \(\Longrightarrow\) (i) is trivial.\\
	\noindent
	Suppose now that (i) holds. Denote \(P_x \in \mathcal{M}(A, b, a, K, \varepsilon_x, \tau_{\Delta}-)\) for \(x \in \mathbb{B}_\Delta\).
	Let \(\eta\) be a probability measure on \((\mathbb{B}_\Delta, \mathscr{B}(\mathbb{B}_\Delta))\). First, Theorem \ref{theo:markov} yields that the GMP \((A, b, a, K, \eta, \tau_{\Delta}-)\) has the unique solution \(P_\eta = \int P_x \eta(\dd x)\). 
	Let \(\xi\) be an arbitrary \(\F\)-stopping time and \(Q_\eta \in \mathcal{M}(A, b, a, K, \eta, \xi)\).
	Due to similar arguments as used in the proof of Lemma \ref{lem: Mark prep}, the regular conditional probability \(Q_\eta(\cdot |\mathscr{F}_0)\) solves \(Q_\eta\)-a.s. the GMP \((A, b, a, K, \varepsilon_{X_0}, \xi)\). Hence, since we assume (i), \(Q_\eta\)-a.s. 
	\(
	Q_\eta(\cdot |\mathscr{F}_0) = P_{X_0}\) on \(\mathscr{F}_\xi.
	\)
	Now, we obtain that for all \(G \in \mathscr{F}_\xi\)
	\[
	Q_\eta(G) = E^{Q_\eta} \left[ Q_\eta (G|\mathscr{F}_0)\right] = E^{Q_\eta} \left[P_{X_0} (G)  \right] = \int P_x (G) \eta(\dd x) = P_\eta(G).
	\]
	We conclude that (ii) holds.
	
	Let us now turn to last implication and suppose that (iii) holds. In the following, we adapt ideas used in the proof of \cite[Theorem III.2.40]{JS} in a finite-dimensional setting without explosion.
	Again, we denote \(P_x \in \mathcal{M}(A, b, a, K, \varepsilon_x, \tau_{\Delta}-)\) for \(x \in \mathbb{B}_\Delta\).
	Following Dellacherie and Meyer \cite{DellacherieMeyer78}, we denote for two \(\omega, \omega' \in \Omega\) and \(t \in [0, \infty)\), by \(\omega/t/\omega'\) the element in \(\Omega\) defined by 
	\begin{align*}
	&\textup{ if } t \leq \tau_\Delta(\omega),\ X_s(\omega/t/\omega') = \begin{cases}
	X_s(\omega),&s <t,\\
	X_{s - t}(\omega'),& s \geq t,
	\end{cases}
	\\
	&\textup{ if } t > \tau_\Delta(\omega),\ \omega/t/\omega' = \omega.
	\end{align*}
	Now, let \(\alpha, \beta\) be two \(\F\)-stopping times and define \(V\colon \Omega \times\Omega \to [0, \infty]\) by 
	\begin{align}\label{V}
	V(\omega, \omega') = \begin{cases} (\alpha \vee \beta) (\omega/\alpha(\omega)/\omega') - \alpha(\omega/\alpha(\omega)/\omega'),& \alpha(\omega) < \infty, \omega'_0 = \omega_{\alpha(\omega)},\\
	0,&\textup{otherwise}.
	\end{cases}
	\end{align}
	The following lemma restates a special case of Courr\`ege and Priouret's \cite{zbMATH03218236} decomposition theorem for stopping times as given by Dellacherie and Meyer \cite[Theorem IV.103]{DellacherieMeyer78}. 
	\begin{lemma}\label{JS lemma}
		The map \(V\) is \(\mathscr{F}_\alpha \otimes \mathscr{F}\)-measurable such that \(V(\omega, \cdot)\) is an \(\F\)-stopping time for any \(\omega \in \Omega\), and \(\alpha(\omega) \vee \beta(\omega) = \alpha(\omega) + V(\omega, \theta_{\alpha(\omega)}\omega)\) for all \(\omega \in \{\alpha < \infty\}\).
	\end{lemma}
	Let \(\xi\) be an \(\F\)-stopping time and suppose that \(Q_{x} \in \mathcal{M}(A, b, a, K, \varepsilon_{x}, \xi)\). 
	Recalling Lemma \ref{lem: kern Borel}, we set for \(G \in \mathscr{F}\)
	\begin{align*}
	\mathsf{Q}_{x} (G) \triangleq Q_x(G \cap \{\xi = \infty\}) + \iint \1_{\{\xi(\omega) < \infty\}}\1_{G'}(\omega, \omega')P_{X_{\xi(\omega)}(\omega)} (\dd \omega')Q_{x}(\dd \omega),
	\end{align*}
	where \(G' \in \mathscr{F}_\xi \otimes \mathscr{F}\) is so that \[G \cap \{\xi < \infty\} 
	= \{\omega \in \Omega \colon \xi(\omega) < \infty, (\omega, \theta_{\xi (\omega)}\omega) \in G'\}.\]
	We stress that we can find \(G'\) since the trace of \(\mathscr{F}\) on \(\{\xi < \infty\}\) equals the trace of \(\mathscr{F}_\xi \vee \theta_{\xi}^{-1} (\mathscr{F})\) on \(\{\xi < \infty\}\), see \cite[Lemma III.2.44]{JS}.
	\begin{lemma}
		\(\mathsf{Q}_{x}\) is a probability measure on \((\Omega, \mathscr{F})\).
	\end{lemma}
	\begin{proof}
		Clearly, \[
		\hat{\mathsf{Q}}_{x}(\cdot) \triangleq \iint \1_{\{\xi(\omega) < \infty\}}\1_{\cdot}(\omega, \omega')P_{X_{\xi(\omega)}(\omega)} (\dd \omega')Q_{x}(\dd \omega) 
		\]
		defines a measure on \((\Omega \times \Omega, \mathscr{F}_\xi \otimes \mathscr{F})\). The total mass is given by
		\(
		\hat{\mathsf{Q}}_{x} (\Omega \times \Omega) =Q_{x} (\xi < \infty).
		\)
		We have to show that \(\mathsf{Q}_x(G)\) is defined unambiguously for all \(G \in \mathscr{F}\). 
		For \(i = 1, 2\), let \(G^i \in \mathscr{F}_\xi \otimes \mathscr{F}\) be such that 
		\begin{align*}
		G \cap \{\xi < \infty\} = \{\omega \in \Omega \colon \xi (\omega) < \infty, (\omega, \theta_{\xi(\omega)}\omega) \in G^i\}.
		\end{align*}
		By symmetry, it suffices to show that 
		\begin{align*}
		\hat{\mathsf{Q}}_x (G^1 \cup G^2) = \hat{\mathsf{Q}}_x (G^1), 
		\end{align*}
		which is equivalent to 
		\begin{align*}
		\hat{\mathsf{Q}}_x (G^2 \backslash G^1) = 0.
		\end{align*}
		Thus, since it holds that
		\begin{align*}
	 \{\omega \in \Omega& \colon \xi (\omega) < \infty, (\omega, \theta_{\xi(\omega)}\omega) \in G^2 \backslash G^1\} \\&= G \cap \{\xi < \infty\} \cap
	 \{\omega \in \Omega \colon (\omega, \theta_{\xi(\omega)}\omega) \in \complement G^1\} 
	 \\&= \{\xi < \infty\} \cap \{\omega \in \Omega \colon (\omega, \theta_{\xi(\omega)}\omega) \in G^1\}  \cap \{\omega \in \Omega \colon (\omega, \theta_{\xi(\omega)}\omega) \in \complement G^1\}  \\&= \emptyset,
		\end{align*}
		 it suffices to prove \(\hat{\mathsf{Q}}_{x}(G') = 0\) whenever \(G \triangleq \emptyset\).
		Suppose that \(\hat{\mathsf{Q}}_{x} (G') > 0\). Then, by the definition of \(\hat{\mathsf{Q}}_x\), there exists \((\omega, \omega') \in G'\) such that \(\xi(\omega) < \infty\) and \(\omega_{\xi(\omega)} = \omega'_0\). 
		We have for all \(t \in [0, \xi(\omega)]\) that
		\(
		X_t (\omega/\xi(\omega)/\omega') = X_t (\omega).
		\)
		Hence, by Galmarino's test, see \cite[Theorem IV.100]{DellacherieMeyer78}, we have 
		\(
		\xi (\omega/ \xi(\omega)/\omega') = \xi(\omega).
		\)
		Since \(\{\omega^* \in \Omega \colon (\omega^*, \omega') \in G'\} \in \mathscr{F}_\xi\), 
		we deduce again from Galmarino's test that \((\omega/\xi(\omega)/\omega', \omega') \in G'\).
		Moreover, by the definition of \(\omega/\xi(\omega)/\omega'\) and \(\omega_{\xi(\omega)} = \omega'_0\) we obtain
		\[
		\theta_{\xi(\omega/\xi(\omega)/\omega')}(\omega/\xi(\omega)/\omega') = \theta_{\xi(\omega)}(\omega/\xi(\omega)/\omega') = \omega'.
		\]
		Hence, we conclude \((\omega/\xi(\omega)/\omega', \theta_{\xi(\omega/\xi(\omega)/\omega')} (\omega/\xi(\omega)/\omega')) \in G'\). This, however, implies the contradiction \(\omega/\xi(\omega)/\omega' \in \emptyset\). Therefore, \(\hat{\mathsf{Q}}_{x} (G') = 0\).
	\end{proof}
	
	We claim that \(\mathsf{Q}_{x} \in \mathcal{M}(A, b, a, K, \varepsilon_{x}, \tau_\Delta-)\). To see this, note first the following
	\begin{lemma}
		For all \(n \in \mathbb{N}\) and \(f \in \mathcal{C}\) the process \(M^f_{\cdot \wedge \tau_n}\) is an \((\F, \mathsf{Q}_{x})\)-martingale.
	\end{lemma}
	\begin{proof}
		Let \(n \in \mathbb{N}\) and \(f \in \mathcal{C}\).
		We show that \(\E^{\mathsf{Q}_{x}}[M^f_{\gamma \wedge\tau_n}] = 0\) for all bounded \(\F\)-stopping times \(\gamma\), which is equivalent to our claim, see \cite[Proposition II.1.4]{RY}.
		Due to Corollary \ref{coro: mart C}, see Remark \ref{rem: trivial}, for all \(y \in \mathbb{B}_\Delta\), the process \(M^f_{\cdot \wedge \tau_n}\) is an \((\F, P_y)\)-martingale.
		Moreover, by Lemma \ref{lem: mart}, the process \(M^f_{\cdot \wedge \tau_n \wedge \xi}\) is an \((\F, Q_{x})\)-martingale.  
		Let \(V\) be defined as in \eqref{V} with \(\alpha\) replaced by \(\xi\) and \(\beta\) replaced by \(\gamma\), and let \(V^n\) defined as \(V\) in \eqref{V} with \(\alpha\) replaced by \(\xi\) and \(\beta\) replaced by \(\tau_n\).
		We compute 
		\begin{align*}
		\E^{\mathsf{Q}_{x}}\left[M^f_{\gamma \wedge \tau_n}\right] &= \E^{\mathsf{Q}_{x}}\left[M^f_{\gamma \wedge \tau_n \wedge \xi} + M^f_{\gamma \wedge \tau_n} - M^f_{\gamma \wedge \tau_n \wedge \xi}\right] 
		\\&= \E^{Q_{x}}\left[M^f_{\gamma \wedge\tau_n \wedge \xi}\right] + \E^{\mathsf{Q}_{x}}\left[\left(M^f_{\gamma \wedge\tau_n} - M^f_{\xi}\right)\1_{\{\gamma \wedge \tau_n> \xi\}} \right] 
		\\&=  \E^{\mathsf{Q}_{x}}\left[\left(M^f_{(\gamma \vee \xi) \wedge (\tau_n\vee \xi)} - M^f_{\xi}\right)\1_{\{\gamma \wedge\tau_n> \xi\}} \right] 
		\\&= \E^{\mathsf{Q}_{x}}\left[\left(\left[M^f \circ \theta_\xi\right]_{V(\cdot, \theta_\xi) \wedge V^n(\cdot, \theta_\xi)} \right)\1_{\{\gamma \wedge \tau_n> \xi\}} \right] 
		\\&= \int  \1_{\{\gamma(\omega) \wedge \tau_n(\omega) > \xi(\omega)\}} \E^{P_{X_{\xi(\omega)}(\omega)}}\left[M^f_{V(\omega, \cdot) \wedge V^n(\omega, \cdot)}\right]Q_{x}(\dd \omega).
		\end{align*}
		Note that for \(\omega \in \Omega\) with \(\tau_n(\omega) > \xi(\omega)\), and for \(P_{X_{\xi(\omega)}(\omega)}\)-a.a. \(\omega' \in \Omega\)
		it holds that \begin{align*}
		\tau^*_n(\omega/\xi(\omega)/\omega') &= \tau^*_n(\omega') + \xi(\omega),\\
\xi(\omega /\xi(\omega)/\omega') &= \xi(\omega),\end{align*}
		which implies
		\begin{align*}
		V^n(\omega, \omega') &= \begin{cases}
		n - \xi (\omega),&n - \xi(\omega) \leq \tau^*_n(\omega'),\\
		\tau^*_n(\omega'),&\tau^*_n(\omega') + \xi(\omega) \leq n.
		\end{cases}
		\end{align*}
		For these \(\omega\) and \(\omega'\) we conclude that \(V^n(\omega, \omega') = V^n(\omega, \omega') \wedge \tau_n(\omega')\).
		Hence, for \(\omega \in \Omega\) with \(\gamma(\omega)\wedge \tau_n(\omega) > \xi(\omega)\) it holds that
		\begin{align*}
		\E^{P_{X_{\xi(\omega)}(\omega)}} \left[ M^f_{V(\omega, \cdot)\wedge V^n(\omega, \cdot)}\right] =
		\E^{P_{X_{\xi(\omega)}(\omega)}} \left[ M^f_{V(\omega, \cdot) \wedge V^n(\omega, \cdot) \wedge \tau_n}\right] = 0.
		\end{align*}
		Thus, we conclude the proof.
	\end{proof}
Since we also have
	\(
	\mathsf{Q}_{x}(X_0 = x) = Q_{x}(X_0 = x) = 1,
	\)
	it follows that \(\mathsf{Q}_{x}\in \mathcal{M}(A, b, a, K, \varepsilon_{x}, \tau_\Delta-)\). The assumption of uniqueness yields \(\mathsf{Q}_{x} = P_{x}\) on \(\mathscr{F}_{\tau_\Delta-} = \mathscr{F}\). 
	Since for \(G \in \mathscr{F}_\xi\) we may choose \(G'\) to be \(G \times \Omega\), for all \(G \in \mathscr{F}_\xi\) it holds that
	\(
	Q_{x}(G) = \mathsf{Q}_{x}(G) = P_{x}(G).
	\)
	This implies (i), and the proposition is proven.
\end{proof}
\subsubsection{A Candidate Density Process}\label{sec: cand density}
We introduce the random measure \(\mu^{X}\) on \(([0, \infty) \times \mathbb{B}_\Delta, \mathscr{B}([0, \infty)) \otimes \mathscr{B}(\mathbb{B}_\Delta))\) by 
\begin{align}\label{mu X}
\mu^{X} (\omega, \dd t, \dd x) \triangleq \sum_{s \in [0, \infty)} \1_{\{\Delta X_s(\omega) \not =0\}}  \varepsilon_{(s, \Delta X_s(\omega))}(\dd t, \dd x),\ \omega \in \Omega.
\end{align}
Let \(P\) be a probability measure on \((\Omega, \mathscr{F})\).
\begin{lemma}\label{lem: sigma finite 1}
	The random measure \(\mu^X\) is \(\F^P\)-optional and the associated Dol\'eans measure \(M^P_{\mu^X}\) is \(\mathscr{P}^P\)-\(\sigma\)-finite.
\end{lemma}
\begin{proof}
	Thanks to \cite[Theorem IV.88B, Remark below]{DellacherieMeyer78}, since \(\mathbb{B}_\Delta\) is Polish, the set \(\{\Delta X \not = 0\}\) is \(\F^P\)-thin. Hence, \(\mu^X\) is of the form \eqref{eq: rmj} and \cite[II.1.15]{JS} yields that \(\mu^X\) is \(\F^P\)-optional.
	
	It remains to show that \(M^P_{\mu^X}\) is \(\mathscr{P}^P\)-\(\sigma\)-finite. This follows as in \cite[Example 2, pp. 160]{liptser1989theory}.
	More precisely, for \(n, p \in \mathbb{N}\) set 
	\[	\gamma (n, p) \triangleq \inf( t \in [0, \infty) \colon \mu^X([0, t] \times \{x \in \mathbb{B}_\Delta \colon \|x\| \geq p^{-1}\}) > n).
	\]	
	The random time \(\gamma(n, p)\) is an \(\F^P\)-stopping time. 
	Hence, the set \(\of 0, \gamma(n, p)\gs\) is \(\mathscr{P}^P\)-measurable and the set 
	\[
	\Omega_{n, p} \triangleq \of 0, \gamma(n, p)\gs \times \{z \in \mathbb{B}_\Delta \colon \|z\| \geq p^{-1}\}
	\]
	is \(\mathscr{P}^P \otimes \mathscr{B}(\mathbb{B}_\Delta)\)-measurable. Clearly, we have \begin{align*}
	\Omega_{n, p} &\nearrow \Omega \times [0, \infty) \times \{z \in \mathbb{B}_\Delta \colon \|z\| \geq p^{-1}\}, \quad n \to \infty 
	\\&\nearrow \Omega \times [0, \infty) \times \mathbb{B}_\Delta,\quad p \to \infty,
	\end{align*} and 
	\(
	\1_{\Omega_{n, p}} \star \mu^X_\infty \leq n.
	\)
	Now, we find a sequence \((n_k)_{k \in \mathbb{N}} \subset \mathbb{N}\) such that \(\bigcup_{p \in \mathbb{N}} \Omega_{n_p, p} = \Omega \times [0, \infty) \times E\).
	Hence, \(M^P_{\mu^X}\) is \(\mathscr{P}^P\)-\(\sigma\)-finite and the proof is completed. 
\end{proof}

For the rest of this section we suppose that \(P \in \mathcal{M}(A, b, a, K, \eta, \tau_{\Delta}-)\) and \(\eta(\mathbb{B}) = 1\).
In this case we can explicitly determine a \(P\)-version of the \(\F^P\)-predictable compensator of \(\mu^X\).
We define the random measure \(\nu^X\) on \(([0, \infty) \times \mathbb{B}_\Delta, \mathscr{B}([0, \infty)) \otimes \mathscr{B}(\mathbb{B}_\Delta))\) by 
\[
\nu^X (\omega, \dd t, \dd x) \triangleq 
\1_\mathbb{B} (x) K(X_{t-}(\omega), \dd x) \dd t, 
\]
with the convention that \(K(\Delta, \cdot) \triangleq 0\).
\begin{lemma}\label{lem: nu comp}
	The random measure \(\nu^X\) is \(\F^P\)-predictable and the associated Dol\'eans measure \(M^P_{\nu^X}\) is \(\mathscr{P}^P\)-\(\sigma\)-finite.
\end{lemma}
\begin{proof}
	Let \(W\) be a non-negative \(\mathscr{P}^P \otimes \mathscr{B}(\mathbb{B}_\Delta)\)-measurable function \(W\). It is well-known that there exists an increasing sequence \((W^k)_{k \in \mathbb{N}}\) of simple non-negative \(\mathscr{P}^P \otimes \mathscr{B}(\mathbb{B}_\Delta)\)-measurable functions such that \(\lim_{k \to \infty} W^k(\omega, t, x) = W(\omega, t, x)\) for all \((\omega, t, x) \in \Omega \times [0, \infty) \times \mathbb{B}_\Delta\). Set \(G_m \triangleq \{x \in \mathbb{B}_\Delta \colon \|x\| \geq m^{-1}\}\) for \(m \in \mathbb{N}\).
	Note that \(\1_{\gs \tau_{\Delta}, \infty\of} \star \nu^X_\infty = 0\) by the definition of \(\Delta\) as an absorbing state. 
	Recall \eqref{eq: gamma}. It holds that \begin{align}\label{eq: decomp Gamma}
	\bigcup_{n \in \mathbb{N}}\of 0, \tau_n\gs = \left(\of 0, \tau_{\Delta}\of \ \cap\ \Gamma \times [0, \infty)\right) \cup \left( \of 0, \tau_{\Delta}\gs\ \cap\ \complement \Gamma \times [0, \infty)\right).
	\end{align}
	Since \(X_{\tau_{\Delta}-} = \Delta\) on \(\Gamma \cap \{\tau_{\Delta} < \infty\}\), 
	it follows from the convention~\(K(\Delta, \cdot) = 0\) that \(\1_{\of \tau_{\Delta}\gs} \1_\Gamma \star \nu^X_\infty = 0\). Hence,  
	using the monotone convergence theorem, we obtain
	\[
	W \star \nu^X_t (\omega) = \lim_{k \to \infty} \lim_{n \to \infty} \lim_{m \to \infty} W^k\1_{\of 0, \tau_n\gs} \1_{G_m}\star \nu^X_t (\omega)
	\]
	for all \((\omega, t) \in \Omega \times [0, \infty)\).
	Due to \eqref{eq: int ball around origin}, we have
	\begin{align*}
	\int W^k (\cdot, \cdot, x)& \1_{\of 0, \tau_n\gs}  \1_{G_m \cap \mathbb{B}} (x) K(X_{-}, \dd x) \\&\leq \textup{const. } \1_{[0, n]} \sup_{\|z\| \leq n}K(z, G_m \cap \mathbb{B} ) < \infty,
	\end{align*}
	which implies that \(W^k \1_{\of 0, \tau_n\gs} \1_{G_m} \star \nu^X\) is a real-valued continuous \(\F^P\)-adapted process.
	Thus, it is \(\F^P\)-predictable and, as a pointwise limit, so is the process \(W \star \nu^X\). In other words, we have shown that the random measure \(\nu^X\) is \(\F^P\)-predictable. 
	
	Define 
	\[
	\Omega_{n, m} \triangleq \of 0, \tau_n\gs \cup \of \tau_\Delta, \infty\of \ \times \{z \in \mathbb{B}_\Delta\colon \|z\| \geq m^{-1}\},
	\]
	which is a \(\mathscr{P}^P\otimes \mathscr{B}(\mathbb{B}_\Delta)\)-measurable set, see Remark \ref{rem: exp time pred} and \cite[Proposition I.2.12]{JS}.
	We have 
	\[
	\1_{\Omega_{n, m}} \star \nu^X_\infty 
	\leq n \sup_{\|x\| \leq n}K(x, \{z \in \mathbb{B}\colon \|z\| \geq m^{-1}\}) < \infty,
	\]
	see \eqref{eq: int ball around origin}.
	Now, since \begin{align*}
	\Omega_{n, m} &\nearrow \Omega \times [0, \infty) \times \{z \in \mathbb{B}_\Delta \colon \|z\| \geq m^{-1}\},\quad n \to \infty\\ &\nearrow \Omega \times [0, \infty) \times \mathbb{B}_\Delta,\quad m \to \infty,\end{align*}
	we find a sequence \((n_k)_{k \in \mathbb{N}} \subset \mathbb{N}\) such that \(\bigcup_{m \in \mathbb{N}} \Omega_{n_m, m} = \Omega \times [0, \infty) \times E\).
	Thus, \(M^P_{\nu^X}\) is \(\mathscr{P}^P\)-\(\sigma\)-finite.
\end{proof}
\begin{lemma}\label{lem: comp}
	The random measure \(\nu^X\) is the \(\F^P\)-predictable \(P\)-compensator of~\(\mu^X\).
\end{lemma}
\begin{proof}
	As already pointed out before, there exists an \(\F^P\)-predictable \(P\)-compensator \(\nu\) of \(\mu^X\) which is unique up to a \(P\)-null set. We now show that \(\nu = \nu^X\) up to a \(P\)-null set. 
	
	Note that
		\begin{align*}
	M^P_\nu (\of 0, \infty\of \times \{\Delta\}) = M^P_{\mu^X} (\of 0, \infty\of \times \{\Delta\}) &= 0,\\
	M^P_{\nu^X} (\of 0, \infty\of \times \{\Delta\}) &= 0,\end{align*}
	see Lemma \ref{lem:tau}. Hence, since \(\mathscr{B}(\mathbb{B}_\Delta)\cap \mathbb{B} = \mathscr{B}(\mathbb{B})\),
	by a monotone class argument, it suffices to prove \(M^P_{\nu} = M^P_{\nu^X}\) on \(\mathcal{Z}_1 \times \mathcal{Z}_2\triangleq \{Z_1 \times Z_2\colon Z_1 \in \mathcal{Z}_1, Z_2 \in \mathcal{Z}_2\}\), where \(\mathcal{Z}_1\) is an intersection stable generator of \(\mathscr{P}^P\) which includes a sequence \((Z^1_n)_{n \in \mathbb{N}}\) such that \(\bigcup_{n \in \mathbb{N}} Z^1_n = \Omega \times [0, \infty)\), and \(\mathcal{Z}_1\) is an intersection stable generator of \(\mathscr{B}(\mathbb{B})\) which includes a sequence \((Z^2_n)_{n \in \mathbb{N}}\) such that \(\bigcup_{n \in \mathbb{N}} Z^2_n = \mathbb{B}\), and \(\mathcal{Z}_1 \times \mathcal{Z}_2\) includes a sequence \((Z_n)_{n \in \mathbb{N}}\) with \(M^P_{\nu}(Z_n) = M^P_{\nu^X}(Z_n) < \infty\) for all \(n \in \mathbb{N}\) and \(\bigcup_{n \in \mathbb{N}} Z_n = \Omega \times [0, \infty) \times \mathbb{B}\).
	 
	 Let \(\mathcal{Z}_1\) be the collection of sets \(A \times \{0\}\) for \(A \in \mathscr{F}^P_0\) and \(\of 0, \xi\gs\) for all \(\F^P\)-stopping times \(\xi\) and let \(\mathcal{Z}_2\) be the collection of all sets
	\begin{align}\label{eq:cs G}
	G \triangleq \{x \in \mathbb{B} \colon (\langle x, y^*_1\rangle, ..., \la x, y^*_d \ra) \in A\} \in \mathscr{B}(\mathbb{B}),
	\end{align}
	for \(A \in \mathscr{B}(\mathbb{R}^d), y^*_1, ..., y^*_d \in D(A^*)\) and \(d \in \mathbb{N}\). Recalling the proofs of the Lemmata \ref{lem: sigma finite 1} and \ref{lem: nu comp} and in view of \cite[Theorem I.2.2]{JS} and \cite[Proposition 1.1.1]{hytönen2016analysis}, we note that \(\mathcal{Z}_1 \times \mathcal{Z}_2\) has all necessary properties.
 
 Note that \(M^P_{\nu} (A \times \{0\} \times G) = M^P_{\nu^X} (A \times \{0\} \times G) = 0\) for all \(A \in \mathscr{F}^P_0\) and \(G \in \mathscr{B}(\mathbb{B})\), see \cite[Definition II.1.3]{JS}.
	Fix an \(\F^P\)-stopping time \(\xi\) and the cylindrical set \(G\) given by \eqref{eq:cs G}.
	Denote \(Y^n \triangleq (\langle \widehat{X}_{\cdot \wedge \tau_n}, y^*_1\ra, ..., \la \widehat{X}_{\cdot \wedge \tau_n}, y^*_d\ra).\)
	Recall \eqref{eq: decomp Gamma} and the fact that \(P\)-a.s. \(\1_{\of \tau_{\Delta}, \infty\of} \star \mu^X_\infty = 0\) by the definition of \(\Delta\) as absorbing state and Lemma \ref{lem:tau}.
	Hence, thanks to Lemma \ref{lem: equi JS} and the monotone convergence theorem, we have 
	\begin{align*}
	\E^P\bigg[ \1_{\of 0, \xi\gs \times G} \star \mu^X_\infty\bigg] &= \lim_{n \to \infty} \E^P\bigg[ \1_{\of 0, \xi \wedge \tau_n\gs \times G} \star \mu^X_\infty\bigg]
	\\&= \lim_{n \to \infty} \E^P\bigg[ \1_{\of 0, \xi \wedge \tau_n\gs \times A} \star \mu^{Y^n}_\infty\bigg]
	\\&= \lim_{n \to \infty} E^P \bigg[\1_{\of 0, \xi \wedge \tau_n\gs \times G} \star \nu^X_\infty\bigg]
	\\&= \E^P \bigg[\1_{\of 0, \xi\gs \times G} \star \nu^X_\infty\bigg].
	\end{align*}
	We conclude that \(M^P_{\nu^X} = M^P_\nu\) on \(\mathcal{Z}_1 \times \mathcal{Z}_2\).
Thus, the claim follows. 
\end{proof}

As a consequence of Lemma \ref{lem: equi JS}, for all \(y^* \in D(A^*)\), the process \(\la \widehat{X}_{\cdot \wedge \tau_n}, y^*\ra\) is an \((\F^P, P)\)-semimartingale.
Hence, we obtain the existence of a family of continuous local \((\F^P, P)\)-martingales \(\{\widehat{X}^{n, c}( y^*), y^* \in D(A^*)\}\) such that \(\widehat{X}^{n, c}(y^*)\) is the continuous local \((\F^P, P)\)-martingale part of \(\langle \widehat{X}_{\cdot \wedge \tau_n}, y^*\ra\) whenever \(y^* \in D(A^*)\). Moreover, \(D(A^*) \ni y^* \mapsto \widehat{X}^{n, c}(y^*)\) is linear by the uniqueness of the continuous local martingale part.
This family can be extended to \(\mathbb{B}^*\) as the following lemma shows.
\begin{lemma}\label{lem: cy cont mart}
	For all \(n \in \mathbb{N}\) there exists a unique cylindrical continuous local \((\F^P, P)\)-martingale \(\widehat{X}^{n, c} = \{\widehat{X}^{n, c}(y^*), y^* \in \mathbb{B}^*\}\) such that for all \(y^* \in D(A^*)\) the process \(\widehat{X}^{n, c}(y^*)\) is the continuous local \((\F^P, P)\)-martingale part of \(\la \widehat{X}_{\cdot \wedge \tau_n}, y^*\ra\), \(\mathbb{B}^* \ni y^* \mapsto \widehat{X}^c(y^*)\) is linear and 
	\[
	\lle \widehat{X}^{n, c}(y^*) \rre = \int_0^{\cdot \wedge \tau_n} \la a(\widehat{X}_{s-}) y^*, y^*\ra \dd s,\quad y^* \in \mathbb{B}^*.
	\]
\end{lemma}
\begin{proof}
	Since \(\mathbb{B}\) is assumed to be reflexive, \(D(A^*)\) is weak-\(*\) dense in \(\mathbb{B}^*\).  
	Thus, by \cite[Corollary 5.108]{aliprantis2013infinite}, \(D(A^*)\) separates points of \(\mathbb{B}^*\)
	and the claim follows from \cite[Corollary 29]{Ondrejat2005}.
\end{proof}

Finally, we are in the position to define our candidate density process. We denote by \(\mathcal{E}(Y)\) the stochastic exponential of a semimartingale \(Y\), see \cite[Section I.4.d]{JS}.
\begin{lemma}\label{lem: candidate density}
	Let \((\widehat{X}^{n, c})_{n \in \mathbb{N}}\) be the sequence of cylindrical continuous local \((\F^P, P)\)-martingales as given in Lemma \ref{lem: cy cont mart}.
	There exists a non-negative local \((\F^P, P)\)-martingale \(Z^*\) such that for all \(n \in \mathbb{N}\)
	\begin{align}\label{eq: Z}
	Z^*_{\cdot \wedge \tau_n} = Z^n \triangleq \mathcal{E} \left( \int_0^{\cdot \wedge \tau_n} \langle  \dd \widehat{X}^{n, c}_s, c(\widehat{X}_{s-})\rangle + (Y - 1) \star (\mu^X- \nu^X)_{\cdot \wedge \tau_n}\right).
	\end{align}
	The process \(Z^n\) is a positive uniformly integrable \((\F^P, P)\)-martingale such that \(Z^n_- > 0\) except on a \(P\)-evanescence set.
\end{lemma}
\begin{proof}
	The process \([c(\widehat{X}_-)]_{\cdot \wedge \tau_n}\) is \(\F\)-predictable and \(P\)-a.s.
	\begin{align}\label{eq: ui bound 1}
	\int_0^{t \wedge \tau_n} \la a(\widehat{X}_{s-})c(\widehat{X}_{s-}), c(\widehat{X}_{s-})\ra \dd s \leq n \sup_{\|x\|\leq n} |\la a(x) c(x), c(x)\ra|.
	\end{align}
	The r.h.s. is finite due to the assumptions that \(\la ac, c\ra\) is bounded on bounded subsets of \(\mathbb{B}\). Hence, for all \(n \in \mathbb{N}\), we can define a continuous \((\F^P, P)\)-martingale
	\(
	\int_0^{\cdot \wedge \tau_n} \langle  \dd \widehat{X}^{n, c}_s,c(\widehat{X}_{s-})\ra
	\)
	as in Lemma \ref{lem: definition si}. 
	
	For all \(n \in \mathbb{N}\), we have 
	\begin{equation}\label{eq: ui bound 2}
	\begin{split}
	\left(1 - \sqrt{Y}\right)^2 \star \nu^X_{t \wedge \tau_n} &= \int_0^{t \wedge \tau_n} \int \left(1 - \sqrt{Y(X_{s-}, x)}\right)^2 K(X_{s-}, \dd x)\dd s 
	\\&\leq n \sup_{\|y\| \leq n} \left( \int \left( 1 - \sqrt{Y(y, x)}\right)^2 K(y, \dd x)\right)
	\end{split}
	\end{equation}
	which is finite due to our assumption that \(y \mapsto \int (1 - \sqrt{Y(y, x)})^2 K(y, \dd x)\) is bounded on bounded subsets of \(\mathbb{B}\). 
	Moreover, we have
	\( 
	(Y(X_{-}, \Delta X) - 1) \1_{\{\Delta X \not= 0\}} \geq -1.
	\)
	Hence, by \cite[Proposition II.1.33]{JS}, for all \(n \in \mathbb{N}\) the stochastic integral process \((Y - 1) \star (\mu^X - \nu^X)_{\cdot \wedge \tau_n}\) is a well-defined discontinuous local \((\F^P, P)\)-martingale. We conclude that \(Z^n\) is well-defined.
	Since \(P\)-a.s. for all \(n \in \mathbb{N}\) and all \(t \in [0, \infty)\)
	\begin{equation}\label{eq:jump N}
	\begin{split}
	\Delta &\left( \int_0^{\cdot \wedge \tau_n} \la  \dd \widehat{X}^{n, c}_s,c(\widehat{X}_{s-})\rangle + (Y - 1) \star (\mu^X - \nu^X)_{\cdot \wedge \tau_n} \right)_t 
	\\&\qquad\quad= (Y(X_{t-}, \Delta X_t) - 1) \1_{\{\Delta X_{t} \not = 0\}} \1_{\of 0, \tau_n\gs} (\cdot, t) > - 1,
	\end{split}
	\end{equation}
	we deduce from \cite[Theorem I.4.61]{JS} that \(Z^n > 0\) and \(Z^n_- > 0\) up to a \(P\)-evanescence set.
	Moreover, \cite[Lemma 8.8, Theorem 8.25]{J79}, \eqref{eq: ui bound 1} and \eqref{eq: ui bound 2} yield that \(Z^{n}\) is a uniformly integrable \((\F^P, P)\)-martingale.
	
	Next, we extend the sequence \((Z^n)_{n \in \mathbb{N}}\). By the uniqueness of \(\widehat{X}^{n, c}\), we have \(\widehat{X}^{n + 1, c} = \widehat{X}^{n, c}\) on \(\of 0, \tau_n\gs\). This yields that \(Z^{n + 1} = Z^n\) on \(\of 0, \tau_n\gs\) and we may set 
	\[
	Z^* \triangleq \begin{cases} Z^n,&\textup{ on } \of 0, \tau_n\gs,\\
	\liminf_{n \to \infty} Z^n_{\tau_n},&\textup{ otherwise}.
	\end{cases}
	\]
	We have \(P\)-a.s. \(Z^*_t = \liminf_{n \to \infty} Z^n_{t \wedge \tau_n}\). Hence, by Fatou's lemma and the \((\F^P, P)\)-martingale property of \(Z^{n}\), the process \(Z^*\) is \(P\)-indistinguishable from a non-negative \((\F^P, P)\)-supermartingale. 
	To show that \(Z^*\) is a local \((\F^P, P)\)-martingale we follow the proof of \cite[Lemma 12.43]{J79}.
	By the Doob-Meyer decomposition theorem for supermartingales, there exists a local \((\F^P, P)\)-martingale \(M\) and a \cadlag \(\F^P\)-predictable process of finite variation, both starting at 0, such that up to \(P\)-indistinguishability
	\(
	Z^* = 1 + M + B.
	\)
	Since, for all \(n \in \mathbb{N}\), \(P\)-a.s. \(Z^*_{\cdot \wedge \tau_n} = Z^n\), we have \(B = 0\) on \(\bigcup_{n \in \mathbb{N}} \of 0, \tau_n\gs\) up to a \(P\)-evanescence set and 	recall that \(\tau_{\Delta}\) is an \(\F^P\)-predictable time which is \(P\)-announced by the sequence \((\tau_n)_{n \in \mathbb{N}}\), see Remark \ref{rem: exp time pred}. Thus, \(P\)-a.s. \(B = \Delta B_{\tau_{\Delta}} \1_{\of \tau_{\Delta}, \infty\of}\). 
	For an \((- \infty, \infty]\)-valued \(\mathscr{F} \otimes \mathscr{B}([0, \infty))\)-measurable process \(Y\) we denote by \(^p Y\) the \(\F^P\)-predictable projection of \(Y\), see \cite[Theorem I.2.28]{JS}. 
	Thanks to \cite[Corollary I.2.31]{JS} it holds that \(^p (\Delta M) = 0\) up to \(P\)-evanescence. Hence, \(^p (\Delta Z^*) =\ ^p(\Delta B) = B_{\tau_\Delta} \1_{\of \tau_{\Delta}\gs}\) up to \(P\)-evanescence. However, since \(P\)-a.s. \(\Delta Z^*_{\tau_{\Delta}} = 0\) by construction, we obtain from \cite[Theorem I.2.28]{JS} that 
	\(P\)-a.s. \(B_{\tau_{\Delta}} = 0\). Hence, \(Z^* = 1 + M\) up to \(P\)-indistinguishability and our claim is proven.
\end{proof}
\subsubsection{Proof of Theorem \ref{theo:main1}}
Let us denote by \(Z^*\) the non-negative local \((\F^P, P)\)-martingale given as in Lemma \ref{lem: candidate density}.
Now, by Lemma \ref{lem: candidate density}, for all \(n \in \mathbb{N}\) the process \(Z^*_{\cdot \wedge \tau_n}\) is a positive uniformly integrable \((\F^P, P)\)-martingale starting at 1. Thus, we can define by \(Q_n (G)  \triangleq \E^P[\1_G Z^*_{\tau_n}]\) for \(G \in \mathscr{F}\) a sequence of probability measures on \((\Omega, \mathscr{F})\).
Suppose that we can show that \(Q_n = Q\) on \(\mathscr{F}_{\tau_n}\) for all \(n \in \mathbb{N}\). Let \(\rho\) be an \(\F\)-stopping time. For all \(G \in \mathscr{F}_\rho\) it holds that \(G \cap \{\tau_n > \rho\}\in \mathscr{F}_{\tau_n} \cap \mathscr{F}_{\rho} = \mathscr{F}_{\tau_n \wedge \rho}\). Hence, we may conclude from the optional stopping theorem that for all \(n \in \mathbb{N}\)
\begin{align*}
Q(G \cap \{\tau_n> \rho\}) &= \E^P\left[ Z^*_{\tau_n} \1_{G \cap \{\tau_n > \rho\}}\right] 
\\&= \E^P \left[\E^P\left[Z^*_{\tau_n} |\mathscr{F}^P_{\tau_n \wedge \rho}\right]\1_{G \cap \{\tau_n > \rho\}}\right] 
\\&= \E^P\left[Z^*_{\tau_n \wedge \rho} \1_{G \cap \{\tau_n> \rho\}}\right]
\\&= \E^P\left[Z^*_{\rho} \1_{G \cap \{\tau_n> \rho\}}\right].
\end{align*}
Now, letting \(n \to \infty\) and noting that \(\tau_n \nearrow \tau_{\Delta}\) as \(n \to \infty\) yields that 
\[
Q(G \cap \{\tau_{\Delta} > \rho\}) = E^P \left[ Z^*_\rho \1_{G \cap \{\tau_{\Delta} > \rho\}}\right].
\]
Thanks to \cite[Lemma 7, Appendix 1]{DellacherieMeyer78} there exists an \(\F^+\)-optional process \(\widehat{Z}\) such that \(\widehat{Z} = Z^*\) up to \(P\)-indistinguishability. Moreover, since, due to \cite[Theorem IV.59]{DellacherieMeyer78}, each \(\F^P\)-stopping time is \(P\)-a.s. equal to an \(\F^+\)-stopping time, \(\widehat{Z}\) is a local \((\F^+, P)\)-martingale thanks to the tower rule. 
We conclude that \((Q, P)\) is a CMG pair and that \(\widehat{Z}\) is the corresponding CMG density.

It remains to prove that \(Q_n = Q\) on \(\mathscr{F}_{\tau_n}\) for all \(n \in \mathbb{N}\). By Lemma \ref{lem: Q n sol} below, it holds that \(Q_n \in \mathcal{M}(A, b', a, K', \eta, \tau_n)\).
Then, since the GMP \((A, b', a, K', \tau_{\Delta}-)\) is completely locally well-posed, see Proposition \ref{prop:locuni}, the desired identity holds and the theorem is proven.
\qed
\begin{lemma}\label{lem: Q n sol}
	\(Q_n \in \mathcal{M}(A, b', a, K', \eta, \tau_{n})\).
\end{lemma}
\begin{proof}
	First of all, note that \(Q_n = P\) on \(\mathscr{F}_0\). Hence, \(Q_n \circ X^{-1}_0 = \eta\).
	Let \(y^* \in D(A^*)\) and \(n \in \mathbb{N}\). Then, Lemma \ref{lem: equi JS} yields that the real-valued process \(Y \triangleq \la \widehat{X}_{\cdot \wedge \tau_n}, y^*\ra\) is an \((\F^P, P)\)-semimartingale whose semimartingale characteristics are given by \eqref{eq:smc}.
	We now show that \(Y\) is an \((\F^{Q_n}, Q_n)\)-semimartingale with characteristics 
	\begin{align*}
	B^{Q_n} &=  \int_0^{\cdot\wedge \tau_n} \left( \langle \widehat{X}_{s-}, A^* y^*\rangle + \langle b'(\widehat{X}_{s-}), y^*\rangle\right)\dd s
	\\&\quad\quad + \int_0^{\cdot\wedge \tau_n} \int \left(k(\langle x, y^*\rangle) - \langle h(x), y^*\rangle\right) K'(\widehat{X}_{s-}, \dd x) \dd s
	\\
	C^{Q_n}&= \int_0^{\cdot\wedge \tau_n} \langle a(\widehat{X}_{s-}) y^*, y^*\rangle \dd s,
	\\
	\nu^{Q_n}([0, \cdot ], G) &= \int_0^{\cdot \wedge\tau_n} \int \1_G(\langle x, y^*\rangle) K'(\widehat{X}_{s-}, \dd x) \dd s,\ \  G \in \mathscr{B}(\mathbb{R}), 0 \not \in G.	
	\end{align*}
	Since \(Q_n\) is equivalent to \(P\), we have \(\F^P = \F^{Q_n}\).
	Hence, \(Y\) is an \((\F^{Q_n}, Q_n)\)-semimartingale thanks to \cite[Theorem III.3.13]{JS}.
	In particular, the Dol\'eans measure \(M^{Q_n}_{\mu^Y}\) is \(\mathscr{P}^{Q_n}\)-\(\sigma\)-finite.
	Moreover, the quadratic variation process of the continuous local martingale part of \(Y\) is the same under \(P\) and \(Q_n\), see \cite[Theorem III.3.11]{JS}.
	Therefore, the formula for \(C^{Q_n}\) follows.
	It remains to prove the formulas for \(B^{Q_n}\) and \(\nu^{Q_n}\).
	
	We start with \(\nu^{Q_n}\) and use a Girsanov-type theorem for integer-valued random measures. 
	Let \(\mu^X\) be given by \eqref{mu X} and \(W\) be a non-negative \(\mathscr{F} \otimes \mathscr{B}([0, \infty))\otimes \mathscr{B}(\mathbb{B}_\Delta)\)-measurable function. We denote by \(M^P_{\mu^X}(W |\mathscr{P}^P \otimes \mathscr{B}(\mathbb{B}_\Delta))\) a \(\mathscr{P}^P \otimes \mathscr{B}(\mathbb{B}_\Delta)\)-measurable real-valued function such that
	\begin{align*}
	M^P_{\mu^X}(UW) = M^P_{\mu^X}(U M^P_{\mu^X}(W|\mathscr{P}^P\otimes \mathscr{B}(\mathbb{B}_\Delta)))
	\end{align*}
	for all bounded non-negative \(\mathscr{P}^P\otimes \mathscr{B}(\mathbb{B}_\Delta)\)-measurable functions \(U\).
	If \(W(\omega, t, x) = W(\omega, t)\) is a local \((\F^P, P)\)-martingale, then \(M^P_{\mu^X} (W |\mathscr{P}^P\otimes \mathscr{B}(\mathbb{B}_\Delta))\) exists and it \(M^P_{\mu^X}\)-a.e. equals \(W_- + M^P_{\mu^X}(\Delta W|\mathscr{P}^P\otimes \mathscr{B}(\mathbb{B}_\Delta))\), see \cite[Problem 3.2.9, Theorem 3.3.1]{liptser1989theory}.
	Recalling \eqref{eq:jump N}, on \(\of 0, \tau_n\gs\), up to \(P\)-evanescence, it holds that
	\begin{align*}	\Delta Z^*= Z^*_- (Y (X_{-}, \Delta X) - 1) \1_{\{\Delta X \not = 0\}}.
	\end{align*}
	Hence, \(M^P_{\mu^X}\)-a.e.
	\begin{align*}
	M^P_{\mu^X} (\Delta Z^*_{\cdot \wedge \tau_n} |\mathscr{P}^P \otimes \mathscr{B}(\mathbb{B}_\Delta)) = Z^*_-(Y - 1) \1_{\of 0, \tau_n\gs}.
	\end{align*}
	We deduce from the Girsanov-type theorem for integer-valued random measure given by \cite[Theorem III.3.17]{JS} that \(\mu^X\) has a \(\mathscr{P}^{Q_n}\)-\(\sigma\)-finite Dol\'eans measure \(M^{Q_n}_{\mu^X}\) and that its \(Q_n\)-compensator is given by
	\begin{align}\label{eq:comp mu X Q}
	\nu^{X, Q_n} (\omega, \dd t, \dd x) \triangleq (1 - (Y(X_{t-}(\omega), x) - 1) \1_{\of 0, \tau_n\gs} (\omega, t)) \nu^X(\omega, \dd t, \dd x).
	\end{align}
	
	Let \(W\) be a non-negative \(\mathscr{P}^{Q_n}\otimes \mathscr{B}(\mathbb{R})\)-measurable function. 
	Then, using the formula \eqref{eq:comp mu X Q}, we obtain
	\begin{align*}
	M^{Q_n}_{\mu^Y} (W) &= M^{Q_n}_{\mu^X} (\1_{\of 0, \tau_n\gs}W (\cdot, \cdot, \la \cdot, y^*\ra)) 
	\\&= M^{Q_n}_{\nu^{X, Q_n}} (\1_{\of 0, \tau_n\gs} W (\cdot, \cdot, \la \cdot, y^*\ra)) 
	\\&= M^{Q_n}_{\nu^{Q_n}} (W).
	\end{align*}
	This proves the claimed formula for \(\nu^{Q_n}\).
	
	We now verify the formula for \(B^{Q_n}\) by using a Girsanov-type theorem for local martingales.
	Denote the continuous local \((\F^P, P)\)-martingale part of \(Y\) by \(Y^c\), the continuous local \((\F^P, P)\)-martingale part of \(Z^*_{\cdot \wedge \tau_n}\) by \(Z^{c}_{\cdot \wedge \tau_n}\), and by \(\lle \cdot, \cdot\rre\) the predictable quadratic covariation.
	Using the formula \eqref{eq: qv}, the polarization identity and that \(a\) is symmetric, we obtain
	\begin{align*} 
	\lle Z^c_{\cdot\wedge \tau_n}, Y^c\rre = \int_0^{\cdot\wedge \tau_n} Z^*_{s-}  \la a(\widehat{X}_{s-} ) c(\widehat{X}_{s-}), y^*\ra \dd s.
	\end{align*}
	Hence, the Girsanov-type theorem  \cite[Theorem III.3.11]{JS} yields that 
	\begin{align}\label{eq: clmp Q}
	Y^{c} - \int_0^{\cdot\wedge \tau_n} \la a(\widehat{X}_{s-}) c(\widehat{X}_{s-}), y^*\ra \dd s
	\end{align}
	is a continuous local \((\F^{Q_n}, Q_n)\)-martingale. 
	
	Next, consider the discontinuous local \((\F^P, P)\)-martingale \(M \triangleq k(x) \star (\mu^Y- \nu)\), where \(\nu\) is the \(P\)-compensator of \(\mu^Y\). 
	Since the jumps of \(M\) are bounded by \(\|k\|_\infty\), the quadratic covariation process \([M, Z^*_{\cdot \wedge \tau_n}]\) has \(\F^P\)-locally \(P\)-integrable variation, see \cite[Lemma III.3.14]{JS}.
	Hence, there exists a \(P\)-compensator \(\lle M, Z^*_{\cdot\wedge\tau_n}\rre\) for \([M, Z^*_{\cdot\wedge \tau_n}]\), which we compute next.
	Note that \(P\)-a.s.
	\begin{align*}
	[M, Z^*_{\cdot\wedge \tau_n}] &= \sum_{s \in [0, \cdot \wedge\tau_n]} \Delta M_s \Delta Z^*_s 
	\\&= \sum_{s \in [0, \cdot\wedge \tau_n]} k(\la \Delta X_s, y^*\ra) Z^*_{s-} (Y(X_{s-}, \Delta X_{s}) - 1) \1_{\{\Delta X_s \not = 0\}}
	\\&= Z^*_- k(\la \cdot, y^*\ra)(Y(X_-, \cdot) - 1) \star \mu^X_{\cdot\wedge \tau_n}.
	\end{align*}
	Since \([M, Z^*_{\cdot \wedge \tau_n}]\) has \(\F^P\)-locally \(P\)-integrable variation, the process \[|Z^*_-k(\la\cdot, y\ra) (Y(X_-, \cdot) - 1)| \star \mu^X_{\cdot \wedge  \tau_n}\] is \(\F^P\)-locally \(P\)-integrable. Hence, by the properties of the \(P\)-compensator \(\nu^X\) of \(\mu^X\) as stated in \cite[Theorem II.1.8]{JS}, we conclude
	\[
	\lle M, Z^*_{\cdot \wedge  \tau_n}\rre = Z^*_- k(\la \cdot, y^*\ra)(Y(X_-, \cdot) - 1) \star \nu^X_{\cdot \wedge \tau_n}.
	\]
	Again thanks to the Girsanov-type theorem \cite[Theorem III.3.11]{JS}, the process
	\[
	k(x) \star (\mu^Y - \nu) - k(\la \cdot, y^*\ra)(Y(X_-, \cdot) - 1) \star \nu^X_{\cdot\wedge \tau_n}
	\]
	is the discontinuous local \((\F^{Q_n}, Q_n)\)-martingale whose jumps equal \(k(\Delta Y)\).
	Together with \eqref{eq: clmp Q} and the uniqueness of the first characteristic, we conclude that the first \(Q_n\)-characteristic of \(Y\) is given by
	\begin{align}\label{eq: fc}
	B(k) + \int_0^{\cdot \wedge  \tau_n} \la a(\widehat{X}_{s-}) c(\widehat{X}_{s-}), y^*\ra \dd s + k(\la \cdot, y^*\ra)(Y(X_-, \cdot) - 1) \star \nu^X_{\cdot \wedge \tau_n},
	\end{align}
	where \(B(k)\) is given as in Lemma \ref{lem: equi JS} (iv).
	Now, using that \(\widehat{X}_{\cdot\wedge \tau_n} = X_{\cdot \wedge \tau_n}\) up to \(Q_n\)-indistinguishability and that the integral \(\int h(y) (Y(\cdot, y) - 1) K(\cdot, \dd y)\) was defined as a Pettis integral, i.e. in particular
	\[
	\left \la \int h(y) (Y(\cdot, y) - 1) K(\cdot, \dd y), y^*\right\ra = \int \la h(y), y^*\ra (Y(\cdot, y) - 1) K(\cdot, \dd y), 
	\]
	it follows that \eqref{eq: fc} equals \(B^{Q_n}\) up to a \(Q_n\)-evanescence set. 
	We conclude the claim from Lemma \ref{lem: equi JS}.
\end{proof}
\subsection{Proof of Proposition \ref{theo:main2}}\label{sec: pf coro uni}
Suppose that the GMP \((A, b', a, K',  \tau_{\Delta}-)\) is well-posed.
Denote by \(P\) and \(P'\) two solutions to the GMP \((A, b, a, K, \eta, \tau_{\Delta}-)\), let \(Q\) be the unique solution to the GMP \((A, b', a, K', \eta, \tau_{\Delta}-)\) and fix \(n \in \mathbb{N}\).
By Theorem \ref{theo:main1}, we can relate \(P\) and \(P'\) via the CMG formula \eqref{eq: CMG formula}. More precisely, we have for all \(G \in \mathscr{F}_{\tau_n}\) 
\begin{align}\label{eq: CMG uni}
E^P\big[Z_{\tau_n} \1_{G}\big] = Q(G) = E^{P'} \big[Z'_{\tau_n} \1_{G}\big],
\end{align}
where \(Z\) and \(Z'\) are the corresponding CMG densities. We stress that these are positive, see Lemma \ref{lem: candidate density}.
Here, \eqref{eq: CMG uni} requires Lemma \ref{lem:tau}, but the equality also follows from Proposition \ref{prop:locuni} and Lemma \ref{lem: Q n sol}. Our goal is to show that there exists an \(\mathscr{F}_{\tau_n}\)-measurable random variable \(Z^*_{\tau_n}\) such that \(P\)-a.s. \(Z^*_{\tau_n} = Z_{\tau_n}\) and \(P'\)-a.s. \(Z^*_{\tau_n} = Z'_{\tau_n}\). In this case, for all \(G \in \mathscr{F}_{\tau_n}\), \eqref{eq: CMG uni} implies that
\begin{align*}
E^Q\left[\frac{1}{Z^*_{\tau_n}} \1_{\{Z^*_{\tau_n} > 0\} \cap G}\right] = \begin{cases}  P(G),\\ P'(G),\end{cases}
\end{align*}
which proves that \(P\) and \(P'\) coincide on \(\mathscr{F}_{\tau_n}\). By a monotone class argument, we obtain that \(P = P'\) on \(\mathscr{F}_{\tau_{\Delta}-} = \mathscr{F}\).

By Lemma \ref{lem: candidate density} and \cite[Theorem I.4.61]{JS}, up to a null set, the CMG densities are given by
\begin{align*}
\exp &\left( \int_0^{\tau_n} \la \dd \widehat{X}^{n, c}_s, c(\widehat{X}_{s-}) \rangle + (Y - 1) \star (\mu^X- \nu^X)_{\tau_n}\right) 
\\&\qquad \times \exp \left( - \frac{1}{2} \int_0^{\tau_n} \langle a(X_{s-}) c(X_{s-}), c(X_{s-})\rangle \dd s \right)
\\&\qquad \times\prod_{s \in [0, \tau_n]} (1 + \Delta M_s) \exp (- \Delta M_s),	
\end{align*}
where
\begin{align*}
\Delta M_s \triangleq (Y(X_{s-}, \Delta X_s) - 1) \1_{\{\Delta X_s \not = 0\}},
\end{align*}
and we interpret the stochastic integrals either relative to \(P\) or to \(P'\). It suffices to study the stochastic integrals.
We start with \((Y - 1) \star (\mu^X- \nu^X)_{\tau_n}\).
\begin{lemma}\label{lem: V approx}
	For a \(\mathscr{P} \otimes \mathscr{B}(\mathbb{B}_\Delta)\)-measurable \(V \in G_\textup{loc}(\mu^X)\), where \(\mathscr{P}\) denotes the \(\F\)-predictable \(\sigma\)-field, set
	\[
	V^k (\omega, s, y) \triangleq V(\omega, s, y) \1_{\{\|y\| \geq k^{-1}\}} \1_{\{|V(\omega, s, y)| \leq k\}}.
	\]
	Let \(\xi\) be an \(\F\)-stopping time.
	We have 
	\begin{equation*}\begin{split}
	\sum_{s \in [0, \xi\wedge \tau_n]} V^{k} (\cdot, s, \Delta X_s) - \int_0^{\xi \wedge \tau_n} \int V^{k}(\cdot&, s, x) K(X_{s-}, \dd x) \dd s \\&\xrightarrow{\quad k \to \infty\quad}V \star (\mu^X- \nu^X)_{\xi \wedge \tau_n},
	\end{split}
	\end{equation*}
	where the convergence is in \(P\)-probability if the stochastic integral in relative to \(P\) and in \(P'\)-probability if the stochastic integral is relative to \(P'\).
\end{lemma}
\begin{proof}
	Since \(\nu^X(\{s\}, \dd x) = 0\) for all \(s \in [0, \infty)\), \(V \in G_\textup{loc}(\mu^X)\) implies \(|V| \in G_\textup{loc}(\mu^X)\), see \eqref{eq: jump jump int}.
	Hence, \cite[Proposition 2.35, Theorem 2.36]{Bichteler1983} imply that
	\begin{align*}
	\lim_{k \to \infty} V^k \star (\mu^X- \nu^X)_{\xi \wedge \tau_n} = V \star (\mu^X- \nu^X)_{\xi \wedge \tau_n}
	\end{align*}
	in probability.
	Observe that
	\[
	|V^{k}| \star \nu^X_{\xi \wedge \tau_n} \leq kn \sup_{\|z\| \leq n}\int K( z, \{x \in \mathbb{B} \colon \|x\| \geq k^{-1}\}) < \infty.
	\]
	Thus, by \cite[Proposition II.1.28]{JS}, we have a.s.
	\begin{align*}
	V^{k} \star (&\mu^X- \nu^X)_{\xi \wedge \tau_n} \\&= \sum_{s \in [0, \xi\wedge \tau_n]} V^{k} (\cdot, s, \Delta X_s) - \int_0^{\xi \wedge \tau_n} \int V^{k}(\cdot, s, x) K(X_{s-}, \dd x) \dd s. 
	\end{align*}
	This completes the proof.
\end{proof}
We deduce from Lemma \ref{lem: V approx} that we find an \(\mathscr{F}_{\tau_n}\)-measurable random variable which equals \((Y - 1) \star (\mu^X- \nu^X)_{\tau_n}\) \(P\)-a.s. if the stochastic integral is relative to \(P\) and \(P'\)-a.s. if the stochastic integral is relative to \(P'\).

Let us now turn to the stochastic integral \(\int_0^{\tau_n} \la \dd \widehat{X}^{n, c}_s, c(\widehat{X}_{s-})\ra\). We note the following

\begin{lemma}\label{lem: simple fct} There exists a sequence of simple \(\F\)-predictable processes
	\begin{align}\label{eq: simple process}
	f^p \triangleq \sum_{i = 1}^{m_{p}} \left( \lambda_0^{i, p} \1_{\{0\}} + \sum_{j = 0}^{\infty} \lambda_j^{i, p} \1_{(t_{j}^{i, p}, t^{i, p}_{j + 1}]} \right)y^*_{i, p},\quad p \in \mathbb{N},
	\end{align} such that
	\begin{align*}
	E^P &\left[ \int_0^{\tau_n} \left\langle a(X_{s-}) \left(f^p_s - c(X_{s-})\right), f^p_s - c(X_{s-})\right\rangle \dd s \right]
	\triangleq\| f^p - c(X_-)\|^2_a 
	\to 0
	\end{align*}
	as \(p \to \infty\).
	Here, \(\lambda^{i, p}_0\) is real-valued, bounded and \(\mathscr{F}_{0}\)-measurable, \(\lambda^{i, p}_j\) is real-valued, \(\mathscr{F}_{t^{i, p}_j}\)-measurable and such that \(\sup_{j \in \mathbb{N}} |\lambda^{i, p}_j (\omega)| \leq C < \infty\) for all \(\omega \in \Omega\), for all \(j \in \mathbb{N}\) we have \(t^{i, p}_j < t^{i, p}_{j+1}\), \(t_j^{i, p} \nearrow \infty\) as \(j \to \infty\) and \(y^*_{i, p} \in D(A^*)\).
\end{lemma}
\begin{proof} It is shown in the proof of \cite[Proposition 9]{Mikulevicius1998} that there exists a sequence \((\hat{f}^m)_{m \in \mathbb{N}} = (\sum_{j = 1}^{N_m} f^{j, m} y^*_{j, m})_{m \in \mathbb{N}}\), where \(f^{j, m}\) are bounded \(\F\)-predictable processes and \(y^*_{j, m} \in \mathbb{B}^*\), such that 
	\begin{align*}
	\lim_{m \to \infty} \|\hat{f}^m - c(X_-)\|_a 
	= 0.
	\end{align*}
	By \cite[Proposition 3.2.6]{KaraShre}, where we stress that for \(\F\)-predictable processes the proposition holds without augmenting the filtration, for each \(m \in \mathbb{N}\) there exists a sequence \((\hat{f}^{m, k})_{k \in \mathbb{N}}\)
	of the type \eqref{eq: simple process} with \(y^*_{i, p} \in \mathbb{B}^*\)
	such that 
	\begin{align*}
	\lim_{k \to \infty} \|\hat{f}^{m, k} - \hat{f}^m\|_a 
	= 0.
	\end{align*}
	Now, for each \(m \in \mathbb{N}\) we find a \(k_m\in \mathbb{N}\) such that 
	\begin{align*}
	\|\hat{f}^{m, k_m} - \hat{f}^m\|_a \leq \frac{1}{m}.
	\end{align*}
	Thus, we have 
	\begin{align*}
	\|\hat{f}^{m, k_m} - c(X_-)\|_a &\leq \|\hat{f}^{m, k_m} - \hat{f}^m\|_a + \|\hat{f}^m - c(X_-)\|_a \\&\leq \frac{1}{m} + \|\hat{f}^m - c(X_-)\|_a \to 0
	\end{align*}
	as \(m \to \infty\). 
	Let \(y^* \in \mathbb{B}^*\). Then, since \(D(A^*)\) is dense in \(\mathbb{B}^*\), we find a sequence \((y^*_k)_{k \in \mathbb{N}} \subset D(A^*)\) such that \(y^*_k \to y^*\) as \(k \to \infty\). 
	For the processes \(\tilde{f} \triangleq \big(\lambda_0 \1_{\{0\}} + \sum_{j = 0}^\infty \lambda_j \1_{(t_j, t_{j + 1}]} \big) y^*\) and \(\tilde{f}^k \triangleq \big(\lambda_0 \1_{\{0\}} + \sum_{j = 0}^\infty \lambda_j \1_{(t_j, t_{j + 1}]} \big) y^*_k\), the dominated convergence theorem yields that 
	\(\|\tilde{f} - \tilde{f}^k\|_a \to 0\) as \(k \to \infty\). Therefore, using a similar argument as above, the lemma is proven.
\end{proof}
Since \(\widehat{X}_{\cdot \wedge \tau_n} = X_{\cdot \wedge \tau_n}\) up to \(P\)- and \(P'\)-indistinguishability, for all \(s \in [0, \infty)\) and \(y^* \in D(A^*)\), we have 
\begin{align*}
\widehat{X}^{n, c}(y^*)_{s \wedge \tau_n} = \la X_{s \wedge \tau_n}, y^*\ra &- \la X_0, y^*\ra - \la h, y^*\ra \star (\mu^{X} - \nu^{X})_{s \wedge \tau_n} 
\\&- \int_0^{s \wedge \tau_n} \left( \langle X_{r-}, A^* y^*\rangle + \langle b(X_{r-}), y^*\rangle \right)\dd r 
\\&- (\la \cdot, y^*\ra -\la h, y^*\ra) \star \mu^{X}_{s \wedge \tau_n},
\end{align*}
up to a null set, see \cite[Theorem II.2.34]{JS}. 
In this equality \(\widehat{X}^{n, c}(y^*)_{s \wedge \tau_n}\) and the stochastic integral \(\la h, y^*\ra \star (\mu^{X} - \nu^{X})_{s \wedge \tau_n}\) are interpreted relative to the same probability measure.
Thus, by Lemma \ref{lem: V approx}, for each \(s \in [0, \infty)\) and \(y^* \in D(A^*)\) we find an \(\mathscr{F}_{\tau_n}\)-measurable random variable \(X(s, y^*)\) which equals \(P\)-a.s. \(\widehat{X}^{n, c}(y^*)_{s \wedge \tau_n}\) relative to \(P\) and \(P'\)-a.s. \(\widehat{X}^{n, c}(y^*)_{s \wedge \tau_n}\) relative to \(P'\).

Let \((f^p)_{p \in \mathbb{N}}\) be as in Lemma \ref{lem: simple fct}. We have 
\begin{align}\label{eq: approx seq prob}
\int_0^{\tau_n} \left\langle a(X_{s-}) \left(f^p_s - c(X_{s-})\right), f^p_s - c(X_{s-})\right\rangle \dd s \to 0
\end{align}
as \(p \to \infty\) in \(P\)-probability. Since the CMG densities are positive, see Lemma \ref{lem: candidate density}, it follows from \eqref{eq: CMG uni} that \(P\) and \(P'\) are equivalent on \(\mathscr{F}_{\tau_n}\). Thus, \eqref{eq: approx seq prob} holds also in \(P'\)-probability, see \cite[Exercise A.8.11]{Bichteler02}.

We deduce from Lemma \ref{lem: definition si} that as \(p \to \infty\) the sequence of \(\mathscr{F}_{\tau_n}\)-measurable random variables
\begin{align*}
\sum_{i = 1}^{m_p} \sum_{j = 0}^\infty \lambda^{i, p}_j \1_{\{t_j^{i, p} < \tau_n\}} \left(X\left(t^{i, p}_{j + 1}, y^*_{i, p}\right) - X\left(t^{i, p}_{j}, y^*_{i, p}\right)\right) 
\end{align*}
converges in \(P\)-probability to the stochastic integral \(\int_0^{\tau_n} \la \dd \widehat{X}^{n, c}_s, c(\widehat{X}_{s-})\ra\) relative to \(P\) and in \(P'\)-probability to the stochastic integral \(\int_0^{\tau_n} \la \dd \widehat{X}^{n, c}_s, c(\widehat{X}_{s-})\ra\) relative to \(P'\).
Therefore, we can conclude the proof of the first statement.

The second statement follows by symmetry. More precise, one has to consider \(-c\) instead of \(c\) and \(Y^{-1}\) instead of \(Y\).
\qed

\subsection{Proof of Proposition \ref{coro:main2}}\label{sec: pf coro cons existence}
Let us first prove the implication (i) \(\Longrightarrow\) (ii).
Thanks to Proposition \ref{theo:main2}, (i) implies that the GMP \((A, b', a, K', \tau_\Delta-)\) is well-posed, i.e. by Proposition \ref{prop:locuni} also completely and locally well-posed. Denote by \(Q_x\) the unique solution to the GMP \((A, b', a, K', \varepsilon_x, \tau_\Delta-)\). Then, since by Lemma \ref{lem: Q n sol} the probability measure \(Q^n_x\) is a solution to the GMP \((A, b', a, K', \varepsilon_x,  \tau_n)\), we conclude that 
\(
Q^n_x = Q_x\) on \(\mathscr{F}_{\tau_n}.
\)
Therefore, for all \(t \in [0, \infty)\),
\[
\lim_{n \to \infty} Q^n_x (\tau_n > t) = \lim_{n \to \infty} Q_x (\tau_n > t) = 1, 
\]
by the assumption that the GMP \((A, b', a, K', \varepsilon_x, \tau_\Delta-)\) is conservative. This proves the implication (i) \(\Longrightarrow\) (ii).

We now suppose that (ii) holds. 
The following proposition is an extension of \cite[Theorem 1.3.5]{SV} to a \cadlag setting.
\begin{proposition}\label{theo:TET}
	Let \((\rho_n)_{n \in \mathbb{N}}\) be an increasing sequence of \(\mathbf{D}\)-stopping times and for each \(n \in \mathbb{N}\) suppose that \(P^o_n\) is a probability measure on \((\mathbb{D}, \mathscr{D}_{\rho_n})\). Assume that \(P^o_{n +1} = P^o_n\) on \(\mathscr{D}_{\rho_n}\) for all \(n \in \mathbb{N}\). If \(\lim_{n \to \infty} P^o_n(\rho_n \leq t) = 0\) for all \(t \in [0, \infty)\), then there exists a unique probability measure \(P^o\) on \((\mathbb{D}, \mathscr{D})\) such that \(P^o = P^o_n\) on \(\mathscr{D}_{\rho_n}\).
\end{proposition}
\begin{proof}
	We follow closely the proof of \cite[Theorem 1.3.5]{SV}.
	If \(P^o\) exists, it is unique on \(\mathscr{D}_t\) due to the identity
	\begin{align*}
	P^o(G) = \lim_{n \to \infty} P^o_n(G \cap \{\rho_n> t\}),\quad G \in \mathscr{D}_t.
	\end{align*}
	Thus, by a monotone class argument, \(P^o\) is unique on \(\bigvee_{t \in [0, \infty)} \mathscr{D}_t = \mathscr{D}\). 
	
	For all \(G \in \mathscr{D}_n\) we have that \(G \cap \{\rho_k > n\} \in \mathscr{D}_{\rho_k}\). Thus, it holds that 
	\begin{align*}
	P^o_k(G \cap \{\rho_k > n\}) = P^o_{k + 1} (G \cap \{\rho_k > n\}) \leq P^o_{k + 1} (G \cap \{\rho_{k +1} > n\}),\quad G \in \mathscr{D}_n.
	\end{align*}
	Due to this observation, we can define 
	\begin{align*}
	P^*_n (G) \triangleq \lim_{k \to \infty} P^o_k(G \cap \{\rho_k > n\}),\quad G \in \mathscr{D}_n.
	\end{align*}
	Clearly, \(P^*_n\) is finitely additive. Thus, if for any sequence \((G_k)_{k \in \mathbb{N}} \subset \mathscr{D}_n\) such that \(G_k \searrow \emptyset\) as \(k \to \infty\) it holds that \(\lim_{k \to \infty} P^*_n (G_k) = 0\), \(P^*_n\) is a measure on \((\mathbb{D}, \mathscr{D}_n)\), see \cite[Proposition 1.2.6]{cohn13}.
	In fact, in this case it is easily seen that \(P^*_n\) is even a probability measure. 
	Let \((G_k)_{k \in \mathbb{N}} \subset \mathscr{D}_n\) be such that \(G_k \searrow \emptyset\) as \(k \to \infty\).
	Then, we have 
	\begin{align*}
	P^*_n (G_k) &= P^*_n(G_k \cap \{\rho_m \leq n\}) + P^*_n(G_k \cap \{\rho_m > n\})
	\\&= \lim_{i \to \infty} P_{i + m}^o (G_k \cap \{\rho_m \leq n\} \cap \{\rho_{i + m} > n\})
	\\&\qquad\qquad+ \lim_{i \to \infty} P^o_{i + m}(G_k \cap \{\rho_m > n\} \cap \{\rho_{i + m} > n\})
		\\&= \lim_{i \to \infty} P_{i + m}^o (G_k \cap \{\rho_m \leq n\} \cap \{\rho_{i + m} > n\})
	+ P^o_{m}(G_k \cap \{\rho_m > n\})
	\\&\leq \lim_{i \to \infty} P^o_{i + m}(\rho_m \leq n) + P^o_m(G_k \cap \{\rho_m > n\})
	\\&= P^o_m(\rho_m \leq n) + P^o_m(G_k \cap \{\rho_m > n\}),
	\end{align*}
	where we use that \(P^o_i = P^o_m\) on \(\mathscr{D}_{\rho_m}\) for \(i \geq m\) and that  \(G_k \cap \{\rho_m> n\} \in \mathscr{D}_{\rho_m}\). Letting first \(k \to \infty\) and then \(m \to \infty\) yields that \(\lim_{k \to \infty} P^*_n (G_k) = 0\). Thus, \(P^*_n\) is a probability measure on \((\mathbb{D}, \mathscr{D}_n)\).
	
	For any \(G \in \mathscr{D}_{n}\) note that 
	\begin{align*}
	0 \leq  \lim_{k \to \infty} P^o_{k} (G \cap \{\rho_k > n + 1 \} \cap \{\rho_k \leq n\}) \leq \lim_{k \to \infty} P^o_k(\rho_k \leq n) = 0, 
	\end{align*}
and hence that 
	\begin{align*}
P^*_{n + 1} (G) &= \lim_{k \to \infty}P^o_{k} (G \cap \{\rho_k > n + 1\} \cap \{\rho_k > n\})
= P^*_n (G).
\end{align*}
	Summarizing this, we have \(P^*_{n + 1} = P^*_n\) on \(\mathscr{D}_{n}\) for all \(n \in \mathbb{N}\). Our next step is to extend the family \((P^*_n)_{n \in \mathbb{N}}\).
	
	Note that for \(\omega \in \mathbb{D}\)
	\begin{align*}
	A_n(\omega) \triangleq \bigcap_{B \in \mathscr{D}_n\colon \omega \in B} B = \{\omega^o \in \mathbb{D}\colon \omega^o(s) = \omega(s), s \in [0, n]\}.
	\end{align*}
	Let \((\omega_n)_{n \in \mathbb{N}} \subset \mathbb{D}\) be such that \(\bigcap_{k = 1}^N A_k(\omega_k)\not = \emptyset\) for all \(N \in \mathbb{N}\). Then, \(\omega_{n^o} = \omega_{n}\) on \([0, n^o]\) for all \(n^o \leq n \leq N\) and \(N \in \mathbb{N}\), and we have that 
	\begin{align*}
	\omega_1 (0) + \sum_{n = 1}^\infty \omega_n \1_{(n - 1, n]} \in \bigcap_{k = 1}^\infty A_k(\omega_k).
	\end{align*}
	Since \(\mathbb{D}\) is a Polish space, \(\mathscr{D}\) is its Borel \(\sigma\)-field and \(\mathscr{D}_n = \sigma(X_s, s \in [0, n])\) is countably generated, a regular conditional probability \(P^*_n (\cdot |\mathscr{D}_{n -1})\) exists, see \cite[Theorem II.89.1]{RW1}.
	There exists a \(P_{n+1}^*\)-null set \(N \in \mathscr{D}_n\) such that for all \(\omega \in \complement N\) we have 
	\begin{align*}
	P^*_{n + 1} (A_n (\omega)|\mathscr{D}_n) (\omega) = \1_{A_n(\omega)} (\omega) = 1.
	\end{align*}
	Thus, for all \(\omega \in \complement N\) and all \(B \in \mathscr{D}_{n+1}\) such that \(B \cap A_n(\omega) = \emptyset\), we have 
	\begin{align*}
	P^*_{n + 1}(B|\mathscr{D}_{n}) (\omega) = P^*_{n + 1} (B \cap A_n(\omega) |\mathscr{D}_{n}) (\omega) = 0.
	\end{align*}

We can apply Tulcea's extension theorem, see \cite[Theorem 1.1.9]{SV}, and obtain that there exists a probability measure \(P^o\) on \(\mathscr{D}\) such that \(P^o = P^*_1\) on \(\mathscr{D}_1\) and
	\begin{align*}
	P^o(B) = \int P_n^* (B|\mathscr{D}_{n - 1})(\omega) P^o(\dd \omega), \quad B \in \mathscr{D}_n,\quad n \geq 2.
	\end{align*}
	
	It remains to show that \(P^o = P^o_n\) on \(\mathscr{D}_{\rho_n}\). 
	We check first that \(P^o = P^*_n\) on \(\mathscr{D}_n\) for all \(n \in \mathbb{N}\). 
	If \(P^o = P^*_{n-1}\) on \(\mathscr{D}_{n-1}\), then 
	\begin{align*}
	P^o(B) &= \int P_n^*(B|\mathscr{D}_ {n - 1})(\omega) P^o(\dd \omega)\\&= \int P_n^*(B |\mathscr{D}_{n-1})(\omega) P^*_{n-1}(\dd \omega)
	\\&= \int P_{n}^*(B|\mathscr{D}_{n-1})(\omega) P^*_n (\dd \omega)
	\\&= P_n^*(B).
	\end{align*}
	The claim follows by induction.
	
	For \(G \in \mathscr{D}_{\rho_n} \cap \mathscr{D}_k = \mathscr{D}_{\rho_n \wedge k}\) we have
	\begin{align*}
	P^o(G) = P^*_k(G) = P^o_n(G), 
	\end{align*}
	since
	\begin{align*}
	|P^o_n(G) - P^*_k(G)| &= \left|\lim_{i \to \infty} P^o_{i + n} (G) - \lim_{i \to \infty} P^o_{i + n} (G \cap \{\rho_{i + n} > k\})\right| 
	  \\&= \lim_{i \to \infty} P^o_{i + n} (G \cap \{\rho_{i + n} \leq k\}) = 0.
	\end{align*}
	Since \(\bigvee_{k \in \mathbb{N}} \mathscr{D}_{\tau_n\wedge k} = \mathscr{D}_{\tau_n}\), a monotone class argument completes the proof.
\end{proof}

We can consider \(Q^n_x\) as a probability measure on \((\mathbb{D}, \mathscr{D}_{\tau_n})\) and \(Q_x^{n+1} = Q_x^n\) on \(\mathscr{D}_{\tau_n}\) follows from the optional stopping theorem. Thus, the sequence \((Q^n_x)_{n \in \mathbb{N}}\) satisfies the prerequisites of Proposition \ref{theo:TET}. We conclude from Lemma \ref{lem: Q n sol} that a solution \(Q_x\) to the MP \((A, b', a, K', \varepsilon_x)\) exists. Of course, this implies that the GMP \((A, b', a, K', \varepsilon_x, \tau_{\Delta}-)\) has a conservative solution.
By Proposition \ref{theo:main2}, the GMP \((A, b', a, K', \tau_{\Delta}-)\) is well-posed and by Proposition \ref{prop:locuni} completely well-posed.

It remains to show that for any initial law \(\eta\) which is supported on \((\mathbb{B}, \mathscr{B}(\mathbb{B}))\) the GMP \((A, b', a, K', \eta, \tau_\Delta-)\) has a conservative solution.
Due to Theorem \ref{theo:markov}, the unique solution \(Q_\eta\) to the GMP \((A, b', a, K', \eta, \tau_\Delta-)\) is given by 
\(
Q_\eta = \int Q_x \eta(\dd x).
\)
Hence, we have 
\[
Q_\eta(\tau_{\Delta} = \infty) = \int Q_x (\tau_{\Delta} = \infty) \eta(\dd x) = \int \eta(\dd x) = 1.
\]
This concludes the proof.
\qed

\subsection{Proof of Proposition \ref{coro: well-posed equivalence}}\label{sec: pf coro equivalence bdd coe}
Let us suppose that (i) holds. Let \(x \in \mathbb{B}\).
Assume that \(P_x\) solves the MP \((A, b, a, K, \varepsilon_x)\). 
In conservative cases all previous results can be proven for \(\F^P\) replaced by \(\mathbf{D}^+\). In particular, the stochastic integrals can be defined for \(\mathbf{D}^+\), see \cite{JS,Mikulevicius1998}, and it is not necessary to introduce the process \(\widehat{X}\).
Let \(Z^*\) be as in Lemma \ref{lem: candidate density}. 
By hypothesis, for all \(t \in [0, \infty)\) the random variable 
\begin{align*}
\int_0^t \la a(\widehat{X}_{s-}) &c(\widehat{X}_{s-}), c(\widehat{X}_{s-}) \ra \dd s + \int_0^t \int \left(1 - \sqrt{Y(\widehat{X}_{s-}, y)}\right)^2 K(\widehat{X}_{s-}, \dd y) \dd s 
\end{align*}
is bounded.
Hence, by \cite[Lemma 8.8, Theorem 8.25]{J79}, \(Z^*\) is a positive \((\mathbf{D}^+, P_x)\)-martingale. 
For all \(t \in [0, \infty)\), define a probability measure \(Q_x^t\) on \((\mathbb{D}, \mathscr{D}_{t})\) by \(Q^t_x (G) = E^{P_x} [Z^*_t \1_G]\) for \(G \in \mathscr{D}_{t}\).
Clearly, \(Q_x^t = Q^s_x\) on \(\mathscr{D}_s\) for \(s \in [0,t]\), by the \((\mathbf{D}^+, P_x)\)-martingale property of \(Z^*\).
Thus, due to Proposition \ref{theo:TET}, there exists a probability measure \(Q_x\) on \((\mathbb{D}, \mathscr{D})\) such that \(Q_x = Q^t_x\) on \(\mathscr{D}_{t}\) for all \(t \in [0, \infty)\). In particular, for \(G \in \mathscr{D}_{t+}\) we have 
\begin{align*}
Q_x(G) = Q_x^{t + \epsilon} (G) = E^{P_x} \left[Z^*_{t + \epsilon} \1_G\right]
\end{align*}
for all \(\epsilon > 0\). Thus, by the downward theorem, see \cite[Theorem II.51.1]{RW1}, letting \(\epsilon \searrow 0\) yields that the Radon-Nikodym density of \(Q_x\) w.r.t. \(P_x\) on \(\mathscr{D}_{t+}\) is given by \(Z^*_t\).
Since the Girsanov-type theorems used in the proof of Lemma \ref{lem: Q n sol} hold for any right-continuous filtration and any pair of locally absolutely continuous probability measures, we may use the same arguments as in the proof of Lemma \ref{lem: Q n sol} together with Lemma \ref{lem: equi JS} to conclude that \(Q_x\) solves the MP \((A, b', a, K', \varepsilon_x)\). Thus, since uniqueness for this MP holds by Proposition \ref{theo:main2}, we conclude that (ii) holds.

The converse implication follows by symmetry. More precisely, the same arguments as above with \(-c\) instead of \(c\), \(Y^{-1}\) instead of \(Y\), \(b'\) instead of \(b\) and \(K'\) instead of \(K\) yield the claim.
\qed

\subsection{Proof of Proposition \ref{coro:main1}}\label{sec: pf coro bdd coefficients}
(i). 
Suppose that the GMP \((A, b', a, K', \tau_{\Delta}-)\) is well-posed.
Due to Theorem \ref{theo:main1} and Lemma \ref{lem: candidate density}, for all \(t \in [0, \infty)\), there exists a non-negative random variable \(Z_t\) such that for all \(G \in \mathscr{F}_t\)
\begin{align}\label{eq: equivalence}
Q(G \cap\{\tau_\Delta > t\}) = E^P \left[Z_t \1_{G\cap \{\tau_\Delta > t\}}\right].
\end{align}
In particular, if \(P\) is conservative, \(P\)-a.s. \(Z_t > 0\). Thus, \(P\) and \(Q\) are locally equivalent if \(P\) and \(Q\) are conservative.

If \(P\) and \(Q\) are locally equivalent and \(P\) is conservative, then for all \(t \in [0, \infty)\) it holds that \(Q(\tau_\Delta < t) = 0\). Hence, since \(\{\tau_{\Delta} < \infty\} = \bigcup_{s \in \mathbb{Q}_+} \{\tau_{\Delta} < s\}\), we can conclude \(Q(\tau_{\Delta} = \infty) = 1.\)
Similarly, if \(P\) and \(Q\) are locally equivalent and \(Q\) is conservative, we have \(P(\tau_\Delta = \infty) = 1\).

These arguments are symmetric, i.e. in the case where the GMP \((A, b, a, K, \tau_{\Delta}-)\) is well-posed one has to consider \(-c\) instead of \(c\) and \(Y^{-1}\) instead of \(Y\). This concludes the proof of (i).

If \(Q\) is conservative and \(G \in \mathscr{F}_t\) is \(P\)-null, \eqref{eq: equivalence} immediately implies that \(G\) is also \(Q\)-null,
which proves (ii).
\qed
\subsection{Proof of Theorem \ref{theo: GMP SDE}}\label{sec: pf equi GMP SDE}
Let \(Y\) be a solution process to the SDE and \(y^* \in D(A^*)\). Then, the real-valued process \(\la Y, y^*\ra\) is a semimartingale with characteristics \((B^Y(k), C^Y, \nu^Y)\) given by 
\begin{align*}
B^Y (k) &= \int_0^\cdot \left (\la Y_{s-}, A^* y^*\ra + \la b(Y_{s-}), y^*\ra \right)\dd s 
\\&\qquad\quad+ \int_0^\cdot \int \left( k (\la \delta (Y_{s-}, x), y^*\ra) - \la h(\delta (Y_{s-}, x)), y^*\ra \right) F(\dd x) \dd s,\\
C^Y &= \int_0^\cdot \la \sigma (Y_{s-}) U \sigma^*(Y_{s-}), y^*\ra \dd s,\\
\nu^Y(\dd t, G) &= \int \1_G (\la \delta (Y_{t-}, x), y^*\ra) F(\dd x) \dd t,\quad G \in \mathscr{B}(\mathbb{R}), 0 \not \in G,
\end{align*}
see \cite[Theorem 14.80]{J79}.
Denote by \(P\) the law of \(Y\) and recall that we consider \(P\) as a probability measure on \((\mathbb{D},\mathscr{D})\). Furthermore, recall that we denote the coordinate process on \((\mathbb{D},\mathscr{D})\) by \(X\). Then, \cite[Theorem 10.37, Proposition 10.38, Proposition 10.39]{J79} yield that \(\la X, y^*\ra\) is a \((\mathbf{D}^P, P)\)-semimartingale with characteristics \((B(k), C, \nu)\) given by
\begin{align*}
B (k) &= \int_0^\cdot \left \la X_{s-}, A^* y^*\ra + \la b(X_{s-}), y^*\ra \right)\dd s 
\\&\qquad\quad+ \int_0^\cdot \int \left( k (\la \delta (X_{s-}, x), y^*\ra) - \la h(\delta (X_{s-}, x)), y^*\ra \right) F(\dd x) \dd s,\\
C &= \int_0^\cdot \la \sigma (X_{s-}) U \sigma^*(X_{s-}), y^*\ra \dd s,\\
\nu(\dd t, G) &= \int \1_G (\la\delta (X_{t-}, x), y^*\ra) F(\dd x) \dd t,\quad G \in \mathscr{B}(\mathbb{R}), 0 \not \in G.
\end{align*}
Hence, Lemma \ref{lem: equi JS} yields that \(P\) is a solution to the claimed MP.

Conversely, suppose that \(P\) solves the MP. Due to Lemma \ref{lem: equi JS} there exists a family of continuous local \((\mathbf{D}^P, P)\)-martingales \(\{X^c(y^*), y^* \in D(A^*)\}\) such that \(D(A^*) \ni y^* \mapsto X^c(y^*)\) is linear, and for each \(y^* \in D(A^*)\) the process \(X^c(y^*)\) is the continuous local \((\mathbf{D}^P, P)\)-martingale part of \(\la X, y^*\ra\). In particular, for all \(y^* \in D(A^*)\)
\[
\lle X^c(y^*)\rre = \int_0^\cdot \la \sigma (X_{s-}) U \sigma^*(X_{s-}) y^*, y^*\ra \dd s.
\]
Moreover, by Lemma~\ref{lem: comp}, the \(\mathbf{D}^P\)-predictable \(P\)-compensator of \(\mu^X\) is given by 
\[
\nu^X(\dd t, G) = \int \1_G (\delta (X_{t-}, x)) F(\dd x) \dd t,\quad G \in \mathscr{B}(\mathbb{B}), 0 \not \in G.
\]

Recalling Lemma \ref{lem: equi JS} and the canonical decomposition of a semimartingale, see \cite[Theorem II.2.34]{JS}, we deduce from the representation theorem for cylindrical local martingales as given by \cite[Theorem 2]{Ondrejat2005}
and the representation theorem for integer-valued random measures given by \cite[Theorem 14.56]{J79} that we find an extension of \((\mathbb{D}, \mathscr{D}^P, \mathbf{D}^P, P)\) which supports a cylindrical Brownian motion \(W\) with covariance \(U\), a Poisson random measure \(\mu\) with intensity measure \(p\) and a solution process to the SDE driven by \((W, \mu)\), which law coincides with \(P\).
The theorem is proven.
\qed
\subsection{Proof of Lemma \ref{lem: linear well-posed}}\label{sec: pf lemma well-posed}
Let \(W\) be a cylindrical Brownian motion with covariance operator \(a\) and \(\mu\) be a Poisson random measure with intensity measure \(p\).
Then, the process 
\begin{align}\label{eq: Levy process}
L_t \triangleq bt + W_t + h(x) \star (\mu - p)_t + (x - h(x)) \star \mu_t, \quad t \in [0, \infty),
\end{align}
is an \(\mathbb{H}\)-valued L\'evy process, see \cite{mandrekar2014stochastic, peszat2007stochastic}. 
Here, the integrals in \eqref{eq: Levy process} can be defined in Pettis' sense, see \cite[Chapter 3]{mandrekar2014stochastic}. In particular, for all \(y \in \mathbb{H}\) we have 
\begin{equation} \label{eq: pettis}
\begin{split}
\la x \star (\mu - p), y\ra &= \la x, y\ra \star (\mu - p), \\
\la (x - h(x)) \star \mu, y\ra &= \la x - h(x), y \ra \star \mu.
\end{split}
\end{equation}
The law of \(L\) can be seen as a probability measure on \((\mathbb{D},\mathscr{D})\). We denote it by \(Q\). 
In view of the L\'evy-Khinchine formula for \(\mathbb{H}\)-valued L\'evy processes, see, for instance, 
\cite[Theorem 4.27]{peszat2007stochastic}, and a classical lemma of Gnedenko and Kolmogorov, see \cite[Lemma II.2.44]{JS}, \(Q\) is uniquely characterized by \(b, a\) and \(p\).

Let \(l\) be as in the statement of Corollary \ref{coro: appl uni}, i.e. the identity on \(\mathbb{H}\).
For each solution process \(Y\) to the SDE \((a, p, A, b, l, \beta, \varepsilon_x)\) with driving noise \(W\) and \(\mu\) we can write 
\[
\la Y, y\ra = \la x, y\ra + \int_0^\cdot \la Y_{s-}, A^* y\ra\dd s + \la L, y\ra,\quad y \in D(A^*),
\]
i.e. in short notation
\begin{align}\label{eq: short SDE}
\dd Y_t =A Y_{t-}\dd t + \dd L_t,\quad Y_0 = x.
\end{align}
Here, \(\beta\) denotes the mapping \(\beta (x, y) = y\) for \(x, y \in \mathbb{H}\). 
By \cite[Proposition 8.7, Corollary 10.9]{doi:10.1080/17442501003624407}, the SDE \eqref{eq: short SDE} has a solution for all \(x \in \mathbb{H}\) and
all solution processes on the same probability space w.r.t. the same driving noise \(L\) are indistinguishable.

Next, we use a Yamada-Watanabe-type idea to prove that the laws of the solution processes coincide. 
Let \(Y^1\) and \(Y^2\) be two solution processes to \eqref{eq: short SDE} which are possibly defined on different probability spaces.
Take \(k = 1, 2\). The driving noise of \(Y^k\) is denoted by \(L^k\).
We denote by \(P^k\) the law of \((Y^k, L^k)\), considered as a probability measure on \((\mathbb{D} \times \mathbb{D}, \mathscr{D} \otimes \mathscr{D})\).
It is well-known, see, for instance, \cite[Theorem 5.3.19]{KaraShre}, that there exists a regular conditional probability \(Q^k\) (given the second coordinate) such that
\(
P^k (\dd \omega, \dd \alpha) = Q^k(\alpha, \dd \omega) Q(\dd \alpha).
\)
We set
\begin{align*}
\Omega^\star &\triangleq \mathbb{D} \times \mathbb{D} \times \mathbb{D},\\ \mathscr{F}^\star &\triangleq \mathscr{D} \otimes \mathscr{D} \otimes \mathscr{D},\\ Q^\star(\dd \omega^1, \dd \omega^2, \dd \alpha) &\triangleq Q^1(\alpha, \dd \omega^1) Q^2(\alpha, \dd \omega^2) Q(\dd \alpha).
\end{align*}
Denote the coordinate process on \(\Omega^\star\) by \(X = (X^1, X^2, X^3)\).
Now, it is easy to see that 
\(
Q^\star \circ (X^k, X^3)^{-1} = P^k\) and \(Q^\star \circ (X^3)^{-1} = Q.
\)
In particular, \(Q^\star \circ (X^k_0)^{-1} = \varepsilon_x\), and, for all \(y \in D(A^*)\), \(Q^\star\)-a.s.
\begin{align*}
\la X^k, y\ra = \la x, y\ra &+ \int_0^\cdot \la X^k_{s-}, A^* y\ra \dd s + \la X^3, y\ra.
\end{align*}
We denote 
\(
\mathscr{F}^\star_t \triangleq \sigma (X^1_s, X^2_s, X^3_s, s \in [0, t])\) and \(\F^\star \triangleq (\mathscr{F}^\star_t)_{t \in [0, \infty)}.\)
\begin{lemma}
	For all \(t \in [0, \infty)\) and \(s < t\) the random variable \(X^3_t - X^3_s\) is \(Q^\star\)-independent of \(\mathscr{F}^\star_s\).
\end{lemma}
\begin{proof}
	We adapt the proof of \cite[Theorem IX.1.7]{RY}.
	Let us first recall the following fact, see \cite[Theorem 15.5]{bauer2002wahrscheinlichkeitstheorie}:
	If \(\mathscr{G}^i\), \(i = 1, 2, 3\), are three \(\sigma\)-fields and \(P\) is a probability measure such that \(\mathscr{G}^1 \vee \mathscr{G}^2\) is \(P\)-independent of \(\mathscr{G}^3\), then for all \(A \in \mathscr{G}^1\) it holds \(P\)-a.s.
	\(
	P(A|\mathscr{G}^2) = P(A |\mathscr{G}^2 \vee \mathscr{G}^3).
	\)
	
	Denote by \(X = (X^1, X^2)\) the coordinate process on \((\mathbb{D} \times \mathbb{D}, \mathscr{D}\otimes\mathscr{D})\).
	Applying the fact above with \(\mathscr{G}^1 \triangleq \sigma (X^1_r, r \in [0, s]), \mathscr{G}^2 \triangleq \sigma(X^2_r, r \in [0, s]), \mathscr{G}^3 \triangleq \sigma(X^2_r - X_s^2, r \in [s, \infty))\) and \(P \triangleq P^1\), we obtain for all \(A \in \mathscr{G}^1\) that \(P^1\)-a.s.
	\(
	Q^1(\cdot, A) = P^1(A|\mathscr{G}^2 \vee \mathscr{G}^3) = P^1(A|\mathscr{G}^2). 
	\)
	Hence, \(Q^1(\cdot, A)\) is \(\mathscr{G}^2\)-measurable up to a \(Q\)-null set. Similarly, we obtain that for all \(A \in \mathscr{G}^1\) the mapping \(Q^2(\cdot, A)\) is \(\mathscr{G}^2\)-measurable up to a \(Q\)-null set.
	
	Now, we consider again \((\Omega^\star, \mathscr{F}^\star, \F^\star, Q^\star)\). 
	Let \(A_1 \in \sigma(X^1_r, r \in [0, s]),\) \(A_2 \in \sigma(X^2_r, r \in [0, s])\) and \(B \in \sigma(X^3_r, r \in [0, s])\). Then, we obtain for all \(y \in \mathbb{H}\)
	\begin{align*}
	E^{Q^\star} &\left[ \exp \left(\sqrt{-1} \la X^3_t - X^3_s, y\ra \right) \1_{A^1} \1_{A^2} \1_{B} \right] 
	\\&= \int_B \exp \left(\sqrt{-1} \la \omega(t) - \omega(s), y \ra\right) Q^1(\omega, A^1) Q^2(\omega, A^2) Q(\dd \omega)
	\\&= \int \exp \left(\sqrt{-1} \la \omega(t) - \omega(s), y \ra \right)Q(\dd \omega) Q^\star(A^1 \times A^2 \times B)
	\\&= E^{Q^\star} \left[ \exp \left(\sqrt{-1} \la X^3_t - X^3_s, y \ra \right)\right] Q^\star(A^1 \times A^2 \times B).
	\end{align*}
	By \cite[Lemma 2.6]{Kallenberg}, this implies our claim.
\end{proof}

Therefore, \(X^3\) is a L\'evy process on \((\Omega^\star, \mathscr{F}^\star, \F^\star, Q^\star)\) (and also on its augmentation, see the proof of \cite[Theorem II.68.2]{RW1}) with law \(Q\), and \(X^1\) and \(X^2\) solve ≥\eqref{eq: short SDE} w.r.t. the same driving noise \(X^3\). 
We conclude that \(Q^\star\)-a.s. \(X^1 = X^2\). 
Hence, \(P^1 = P^2\) follows readily. The claim of the lemma is due to Theorem \ref{theo: GMP SDE}.
\qed
\bibliographystyle{agsm}
\bibliography{References}

\end{document}